\documentclass[oneside,english]{amsart}
\usepackage[T1]{fontenc}
\usepackage[utf8]{inputenc}
\usepackage{amsthm}
\usepackage{amssymb}
\usepackage{esint}
\usepackage[all]{xy}
\usepackage{lmodern}
\usepackage{color}
\usepackage{hyperref}
\hypersetup{
    colorlinks=true, 
    linktoc=all,     
    linkcolor=blue,  
}

\makeatletter

\numberwithin{equation}{section}
\numberwithin{figure}{section}
\theoremstyle{plain}
\newtheorem{thm}{\protect\theoremname}[section]
  \theoremstyle{plain}
  \newtheorem{cor}[thm]{\protect\corollaryname}
  \newtheorem{prop}[thm]{\protect\propositionname}
  \theoremstyle{definition}
  \newtheorem{defn}[thm]{\protect\definitionname}
  \theoremstyle{remark}
  \newtheorem{rem}[thm]{\protect\remarkname}
  \theoremstyle{plain}
  \newtheorem{lem}[thm]{\protect\lemmaname}
  \theoremstyle{definition}
  \newtheorem{example}[thm]{\protect\examplename}
  \theoremstyle{plain}
  \newtheorem{conj}[thm]{\protect\conjecturename}
  
  \theoremstyle{plain}

\makeatother

\usepackage{babel}
  \providecommand{\definitionname}{Definition}
  \providecommand{\examplename}{Example}
  \providecommand{\lemmaname}{Lemma}
  \providecommand{\propositionname}{Proposition}
  \providecommand{\remarkname}{Remark}
  \providecommand{\corollaryname}{Corollary}
\providecommand{\theoremname}{Theorem}
\providecommand{\conjecturename}{Conjecture}

\newcommand*{\triple}[2][.1ex]{%
  \mathrel{\vcenter{\offinterlineskip%
  \hbox{$#2$}\vskip#1\hbox{$#2$}\vskip#1\hbox{$#2$}}}}
\newcommand*{\triplerightarrow}{\triple{\rightarrow}}

\begin{document}

\title{Derived deformation theory of algebraic structures}

\author{Grégory Ginot, Sinan Yalin}

\begin{abstract}
The main purpose of this article is to develop an explicit derived deformation theory of algebraic structures at a high level of generality, encompassing in a common framework various kinds of algebras (associative, commutative, Poisson...) or bialgebras (associative and coassociative, Lie, Frobenius...), that is algebraic structures parametrized by props. 

A central aspect is that we define and study moduli spaces of deformations of algebraic structures \emph{up to quasi-isomorphisms} (and not only up to isomorphims or $\infty$-isotopies).
To do so, we implement methods coming from derived algebraic geometry, by encapsulating these deformation theories as classifying (pre)stacks with good infinitesimal properties and 
derived formal groups. In particular, we prove that the Lie algebra describing the deformation theory of an object in a given $\infty$-category of dg algebras can be obtained equivalently as the tangent complex of loops on a derived quotient of this moduli space by the homotopy automorphims of this object.

Moreover, we provide explicit formulae for such derived deformation problems of algebraic structures up to quasi-isomorphisms and relate them in a precise way to other standard deformation problems of algebraic structures. This relation is given by a fiber sequence of the associated dg-Lie algebras of their deformation complexes. 
 Our results  provide simultaneously a vast generalization of standard deformation theory of algebraic structures which is suitable (and needed) to set up algebraic deformation theory both at the $\infty$-categorical level and at a higher level of generality than algebras over operads.

 In addition, we study a general criterion to compare formal moduli problems of different algebraic structures and apply our formalism to $E_n$-algebras and bialgebras. 

\end{abstract}
\maketitle

\pagebreak

\tableofcontents{}

\pagebreak

\section*{Introduction}

Deformations of algebraic structures of various kind, both classical and homotopical, have played a central role in mathematical physics and algebraic topology since the pioneering work of Drinfeld~\cite{Dri, Dri2} in the 80s as well as the work of Kontsevich~\cite{Ko1,Ko2} or Chas-Sullivan~\cite{CS} in the late 90s.  For instance, in classical deformation quantization, a star-product is a deformation of the commutative algebra of functions to an associative algebra while a quantum group is a deformation of the cocommutative bialgebra structure of a universal envelopping algebra.

In most applications,  one consider deformations of algebraic structures up to some equivalence relations, usually called gauge equivalences. In particular, different gauge equivalences on the same algebraic structure lead to \emph{different} deformation theories. This data is organized into a moduli \emph{space} of deformations whose connected components are the gauge equivalence classes of the deformed structure. Their higher homotopy groups encode (higher) symmetries which are becoming increasingly important in modern applications. By the Deligne philosophy, now a deep theorem by Lurie~\cite{Lur02} and Pridham~\cite{Pri} 
such a moduli space is equivalent to the data of a homotopy Lie algebra.

The emergence of 
derived/higher structures techniques allows not only to consider general moduli spaces of deformations (derived formal moduli problems), but also to consider deformations of algebraic structures more general than those given by Quillen model categories of algebras over operads. In particular, it allows to consider  bialgebraic structures, that is algebras over \emph{props}, in \emph{high generality}. 

The main goal of this paper is to exploit these techniques to  prove several new results about deformation theory of algebraic structures. In particular, we seek to   provide  appropriate extension of classical algebraic deformation theory simultaneously in two directions:
\begin{itemize}
\item[(1)] By considering very general kinds of algebraic structures parametrized by props, which are of crucial importance in various problems of topology, geometry and mathematical physics where such structures appear;

\item[(2)] By considering derived formal moduli problems controlling the deformation theory of algebras \emph{in the $\infty$-category of algebras}, that is \emph{up to quasi-isomorphism}, contrary to the setting of standard operadic deformation theory which considers deformations \emph{up to $\infty$-isotopies} (see \S~\ref{SS:MainresultsIntro} below and section~\ref{S:Examples},\ref{SS:DefOperad}, \ref{SS:DefOperad2} as well for detailed comparison and examples)\footnote{for instance, algebraic structures up to quasi-isomorphisms  form precisely Kontsevich setting encompassing deformation of functions   into star-products in the analytic or algebraic geometry context as well as for smooth manifolds, where it boils down to \rq\rq{}up to isomorphism\lq\lq{} 
}.
\end{itemize}
Both directions require to work out new methods:
\begin{itemize}
\item[(1)] by getting rid of the standard use of Quillen model structures to describe model categories of algebras, which does not make sense anymore for algebras over props: one \emph{has to work directly at an $\infty$-categorical level}.

\item[(2)] by replacing the classical gauge group action and classical deformation functors by appropriate derived moduli spaces of algebraic structures and \emph{derived formal groups} of homotopy automorphisms.
\end{itemize}

We now explain in more details the motivations and historical setting for our work in \S~\ref{SS:IntroMotiv} and then our contributions and main results in \S~\ref{SS:MainresultsIntro}.

\subsection{Motivations}\label{SS:IntroMotiv}
As already mentioned,  many algebraic structures of various types play a key role in algebra, topology, geometry and mathematical physics. This is the case of  associative algebras, commutative algebras, Lie algebras, and Poisson algebras to name a few. All these kinds of algebras share a common feature, being defined by operations with several inputs and one single output (the associative product, the Lie bracket, the Poisson bracket).
The notion of operad is a unifying approach to encompass all these structures in a single formalism, and has proven to be a very powerful tool to study these structures, both from a combinatorial perspective and in a topological or dg-context\footnote{where the strict algebraic structure are no longer invariant under the natural equivalence of the underlying object and need to be replaced by their homotopy enhancement}. The first historical examples, of topological nature, are the operads of little $n$-disks discovered in the study of iterated loop spaces in the sixties. 
Algebras governed by (versions of) these operads, as well as their dg-cousin formed by (shifted) Poisson algebras and their deformation theory play a prominent role in a variety of topics such as the study 
of iterated loop spaces, Goodwillie-Weiss calculus for embedding spaces, deformation quantization of Poisson manifolds and Lie bialgebras, factorization homology and derived symplectic/Poisson geometry~\cite{Ko1, Ko2, Lur0, Lur2, CPTVV, FG, Fra, Fre5, GTZ, Hin, Kap-TFT, KoSo, May, Preygel, Tam1, Toen-ICM}. 

However, algebraic structure governed by operations with several inputs \emph{and several outputs} also appear naturally in a variety of topics related to the same fields of mathematics. Standard example are associative and coassociative bialgebras and Lie bialgebras, which 
 are central in various topics of algebraic topology, representation theory and mathematical physics~\cite{Dri, Dri2, Baues, EK1, EK2, GS, Mer1, Mer2}. 
 Here the formalism of props,  which actually goes back to \cite{MLa}, is the convenient unifying framework to handle such structures. Props plays a crucial role in the deformation quantization process for Lie bialgebras, as shown by Etingof-Kazdhan (\cite{EK1}, \cite{EK2}), and more generally in the theory of quantization functors \cite{EE, MW}.  Props also appear naturally in topology, for example the Frobenius bialgebra structure on the cohomology of compact oriented manifolds coming from Poincaré duality, and the involutive Lie bialgebra structure on the equivariant homology of loop spaces on manifolds, which lies at the heart of string topology (\cite{CS1},\cite{CS2}) and are also central  in symplectic field theory and Lagrangian Floer theory by the work of Cielebak-Fukaya-Latsheev \cite{CFL}. Props also provide a concise way to encode various field theories such as topological quantum field theories and conformal field theories, and have recently proven to be the kind of algebraic structure underlying the topological recursion phenomenom, as unraveled by Kontsevich and Soibelman in \cite{KS} (see \cite{Bor} for  connections with mathematical physics and algebraic geometry).

A meaningful idea to understand the behavior of these various structures and, accordingly, to get more information about the mathematical objects on which they act, is to organize all the possible deformations of a given structure into a single \emph{geometric object} which encapsulates not only the deformations but also an equivalence relation between these deformations. That is, to define a \emph{formal moduli problem}. Such ideas goes back to the pioneering work of Kodaira-Spencer in geometry and the work of Gerstenhaber on associative algebras and 
Hochschild cohomology. In the eighties supported by Deligne and Drinfeld, 
a groundbreaking principle emerged, asserting that any formal moduli problem corresponds to a certain differential graded Lie algebra which parametrizes algebraically the corresponding deformation theory. The deformations correspond to special elements of this Lie algebra called the Maurer-Cartan elements, and equivalences of deformations are determined by a quotient under the action of a gauge group.
This principle had major applications among which one can pick deformation theory of complex manifolds, representation spaces of fundamental groups of projective varieties in Goldman-Millson's theory, and Kontsevich deformation quantization of Poisson manifolds.

The theory of ``classical'' or ``underived'' formal moduli problems was not sufficient to made this principle completely precise and had several limitations (like impossibility to consider weak equivalences or getting non equivalent dg-Lie algebras describing the same moduli problem). 


These difficulties were solved by considering higher structured geometric moduli problem 
using \emph{$\infty$-category theory} and \emph{derived algebraic geometry}.
The appropriate formalism is then the theory of \emph{derived formal moduli problems}, which are simplicial presheaves over the opposite category of augmented artinian cdgas satisfying some extra properties with respect to homotopy pullbacks (a derived version of the Schlessinger condition). Precisely,  Lurie and Pridham  \cite{Lur0,Pri} proved that (derived) formal moduli problems and dg Lie algebras are equivalent as $\infty$-categories. 
In fact, a given formal moduli problem controlling the infinitesimal neighbourhood  of a point on a moduli space corresponds to a dg Lie algebra called the deformation complex of this point.

In this paper, we use rather systematically these ideas of derived formal moduli problems and derived techniques to study   deformation theory of \emph{algebraic structures}.   In particular, we give a conceptual explanation of the differences between various deformation complexes appearing in the literature by explaining which kind of derived moduli problem each of these complexes controls. 
A key part of our study is that we study algebras over \emph{very general props} and that we consider moduli spaces of deformations of algebraic structures up to \emph{quasi-isomorphisms}.

\subsection{Main results}\label{SS:MainresultsIntro}


We study moduli spaces of algebraic structures and formal moduli problems controlling their deformations. In the differential graded setting, algebraic structures are deformed as algebraic structures \emph{up to homotopy}.  A convenient formalism to deal with such at a high level of generality, encompassing not only algebras but also bialgebras, is the notion of (dg-)properad \cite{Val}. 
Briefly, to any complex $X$, on can associate its endomorphism properad $End_X(m,n)=Hom(X^{\otimes m},X^{\otimes n})$. Then,
 given a properad $P$,  a $P$-algebra structure on $X$ is given by a properad morphism
\[
P\rightarrow End_X.
\]

There are several possible natural notions for defining deformations of (possibly homotopy) algebraic structures and we consider and compare several of them.

A standard \emph{(pr)operadic approach} to define a deformation complex of those structure is as follows. Given a properad $P$, the notion of 
homotopy $P$-algebra (or $P$-algebra up to homotopy) can be defined properly by considering cofibrant resolutions of properads. That is, by considering $P_{\infty}$-algebra where $P_{\infty}$ is a \emph{cofibrant resolution} of $P$ in the model category of properads. To any  
$P_{\infty}$-algebra structure $\varphi:P_{\infty}\rightarrow End_X$ on a complex $X$, there is a formal moduli problem $\underline{P_{\infty}\{X\}}^{\varphi}$ controlling the deformation theory of the properad morphism $\varphi$.  The associated deformation complex is an explicit dg Lie algebra noted $g_{P,X}^{\varphi}$. This is a rather standard approach. Let us note that, even though we use properads in the explanation above, there is a well-defined model structure on the category of props such that homotopy $P$-algebras make sense for $P$ a general prop. Homotopy $P$-algebras do form an $\infty$-category which is homotopy invariant under choices of resolution (see \ref{Thm:Yalbis}), and the formal moduli problem $\underline{P_{\infty}\{X\}}^{\varphi}$ is well defined for $P$ a prop as well. The restriction to properads comes in when one wants to use explicit models for the dg Lie algebra $g_{P,X}^{\varphi}$ (and its variants introduced throughout the text). We explain more precisely in Sections $1$ and $2$ where the restriction to properads is needed.

\smallskip

However, we can also construct a derived formal moduli problem controling the deformation theory of a $P_{\infty}$-algebra $A$ directly \emph{in the} $\infty$\emph{-category} $P_{\infty}-Alg$ of $P_\infty$-algebras (with quasi-isomorphisms as weak equivalences). This is not the same as deforming the morphism $\varphi$ (in a way precised below, the Maurer-Cartan elements are the same in both cases but \emph{the gauge equivalence relation differs}).

To set up the appropriate framework for such a deformation theory, we introduce  in Section~\ref{S:formalgeom} the notion of derived prestack group, which can be thought as a family of homotopy formal groups parametrized by a base space  and apply  this formalism to the deformation theory of algebras over properads. Briefly, one associates to $A$ its \emph{derived prestack group of homotopy automorphisms} which is  the $\infty$-functor
\begin{eqnarray*}
G_P(A):cdga_{\mathbb{K}} & \rightarrow & E_1\mathrm{-Alg}^{gp}(\mathrm{Spaces}) \\
R & \longmapsto & haut_{P_{\infty}-Alg(Mod_R)}(A\otimes R)
\end{eqnarray*}
where $haut_{P_{\infty}-Alg(Mod_R)}(A\otimes R)$ is the $\infty$-group of self equivalences of $A\otimes R$ in the $\infty$-category of $R$-linear $P_{\infty}$-algebras.
Taking homotopy fibers over augmented Artinian cdgas, we obtain a derived formal group (see Section~\ref{SS:DerivedPreStackGp}):
\[
\widehat{G_P(A)_{id}}(R)=hofib(G_P(A)(R)\rightarrow G_P(A)(\mathbb{K}))
\]
whose values at an augmented Artinian cdga $R$ is the space of $R$-deformations of $A$. Precisely, we prove
\begin{thm}[See Theorem~\ref{T:hautderivedgroup}]
The simplicial presheaf $G_P(A)$ defines 
a grouplike $E_1$-monoid object in the $\infty$-category of infinitesimally cohesive simplicial $\infty$-presheaves. 
In particular $\widehat{G_P(A)_{id}}$ is a derived formal group.
\end{thm}
By the 
\emph{equivalence between derived formal groups and derived formal moduli problems}, these deformations are parametrized by a dg Lie algebra $Lie(\widehat{G_P(A)_{id}})$.

Two natural questions arise from these constructions. 
\begin{itemize} \item First, can we relate the classical deformation theory of the morphism $\varphi:P_{\infty}\rightarrow End_X$, controled by $g_{P,X}^{\varphi}$, to the deformation theory of $(X,\varphi)$ in $P_{\infty}-Alg$, controled by $Lie(\widehat{G_P(X,\varphi)})$ ? \item Second, is there an explicit formula computing $Lie(\widehat{G_P(X,\varphi)})$ for general $P$ and $(X,\varphi)$ ?
\end{itemize}
The answer to the first question is the following natural homotopy fiber sequence relating these two deformation complexes :
\begin{thm}[See Theorem~\ref{P: hofibgroups}]
There is a  fiber sequence of $L_{\infty}$-algebras
\[
g_{P,X}^{\varphi}\longrightarrow Lie(\widehat{G_P(X,\varphi)})\longrightarrow Lie(\underline{haut}(X))
\]
where $Lie(\underline{haut}(X))$ is the Lie algebra of homotopy automorphisms of $X$ as a complex.
\end{thm}
To illustrate concretely how this fiber sequence explains the difference between $g_{P,X}^{\varphi}$ and $Lie(\widehat{G_P(X,\varphi)})$, let us start with the following observation.
One should note that the deformation complex 
$g_{P,X}^{\varphi}$ does \emph{not} give exactly the usual cohomology theories of algebras. 
As a motivating example, let us consider the case of the Hochschild cochain complex of a dg associative algebra $A$ which can be written as $\mathrm{Hom}( A^{\otimes \bullet}, A)$.  
This Hochschild complex is bigraded, with a cohomological grading induced by the grading of $A$  and a weight grading given by the tensor powers $A^{\otimes \bullet}$. 
It turns out that the classical deformation complex $g_{Ass,X}^{\varphi}$, where $X$ is the underlying complex of the algebra $A$ and $\varphi:Ass\rightarrow End_X$ its associative algebra structure, is $\mathrm{Hom}( A^{\otimes >1}, A)$ and in particular misses the summand $\mathrm{Hom}(A,A)$ of weight $1$; which is precisely the one allowing to consider algebras up to (quasi-)isomorphisms. 

The Lie algebra $g_{P,X}^{\varphi}$ can be described very explicitly in terms of a convolution algebra associated to the properad $P_{\infty}$ (Proposition~\ref{P:TangentLinfty}). In section~\ref{S:Plus}, we provide a similar properadic description of the Lie algebra of the formal moduli of homotopy automophisms $\widehat{G_P(X,\varphi)}$. To do so, we use  the ``plus'' construction $g_{P^+,X}^{\varphi^+}$, which is a functorial construction modifying any dg properad to get a new properad encoding both a $P$-algebra structure on $X$ and a compatible differential.
This gives us an explicit model of the deformation complex of $(X,\varphi)$ in the $\infty$-category of $P_{\infty}$-algebras \emph{up to quasi-isomorphisms} and thus answers the second question:
\begin{thm}[See~Theorems \ref{T:Def+=hAut} and \ref{T:twistedsemidirprod}]\label{T:MainThmIntro2}
There are equivalences of $L_\infty$-algebras
\[
Lie(\widehat{G_P(X,\varphi)}) \simeq  g_{P,X}^{\varphi}\rtimes_{\xi} End(X) \simeq g_{P^+,X}^{\varphi^+}.
\]
\end{thm}
The middle term of this equivalence exhibits  $Lie(\widehat{G_P(X,\varphi)})$ as a twisted semi-direct product of $g_{P,X}^{\varphi}$ with the Lie algebra $End(X)$ of endomorphisms of $X$ (equipped with the commutator of the composition product as Lie bracket). The twist $\xi$ is a degree $-1$ map $End(X)[-1]\rightarrow g_{P,X}^{\varphi}$ added to the differential of $g_{P,X}^{\varphi}\rtimes End(X)$ (see Definition \ref{D:semidirectMCtwist}, Lemmas \ref{L:LiesemidirectMCtwist} and \ref{L:semidirectMCtwist} for the construction of the twist and its properties). Let us note that twisted semi-direct products appeared originally in \cite{Tan} and were used further in \cite{Ber} to construct rational Lie models of fibrations and automorphisms of fiber bundles. To prove this Theorem, we reinterpret the deformation complex $Lie(\underline{haut}_{P_{\infty}-Alg}(X,\varphi))$ as the tangent Lie algebra of a homotopy quotient of $\underline{P_{\infty}\{X\}}$ by the $\infty$-action of $\underline{haut}(X)$ in Section ~\ref{S:TgtLie} and relate it to the action of $(-)^+$ on the properadic side in Section \ref{S:Plus}.

\smallskip

To summarize, the conceptual explanation behind this phenomenon is as follows. On the one hand, the $L_{\infty}$-algebra $g_{P,X}^{\varphi}$ controls the deformations of the $P_{\infty}$-algebra structure over a \emph{fixed complex} $X$, that is, the deformation theory of the properad morphism $\varphi$. On the other hand, we built a \emph{derived formal group} $\widehat{\underline{haut}_{P_{\infty}}(X,\varphi)}_{id}$ whose corresponding $L_{\infty}$-algebra $Lie(\widehat{\underline{haut}_{P_{\infty}}(X,\varphi)}_{id})$ describes \emph{another derived deformation problem}: an $R$-deformation of a $P$-algebra $A$ in the $\infty$-category of $P_{\infty}$-algebras \emph{up to quasi-isomorphisms} is a an $R$-linear $P_{\infty}$-algebra $\tilde{A}\simeq A\otimes R$ with a $\mathbb{K}$-linear $P_{\infty}$-algebra quasi-isomorphism $\tilde{A}\otimes_R\mathbb{K}\stackrel{\sim}{\rightarrow}A$. 
The later $L_{\infty}$-algebra admits two equivalent descriptions
\[
Lie(\widehat{\underline{haut}_{P_{\infty}}(A)}_{id})\simeq g_{P,X}^{\varphi}\ltimes_{hol}End(X)\simeq g_{P^+,X}^{\varphi^+}
\]
where the middle one exhibits this moduli problem as originating from the homotopy quotient of the space of $P_{\infty}$-algebra structures on $X$ by the homotopy action of self-quasi-isomorphisms $haut(X)$, that is, deformations of the $P_{\infty}$-algebra structure \emph{up to self quasi-isomorphisms} of $X$, and the right one encodes this as simultaneous compatible deformations of the $P_{\infty}$-algebra structure \emph{and} of the differential of $X$. We will go back to this in full details in Sections~\ref{S:TgtLie} and \ref{S:Concluding}.

For algebras over cofibrant operads, there is a further characterization of this $L_{\infty}$-algebra as a twisted semi-direct product of the twisted $L_{\infty}$-algebra of coderivations with $End(X)$ (see Corollary \ref{T:identificationBhautQFreeresol}):
\begin{cor}
Let $P_\infty=(\mathcal{F}(s^{-1}\overline{C}),\partial)\stackrel{\sim}{\rightarrow}P$ be 
a cofibrant quasi-free resolution of an operad $P$ where $C$ is a cooperad and $(X,\varphi)$ be a $P_\infty$-algebra.  
One has an equivalence of $L_{\infty}$-algebras
\[
Lie(\underline{haut}_{P_\infty}(X,\varphi))\cong Coder( \overline{C}(X[1]))^{D_{\varphi}} \rtimes_{\xi} End(X,X)
\]
where the last term is the dg Lie algebra of coderivations of the cofree coalgebra on $X[1]$ 
twisted by the Maurer Cartan element $D_{\varphi}$ (the coderivation of square zero corresponding to the $P_{\infty}$-algebra structure $\varphi$) and the action of $End(X,X)$ is given by the composition of coderivations of $C(X[1])$. The twist $\xi$ is defined by Lemma \ref{L:LiesemidirectMCtwist}.
\end{cor}

Returning to the Hochschild complex example (which we also explain in Example \ref{ex:strictassocalgebra}), we now see  the role of the weight $1$ part $\mathrm{Hom}(A,A)$. Indeed, in the case of a an associative dg algebra $A$, the complex $g_{Ass^+,X}^{\varphi^+}\cong Hom ( A^{\otimes >0}, A) [1]$  computes the reduced Hochschild cohomology of $A$, where the right hand side is a sub-complex of the standard Hochschild cochain complex shifted down by $1$ equipped with its standard Lie algebra structure.
The complex $g_{Ass,X}^{\varphi} \cong \mathrm{Hom} ( A^{\otimes >1}, A)[1]$ is the one controlling the formal moduli problem of deformations of $A$ with fixed differential\footnote{Thus, when $A$ is an ordinary,  non dg, vector space, the complex $g_{Ass,X}^{0}$ parametrizes  the moduli space of associative algebra structures on $A$, while 
$g_{Ass^+,X}^{0^+}$  parametrizes  the moduli space of asociative algebra structures up to isomorphism of algebras}, where the right hand side is the subcomplex of the previous shifted Hochschild cochain complex 
where we have removed the $Hom(A,A)$ component\footnote{there is also a third complex, the full shifted Hochschild complex 
$\mathrm{Hom} ( A^{\otimes\geq 0}, A)[1)$, which controls not the deformations of $A$ itself but the linear deformations of its dg category 
of modules $\mathrm{Mod}_A$~\cite{KellerLowen, Preygel}}.

\smallskip

In addition, in Section~\ref{S:formalgeom}, we prove a general criterion to compare formal moduli problems induced by algebras :
\begin{thm}[see Theorems \ref{T:equivhautLie} and \ref{T:equivfiberseq}]
\begin{itemize}
Let $F$ be an equivalence of presheaves of $\infty$-categories
\[
F:\underline{P_{\infty}-Alg}\stackrel{\sim}{\longrightarrow}\underline{Q_{\infty}-Alg}.
\]
Then
\item[(1)] $F$ induces an equivalence of derived formal moduli problems
\[
B_{fmp}\widehat{\underline{haut}_{P_{\infty}-Alg}(X,\varphi)}_{Id_{(X,\varphi)}}\stackrel{\sim}{\rightarrow} B_{fmp}\widehat{\underline{haut}_{Q_{\infty}-Alg}(F(X,\varphi))}_{Id_{(X,\varphi)}},
\]
equivalently an equivalence of the associated $L_{\infty}$-algebras
\[
Lie(\underline{haut}_{P_{\infty}-Alg}(X,\varphi))\stackrel{\sim}{\rightarrow}Lie(\underline{haut}_{Q_{\infty}-Alg}(F(X,\varphi))),
\]
where $F(\varphi)$ is the $Q_{\infty}$-algebra structure on the image of  $(X,\varphi)$  under $F$.

\item[(2)] if moreover $F$ commutes with the maps induced by the forgetful functors of these algebras, then it induces an equivalence of fiber sequences of derived formal moduli problems
\[
\xymatrix{
\underline{P_{\infty}\{X\}}^{\varphi}\ar[d]^-{\sim}\ar[r] & B_{fmp}\widehat{\underline{haut}_{P_{\infty}-Alg}(X,\varphi)}_{Id_{(X,\varphi)}}\ar[d]^-{\sim}\ar[r] & B_{fmp}\widehat{\underline{haut}(X)}_{Id_X}\ar[d]^-=\\
\underline{Q_{\infty}\{F(X)\}}^{F(\varphi)}\ar[r]  & B_{fmp}\widehat{\underline{haut}_{Q_{\infty}-Alg}(F(X,\varphi))}_{Id_{(X,\varphi)}}\ar[r] & B_{fmp}\widehat{\underline{haut}(X)}_{Id_X}
}.
\]
\end{itemize}
\end{thm}

In Section~\ref{S:Examples}, we apply our machinery to   derived deformation theory of $n$-shifted Poisson algebras (that is Poisson algebras with a Poisson bracket of degree $1-n$) and $E_n$-algebras:
\begin{thm}[See Corollary~\ref{C:Tamarkin=Defo}]

(1) The \emph{Tamarkin deformation complex\footnote{which we denote $CH_{Pois_n}^{(\bullet>0)}(A)[n]$  since it is the part of positive weight in the full Poisson complex~\cite{CaWi}}}~\cite{Ta-deformationofd-algebra}  \emph{controls deformations
of $A$ in $Pois_{n,\infty}-Alg[W_{qiso}^{-1}]$, that is, in homotopy dg-$Pois_n$-algebras up to quasi-isomorphisms}. It is thus equivalent to the tangent Lie algebra $g_{Pois_n^+,A}^{\varphi^+}$ of $G_{Pois_n}(A)$.

(2) For $n\geq 2$ the Tamarkin deformation complex of $A$ is equivalent, as an $L_{\infty}$-algebra, to the $E_n$-tangent complex of $A$ seen as an $E_n$-algebra via the formality of $E_n$-operads.
\end{thm}
To the best of the authors knowledge, the proof that this complex is indeed a deformation complex in the precise meaning of formal derived moduli problems is new, as well as the concordance with the $L_{\infty}$-structure induced by the higher Deligne conjecture (which provides an $E_{n+1}$-algebra structure on the $E_n$-tangent complex of an $E_n$-algebra).
We also prove that the deformation complex $g_{Pois_n,A}^{\varphi}$ of the formal moduli problem $\underline{{Pois_{n}}_{\infty} \{A \}}^{\varphi}$ of homotopy $n$-Poisson algebra structures deforming $\varphi$ is given by the $L_\infty$-algebra  $CH_{Pois_n}^{(\bullet>1)}(A)[n]$, which is a further truncation of $CH_{Pois_n}(A)[n]$.

Concerning bialgebras, we obtain the first theorem describing precisely why (a suitable\footnote{there are several closely related versions of the Gerstenhaber-Schack, depending on how we truncate them} version of) the Gerstenhaber-Schack complex 
\[C_{GS}^*(B,B)\cong \prod_{m,n\geq 1}Hom_{dg}(B^{\otimes m},B^{\otimes n})[-m-n]\] is the appropriate deformation complex of a dg bialgebra \emph{up to quasi-isomorphisms}  in terms of derived moduli problems:
\begin{thm}[See Theorem~\ref{T:GSIdentific}]
The Gerstenhaber-Schack complex is quasi-isomorphic to the $L_{\infty}$-algebra controlling the deformations of dg bialgebras up to quasi-isomorphisms:
$$C_{GS}^*(B,B) \cong  g_{Bialg_{\infty}^+,B}^{\varphi^+} \simeq Lie(\widehat{\underline{haut}_{Bialg_{\infty}}(B)}_{id}).$$
\end{thm}
Note that our theorem~\ref{T:MainThmIntro2} implies that the $L_\infty$-algebra structure induced on $C_{GS}^*(B,B)$ contains as a sub $L_\infty$-algebra the Merkulov-Vallette deformation complex~\cite{MV2}.

\smallskip 

Finally, in Section~\ref{S:Concluding}, we give an overview and comparison of  various (derived or not) deformation problems of algebraic structures arising in our work and the litterature.

\begin{rem}
A natural candidate for the deformation $\infty$-functor of a $P_{\infty}$-algebra $A$ in the $\infty$-category of $P_{\infty}$-algebras (localized with respect to quasi-isomorphisms) is defined as follows. One associate, to any augmented artinian dg algebra $R$, the simplicial nerve $\mathcal{N}wP_{\infty}-Alg(Mod_R)$ of the subcategory of weak equivalences of $P_{\infty}$-algebras in $R$-modules. The augmentation $R\rightarrow\mathbb{K}$ induces a simplicial map
\[
\mathcal{N}wP_{\infty}-Alg(Mod_R)\rightarrow \mathcal{N}wP_{\infty}-Alg(Ch_{\mathbb{K}}).
\]
The evaluation of the classifying presheaf of deformations of $A$ on an augmented artinian dg algebra $R$ is the homotopy fiber of the map above at the base point $A$. In other words, it is the formal completion $\widehat{\underline{\mathcal{N}wP_{\infty}-Alg}}_A$ of the functor $R\mapsto \mathcal{N}wP_{\infty}-Alg(Mod_R)$ at $A$. We detail its construction and properties in Section \ref{SS:DKSpace}.  In the \emph{operadic} setting, such a functor has been studied by Hinich in \cite{Hin3}.
 
We have
\[
\widehat{\underline{\mathcal{N}wP_{\infty}-Alg}}_A(R)\simeq P_{\infty}-Alg(Mod_R)[W_{qiso}^{-1}]\times^h_{P_{\infty}-Alg(Ch_{\mathbb{K}})[W_{qiso}^{-1}]}\{A\}
\]
It is the space of $R$-linear $P_{\infty}$-algebras $B$ such that $B\otimes_R\mathbb{K}\simeq A$, so it encapsulates the whole deformation theory of $A$ in the $\infty$-category $P_{\infty}-Alg[W_{qiso}^{-1}]$ as we can think of it, that is, an $R$-deformation of $A$ is an $R$-linear $P_{\infty}$-algebra whose restriction modulo $R$ is quasi-isomorphic to $A$, and equivalences between $R$-deformations are defined by compatible $R$-linear quasi-isomorphisms whose restriction modulo $R$ is homotopic to $Id_A$. However, in general, such a simplicial presheaf \emph{\textbf{does not provide a derived formal moduli problem}}. Even in the operadic case, one needs $A$ and $P$ to be non-positively graded to describe it as the nerve of dg Lie algebra (see \cite[Section 4.3]{Hin3}).

The relationship between this classifying presheaf of algebras and the derived formal group of homotopy automorphisms in the neighbourhood of the identity is given by
\[
\widehat{\underline{haut}_{P_{\infty}}(A)_{id}} = \Omega_*\widehat{\underline{\mathcal{N}wP_{\infty}-Alg}}_{A}
\]
where $\Omega_*$ is the loop space for pointed functors as explained in Section~\ref{S:formalgeom}.
By the general formalism explained in Section~\ref{S:formalgeom} we have
\[
\mathbb{T}_{\widehat{\underline{haut}_{P_{\infty}}(A)_{id}}} = Lie(L(\widehat{\underline{\mathcal{N}wP_{\infty}-Alg}}_{A}))
\]
where $L$ is the completion of $\widehat{\underline{\mathcal{N}wP_{\infty}-Alg}}_{A}$ in a derived formal moduli problem.
Another way to state this is that in general \emph{$\widehat{\underline{\mathcal{N}wP_{\infty}-Alg}}_{A}$ is a $1$-proximate moduli functor in the sense of \cite{Lur0}}.
\end{rem}

\subsection{Further applications and perspectives}

A first major application appeared earlier in our preprint \cite{GY}, where some of the results of the present article were announced. Our article provides complete proofs of these results and add some new ones as well. In \cite{GY},
we use them crucially to prove longstanding conjectures in deformation theory of bialgebras and $E_n$-algebras as well as in deformation quantization. We prove a conjecture stated by Gerstenhaber and Schack (in a wrong way) in 1990 \cite{GS}, whose correct version is that the Gerstenhaber-Schack complex forms an $E_3$-algebra, hence unraveling the full algebraic structure of this complex which remained mysterious for a while. It is a ``differential graded bialgebra version'' of the famous Deligne conjecture for associative differential graded algebras (see for instance \cite{Tam1} and \cite{Ko1}). The second one, enunciated by Kontsevich in his celebrated work on deformation quantization of Poisson manifolds \cite{Ko2} in 2000, is the formality, as an $E_3$-algebra, of the deformation complex of the symmetric bialgebra which should imply as a corollary Drinfeld's and Etingof-Kazdhan's deformation quantization of Lie bialgebras (see \cite{Dri}, \cite{EK1} and \cite{EK2}). We solve both conjectures actually at a greater level of generality than the original statements. Moreover, we deduce from it a generalization of Etingof-Kadhan's celebrated deformation quantization in the homotopical and differential graded setting.

The new methods developed here to approach deformation theory and quantization problems have several possible continuations.
In particular, we aim to investigate in future works how our derived algebraic deformation theory could be adapted to provide new deformation theoretic approach, formality statements and  deformation quantization of shifted Poisson structures in derived algebraic geometry. This problem can help understand quantum invariants of various moduli spaces of $G$-bundles over algebraic varieties and topological manifolds, which are naturally shifted Poisson stacks.

Moreover, our framework for derived algebraic deformation theory shall also be useful to study deformation problems related to the various kinds of (bi)algebras structures mentionned in this introduction, occuring in mathematical physics, algebraic topology, string topology, symplectic topology and so on.

\medskip
\noindent \textbf{Acknowledgement.} The authors wish to thank V. Hinich, S. Merkulov, P. Safranov and T. Willwacher for their useful comments. 
They were also partially supported by ANR grants CHroK and CatAG and the first author benefited from the support of Capes-Cofecub project 29443NE and Max Planck Institut fur Mathematik in Bonn as well.

\section*{Notations and conventions}

The reader will find below a list of the main notations used at several places in this article.
\begin{itemize}
\item We work over a field of characteristic zero denoted $\mathbb{K}$.

\item We work with cochain complexes and a cohomological grading and denote
 $Ch_{\mathbb{K}}$  the category of $\mathbb{Z}$-graded \emph{cochain} complexes over $\mathbb{K}$.

\item Let $(\mathcal{C},W_{\mathcal{C}})$ be a relative category, also called a category with weak equivalences. Meaning $\mathcal{C}$ is a category and $W_{\mathcal{C}}$ its subcategory of weak equivalences. The \emph{hammock localization} (see \cite{DK2}) of such a category with respect to its weak equivalences is denoted $L^H(\mathcal{C},W_{\mathcal{C}})$, and the \emph{mapping spaces} of this simplicial localization are noted $L^H(\mathcal{C},W_{\mathcal{C}})(X,Y)$.

\item Given a relative category $(M,W)$, we denote by $M[W^{-1}]$ its $\infty$-categorical localization.
Further, we will write $\mathcal{N}W$ for the coherent nerve of the subcategory of weak-equivalences $W$.

\item We will write ``cdga'' to shorten ``commutative differential graded algebra''.

\item Given a cdga $A$, the category of $A$-modules is noted $Mod_A$. More generally, if $\mathcal{C}$ is a symmetric monoidal category tensored over $Ch_{\mathbb{K}}$, the category of $A$-modules in $\mathcal{C}$ is noted $Mod_A(\mathcal{C})$.

\item Several categories of algebras and coalgebras over $\mathbb{K}$ will have a dedicated notation: 
$cdga_{\mathbb{K}}$ for the category of commutative differential graded algebras, $dgArt_{\mathbb{K}}$ for the category of Artinian cdgas, $dgCog_{\mathbb{K}}$ for the category of dg coassociative coalgebras and $dgLie_{\mathbb{K}}$ for the category of dg Lie algebras.

\item We will us the upper script ``aug'' for categories of augmented algebras.

\item Given a dg Lie algebra $g$, its Chevalley-Eilenberg algebra is noted $C^*_{CE}(g)$ and its Chevalley-Eilenberg coalgebra is noted $C_*^{CE}(g)$.

\item Semi-direct products of dg Lie algebras as well as homotopy semi-direct products of $L_{\infty}$-algebras will be denoted by $g\rtimes h$, where $h$ is acting on $g$. Twisted semi-direct products will be denoted by $\rtimes_{\xi}$.

\item More general categories of algebras and coalgebras over operads or prop(erad)s will have the following generic notations:
given a prop(erad) $P$, we will note $P-Alg$ the category of dg $P$-algebras and given an operad $P$ we will denote by $P-Cog$ the category of dg $P$-coalgebras.

\item Given a prop(erad) $P$, a \emph{cofibrant resolution} of $P$ is noted $P_{\infty}$.

\item The bar and cobar functors for (pr)operads will be respectively denoted by $\mathcal{B}$ and $\mathcal{B}^{co}$ (although the notation $\Omega$ is more often used for the cobar functor in the litterature, it will be already used in this paper for various loop space objects).

\item When the base category is a symmetric monoidal category $\mathcal{C}$ other than $Ch_{\mathbb{K}}$, we denote by $P-Alg(\mathcal{C})$ the category of $P$-algebras in $\mathcal{C}$ and $P-Cog(\mathcal{C})$ the category of $P$-coalgebras in $\mathcal{C}$.

\item Algebras over properads form a relative category for the weak equivalences defined by chain quasi-isomorphisms. The subcategory of \emph{weak equivalences of $P-Alg$} is noted $wP-Alg$.

\item Given a properad $P$ and a complex $X$, we will consider an associated convolution Lie algebra noted $g_{P,X}$ which will give rise to two deformation complexes: the deformation complex $g_{P,X}^{\varphi}$ controling the formal moduli problem of deformations of a $P$-algebra structure $\varphi$ on $X$, and a variant $g_{P^+,X}^{\varphi^+}$ whose role will be explained in Section 3.

\item We will consider various moduli functors in this paper, defined as simplicial presheaves over the opposite category of artinian augmented cdgas: the simplicial presheaf of $P_{\infty}$-algebra structures on $X$ noted $\underline{P_{\infty}\{X\}}$, the formal moduli problem of deformations of a given $P_{\infty}$-algebra structure $\varphi$ on $X$ noted $\underline{P_{\infty}\{X\}}^{\varphi}$, and the derived prestack group of homotopy automorphisms of $(X,\varphi)$ noted $\underline{haut}_{P_{\infty}}(X,\varphi)$.
The derived prestack group of automorphisms of $X$ as a chain complex will be denoted $\underline{haut}(X)$.
More generally, when we extend a simplicial set $K$ into a simplicial presheaf over Artinian cdgas, we will use the underlined notation $\underline{K}$ for that extension.




\end{itemize}

\section{Formal moduli problems and algebraic structures}
Formal moduli problems arise when one wants to study the deformation theory of an object in a category, of a structure on a given object,
of a point in a given moduli space (variety, scheme, stack, derived stack). The general principle of moduli problems is that the deformation theory
of a given point in its formal neighbourhood (that is, the formal completion of the moduli space at this point) is controlled by a certain tangent
dg Lie algebra.



However, if one does not work in a derived geometric/higher categorical context, there are several well known issues with this principle: 
\begin{itemize}
\item[$\bullet$] 
 Lie algebras which are not quasi-isomorphic can nevertheless describe the same moduli problem (a famous example is the deformation theory of a closed subscheme, seen either as a point of a Hilbert scheme or as a point of a Quot scheme). Even worse, their is no systematic recipe to build a Lie algebra out of a moduli problem;

\item[$\bullet$] Deformation problems for which the equivalence relation is given by weak equivalences (say, quasi-isomorphisms between two deformations of a dg algebra) do not fit in the framework of classical algebraic geometry (that is, deformations which manifests a non trivial amount of homotopy theory);

\item[$\bullet$] There is no natural interpretation of the obstruction theory in terms of the corresponding moduli problem.
\end{itemize}
To overcome these difficulties encountered when working in underived deformation theory, in particular in the correspondence between deformation functors and dg Lie algebras, one has to consider moduli problems in a derived setting.
The rigorous statement of an equivalence between (derived) formal moduli problems and dg Lie algebras was proved independently by Lurie in \cite{Lur0}
and by Pridham in \cite{Pri}.
In this paper, what \emph{we will call moduli problems are actually derived moduli problems}.


\subsection{Formal moduli problems and (homotopy) Lie algebras}

We start by an overview of the notion of formal moduli problem and relevant $L_\infty$-algebras.


\begin{defn} \label{D:FMP}Formal moduli problems are functors $F:dgArt_{\mathbb{K}}^{aug}\rightarrow sSet$ 
from augmented Artinian commutative differential graded algebras to simplicial sets satisfying the following conditions:
\begin{itemize}
\item[(1)] The functor $F$ preserves weak equivalences (that is, quasi-isomorphisms of cdgas are sent to weak equivalences of simplicial sets);

\item[(2)] There is a weak equivalence $F(\mathbb{K})\simeq pt$; 

\item[(3)] The functor $F$ is \textbf{infinitesimally cohesive}: 
Given any (homotopy) pullback diagram $\xymatrix{ A' \ar[r] \ar[d] & A \ar[d] \\ 
B' \ar[r] & B}  $ in $dgArt_{\mathbb{K}}$ such that the induced maps $\pi_0(A) \to \pi_0(B)$ and $\pi_0(B') \to \pi_0(B)$ are surjective, the induced 
diagram \[ \xymatrix{ F(A') \ar[r] \ar[d] & F(A) \ar[d] \\ 
F(B') \ar[r] & F(B)} \] is a (homotopy) pullback in $sSet$.
\end{itemize}
\end{defn}

\begin{rem} Condition (3) is 
 a derived version of the classical \emph{Schlessinger condition} introduced in \cite{Lur0} and developped in~\cite{Lur02} 
 The notion of (infinitesimally) cohesive generalizes to any functor from connective dg-commutative algebras to $sSet$. In that general setting functors 
 satisfying condition (3) in Definition~\ref{D:FMP} are called cohesive, while infinitesimally cohesive stands for those functors 
 satisfying this condition only when the 
 the maps   $\pi_0(A) \to \pi_0(B)$ and $\pi_0(B') \to \pi_0(B)$ are further required to  have nilpotent kernels. 
 For Artinian cdgas, the latter condition is automatic and therefore infinitesimally cohesive and cohesive are the same. We stick to the longer name 
 to recall that special property of the Artinian context.
\end{rem}
\begin{rem}
Often throughout the paper, we call our functors from artinian cdgas or cdgas to simplicial sets \emph{simplicial presheaves over the opposite category of...} as a reminiscence of the functor of points approach in algebraic and derived algebraic geometry, where geometric objects represent presheaves on a category of (possibly affine, possibily derived) $\mathbb{K}$-schemes (which in the affine case is equivalent to the opposite category of commutative $\mathbb{K}$-algebras).
\end{rem}
The value $F(\mathbb{K})$ corresponds to the point of which we study the formal neighbourhood, the evaluation $F(\mathbb{K}[t]/(t^2))$ 
on the algebra of dual numbers encodes infinitesimal deformations of this point, and the $F(\mathbb{K}[t]/(t^n))$ 
are polynomial deformations of a higher order, for instance.

\smallskip

Formal moduli problems form a full sub-$\infty$-category noted $FMP_{\mathbb{K}}$ 
of the $\infty$-category of simplicial presheaves over the opposite category of augmented Artinian cdgas. By \cite[Theorem 2.0.2]{Lur0}, 
this $\infty$-category is equivalent to the $\infty$-category $dgLie_{\mathbb{K}}$ of dg Lie algebras. 
Moreover, one side of the equivalence is made explicit, and is equivalent to the nerve construction of dg Lie algebras studied thoroughly 
by Hinich in \cite{Hin0}. The homotopy invariance of this nerve relies on nilpotence conditions on the dg Lie algebra.
In the case of formal moduli problems, this nilpotence condition is always satisfied because one tensors the Lie algebra with the maximal ideal 
of an augmented Artinian cdga.

\smallskip

In practice it is often convenient to work with homotopy Lie algebras, that is, $L_{\infty}$-algebras, rather than strict dg-Lie algebras:
\begin{defn}\label{D:Linfty}
(1) An $L_{\infty}$-algebra is a graded vector space $g=\{g_n\}_{n\in\mathbb{Z}}$ equipped with maps
$l_k:g^{\otimes k}\rightarrow g$ of degree $2-k$, for $k\geq 1$, satisfying the following properties:
\begin{itemize}
\item $l_k(...,x_i,x_{i+1},...)=-(-1)^{|x_i||x_{i+1}|}l_k(...,x_{i+1},x_i,...)$
\item for every $k\geq 1$, the generalized Jacobi identities
\[
\sum_{i=1}^k\sum_{\sigma\in Sh(i,k-i)}(-1)^{\epsilon(i)}l_k(l_i(x_{\sigma(1)},...,x_{\sigma(i)}),x_{\sigma(i+1)},...,x_{\sigma(k)})=0
\]
where $\sigma$ ranges over the $(i,k-i)$-shuffles and
\[
\epsilon(i) = i+\sum_{j_1<j_2,\sigma(j_1)>\sigma(j_2)}(|x_{j_1}||x_{j_2}|+1).
\]
\end{itemize}
It is standard that the above definition is equivalent to the following:

(2) An $L_{\infty}$-algebra structure on a graded vector space $g=\{g_n\}_{n\in\mathbb{Z}}$ is exactly the data of a
coderivation $Q:Sym^{\bullet\geq 1}(g[1])\rightarrow Sym^{\bullet\geq 1}(g[1])$ of degree $1$ of the cofree cocommutative coalgebra 
$ Sym^{\bullet\geq 1}(g[1])$ such that $Q^2=0$.
\end{defn}
The bracket $l_1$ is in particular a differential that makes $g$ a cochain complex. 
The dg-coalgebra of (2) is called the reduced \emph{Chevalley-Eilenberg chain complex} of the $L_\infty$-algebra $g$, denoted  $\overline{C}^{CE}_*(g)$.
The dg-algebra $\overline{C}_{CE}^*(g)$ obtained by dualizing the dg coalgebra of (2) is called the (reduced) \emph{Chevalley-Eilenberg cochain algebra} 
of $g$.

\smallskip

\begin{defn}\label{D:filteredandcomplteLinfty} A $L_{\infty}$ algebra $g$ is \emph{filtered} if it admits a decreasing filtration
\[
g=F_1g\supseteq F_2g\supseteq...\supseteq F_rg\supseteq ...
\]
compatible with the brackets: for every $k\geq 1$,
\[
l_k(F_rg,g,...,g)\in F_rg.
\]
We require moreover that for every $r$, there exists an integer $N(r)$ such that $l_k(g,...,g)\subseteq F_rg$
for every $k>N(r)$.

A filtered $L_{\infty}$ algebra $g$ is \emph{complete} if the canonical map $g\rightarrow lim_rg/F_rg$ is an isomorphism.
\end{defn}
In particular a nilpotent $L_\infty$-algebra is complete and, if $\mathfrak{m}$ is the augmentation ideal of an Artinian cdga, then $g\otimes \mathfrak{m}$ is also complete for any $L_\infty$-algebra $g$.  

\smallskip

The completeness of a $L_{\infty}$ algebra allows to define properly the notion of Maurer-Cartan element:
\begin{defn}
(1) Let $g$ be a complete $L_{\infty}$-algebra and $\tau\in g^1$, we say that $\tau$ is a Maurer-Cartan element of $g$ if
\[
\sum_{k\geq 1} \frac{1}{k!} l_k(\tau,...,\tau)=0.
\]
The set of Maurer-Cartan elements of $g$ is noted $MC(g)$.

(2) The simplicial Maurer-Cartan set is then defined by
\[
MC_{\bullet}(g)=MC(g\hat{\otimes}\Omega_{\bullet}),
\],
where $\Omega_{\bullet}$ is the Sullivan cdga of de Rham polynomial forms on the standard simplex $\Delta^{\bullet}$ (see~\ref{SS:CDGA} and \cite{Sul})
and $\hat{\otimes}$ is the completed tensor product with respect to the filtration induced by $g$.
\end{defn}
The simplicial Maurer-Cartan set is a Kan complex, functorial in $g$ and preserves quasi-isomorphisms of complete $L_{\infty}$-algebras.
The \emph{Maurer-Cartan moduli set} of $g$ is $\mathcal{MC}(g)=\pi_0MC_{\bullet}(g)$: it is the quotient of the set
of Maurer-Cartan elements of $g$ by the homotopy relation defined by the $1$-simplices.
When $g$ is a \emph{complete dg Lie algebra}, it turns out that this homotopy relation is equivalent to the action of the gauge
group $exp(g^0)$ (a prounipotent algebraic group acting on Maurer-Cartan elements), so in this case
this moduli set coincides with the one usually known for Lie algebras.
We refer the reader to \cite{Yal2} for more details about all these results.

Let us explain briefly why Lurie's equivalence \cite[Theorem 2.0.2]{Lur0}  lifts from the $\infty$-category of dg Lie algebras $dgLie_{\mathbb{K}}$ 
to the $\infty$-category of $L_{\infty}$-algebras $L_{\infty}-Alg$.
Let $p:L_{\infty}\stackrel{\sim}{\rightarrow}Lie$ be the cofibrant resolution of the operad $Lie$ encoding $L_{\infty}$-algebras. 
This morphism induces a functor $p^*:dgLie_{\mathbb{K}}\rightarrow L_{\infty}-Alg$ which associates to any dg Lie algebra the $L_{\infty}$-algebra 
with the same differential, the same bracket of arity $2$ and trivial higher brackets in arities greater than $2$. 
This functor is a right Quillen functor belonging to a Quillen equivalence
\[
p_{!}:L_{\infty}-Alg\leftrightarrows dgLie_{\mathbb{K}} :p^*,
\]
since $p$ is a quasi-isomorphism of $\Sigma$-cofibrant operads (see \cite[Theorem 16.A]{Fre3}).
Quillen equivalences induce equivalences of the $\infty$-categories associated to these model categories. Therefore, 
we get a commutative triangle of equivalences of $\infty$-categories (the model of quasicategories is used in \cite{Lur1}, but actually any model works)
\[
\xymatrix{
L_{\infty}-Alg\ar[dr]^-{\tilde{\psi}}\ar[d]_-{p_{!}} & \\
dgLie_{\mathbb{K}} \ar[r]_-{\psi} & FMP_{\mathbb{K}}
}
\]
where $\psi$  maps a dg Lie algebra $g$ to the $\infty$-functor $Map_{cdga^{aug}}(C^*_{CE}(g),-)$, which is the derived mapping space in the $\infty$-category of augmented cdgas whose source is the (completed) Chevalley-Eilenberg algebra of $g$. Note that one can also consider $Map_{dgLie}(\mathcal{D}(-),g)$, where $\mathcal{D}:cdga^{aug}\rightarrow dgLie^{op}$ is the left adjoint of $C^*_{CE}(-)$ \emph{as an $\infty$-functor} (the \emph{strict} functor $CE$ does not admit such an adjoint).
\begin{cor}\label{C:LieFMP}
The $\infty$-functor $\tilde{\psi}$ admits an inverse which associates to a formal moduli problem $F$ an $L_{\infty}$-algebra that we denote by $\mathfrak{L}_F$.
\end{cor}
The Maurer-Cartan space functor $dgArt_{\mathbb{K}}^{aug} \ni R\mapsto MC_\bullet(g\otimes \mathfrak{m}_R)$ (where $\mathfrak{m}_R$ is the maximal ideal of $R$) is not sufficiently well behaved to be used in Lurie's proof, however, it is pointwise weakly equivalent to the aforementioned one.
Moreover, for any $L_{\infty}$-algebra $g$ we have
\begin{eqnarray*}
MC_\bullet(g\otimes \mathfrak{m}_R) & \simeq & MC_\bullet(p^*\mathbb{L}p_{!}g\otimes \mathfrak{m}_R)\\
 & = & MC_\bullet(\mathbb{L}p_{!}g\otimes \mathfrak{m}_R) \\
 & \simeq & MC_\bullet(p_{!}g\otimes \mathfrak{m}_R) \\
 & \simeq & Map_{cdga^{aug}}(CE(p_{!}g),R) \\
 & = & \tilde{\psi}(g).
\end{eqnarray*}
The first equivalence holds because the derived unit of the Quillen equivalence $(p_{!},p^*)$ is a quasi-isomorphism (notice here that we do not need to derive $p^*$ because all dg Lie algebras are already fibrant), the second because the Maurer-Cartan equation for an $L_{\infty}$-algebra with trivial higher brackets gives back the Maurer-Cartan equation of a dg Lie algebra, the third because the cofibrant resolution is a quasi-isomorphism.

The underlying complex of $\mathfrak{L}_F$ can be computed up to quasi-isomorphism by the shifted tangent complex of $F$. To define it, one starts by considering, for any integer $n$, a homotopy pullback of augmented Artinian cdgas
\[
\xymatrix{
\mathbb{K}\oplus\mathbb{K}[n]\ar[d]\ar[r] & \mathbb{K}\ar[d] \\
\mathbb{K}\ar[r] & \mathbb{K}\oplus\mathbb{K}[n+1]
}
\]
where $\mathbb{K}\oplus\mathbb{K}[n]$ is the square zero extension of $\mathbb{K}$ by $\mathbb{K}[n]$.
Such a square satisfies the conditions required to apply the infinitesimal cohesiveness property of $F$, and moreover $F(\mathbb{K})$ is contractible, hence inducing a weak equivalence of simplicial sets
\[
F(\mathbb{K}\oplus\mathbb{K}[n])\overset{\sim}{\rightarrow}pt\times_{\mathbb{K}\oplus\mathbb{K}[n+1]}^hpt\simeq \Omega_*F(\mathbb{K}\oplus\mathbb{K}[n+1]).
\]
We recognize here the structure of an $\Omega$-spectrum $\mathcal{T}_F$.
\begin{defn} \label{D:TgtcplxFMP}
The tangent complex of a formal moduli problem $F$ is the complex $\mathbb{T}_F$ associated to the $\Omega$-spectrum $\mathcal{T}_F$.
\end{defn}
The second part of \cite[Theorem 2.0.2]{Lur0} reads:
\begin{prop}
For any formal moduli problem $F$, we have a natural (in $F$) equivalence of complexes $\mathfrak{L}_F\simeq \mathbb{T}_F[-1]$.
\end{prop}

Let us say a word about \emph{formal deformations}. Although the ring of formal power series in one variable $\mathbb{K}[[t]]$ is not Artinian, given a formal moduli problem $F$, one can properly define the notion of formal deformation, or deformation over $\mathbb{K}[[t]]$, by setting
\[
F(\mathbb{K}[[t]]) \; := \;\mathop{lim}\limits_{i}\, F(\mathbb{K}[t]/t^i).
\]
(where we consider a homotopy limit in the $\infty$-category of simplicial sets).
By~\cite[Corollaire 2.11]{To-Bourbaki} (or \cite{Lur0}), there is an natural weak-equivalence
\[
F(\mathbb{K}[[\hbar]])  \; \simeq \; Map(\mathbb{K}[-1], \mathfrak{L}_F)
\]
where $\mathbb{K}[-1]$ is the one dimensional Lie algebra concentrated in degree $1$ with trivial Lie bracket. 
Here $Map$ denotes the derived mapping space in the $\infty$-category of dg Lie algebras, which can be explicited 
in the corresponding model category by taking a cofibrant resolution of $\mathbb{K}[-1]$. We refer the reader to \cite[Section 1.1]{To-Bourbaki} 
for example, to see an explicit construction of such a cofibrant resolution. The main point of interest for us here, 
is that the space of Lie morphisms
from such a resolution is equivalent to the space of Maurer-Cartan elements in formal power series without constant terms, that is
\[
F(\mathbb{K}[[t]])\simeq MC_{\bullet}(t \mathfrak{L}_F[[t]]).
\]
This means that formal deformations of a point can be explicitely described in terms of the corresponding Lie algebra. 

To conclude, we also recall  the notion of 
\begin{defn}[Twisting by a Maurer-Cartan element]\label{D:twistingMC}
The twisting of a complete $L_{\infty}$ algebra $g$ by a Maurer-Cartan element $\tau$ is the complete $L_{\infty}$ algebra $g^{\tau}$
with the same underlying graded vector space and new brackets $l_k^{\tau}$ defined by
\[
l_k^{\tau}(x_1,...,x_k)=\sum_{i\geq 0}\frac{1}{i!}l_{k+i}(\underbrace{\tau,...,\tau}_i,x_1,...,x_k)
\]
where the $l_k$ are the brackets of $g$.\end{defn}
Morphisms of $L_{\infty}$-algebras can be twisted as well: for any such morphism $f:g\rightarrow h$ and any Maurer-Cartan element $\tau$ of $g$, $f$ induces a morphism of $L_{\infty}$-algebras $f:g^{\tau}\rightarrow h^{f(\tau)}$. Moreover, if $f$ is a filtered quasi-isomorphism then so does its twisted version, see for example \cite[Proposition 3.8]{Yal1}.

\subsection{Moduli spaces of algebraic structures and their formal moduli problems}\label{SS:DefALgStructClassic}
We now explain  a first prop(erad)ic approach to moduli of algebraic structures. 

Moduli spaces of algebraic structures were originally defined by Rezk as simplicial sets, in the setting of simplicial \emph{operads} \cite{Rez}.
This notion can be extended to algebras over differential graded \emph{props} as follows (see \cite{Yal2}):
\begin{defn}\label{D:mappingspace}
Let $P_{\infty}$ be a cofibrant prop and $X$ be a complex. The \emph{moduli space of $P_{\infty}$-algebra structures} on $X$ is the simplicial set $P_{\infty}\{X\}$ defined by
\[
P_{\infty}\{X\} = Mor_{Prop}(P_{\infty},End_X\otimes\Omega_{\bullet}),
\]
where the prop $End_X\otimes\Omega_{\bullet}$ is defined by $$(End_X\otimes\Omega_{\bullet})(m,n)=End_X(m,n)\otimes\Omega_{\bullet}$$ and $\Omega_{\bullet}$ 
is the Sullivan cdga of the standard simplex $\Delta^{\bullet}$ (see~\ref{SS:CDGA}).

Given a \emph{cofibrant} prop $P_\infty$ and any  prop $Q$, we will denote
\[
Map_{Prop}(P,Q):= Mor_{Prop}(P,Q \otimes \Omega_\bullet)
\]
the \emph{mapping space of prop morphisms}. The same definition holds for mapping spaces of properads.
\end{defn}
Indeed, the aritywise tensor product $(-)\otimes\Omega_{\bullet}$ forms a functorial simplicial resolution in the model category of dg props \cite[Proposition 2.5]{Yal2}.
This simplicial set enjoys the following key properties, see~\cite{Yal2}:
\begin{prop}\label{P: classif}
(1) The simplicial set $P_{\infty}\{X\}$ is a Kan complex and
\[
\pi_0P_{\infty}\{X\} = [P_{\infty},End_X]_{Ho(Prop)}
\]
is the set of homotopy classes of $P_{\infty}$-algebra structures on $X$.

(2) Any weak equivalence of cofibrant props $P_{\infty}\stackrel{\sim}{\rightarrow}Q_{\infty}$ induces a weak equivalence of Kan complexes $Q_{\infty}\{X\}\stackrel{\sim}{\rightarrow}P_{\infty}\{X\}$.
\end{prop}
We can extend the moduli space of $P_\infty$-structure to a simplicial presheaf by base change from $\mathbb{K}$ to any Artinian cdga. 
\begin{defn}\label{D:spresheafofAlgstructure} Let $P_\infty$ be a cofibrant prop and $X$ be a complex.
 We define a simplicial presheaf $ \underline{P_{\infty}\{X\}}: dgArt_{\mathbb{K}}^{aug} \to sSet$ by the formula
$$
 \underline{P_{\infty}\{X\}}: A\in dgArt_{\mathbb{K}}^{aug}\mapsto P_{\infty}\otimes A\{X\otimes A\}_{Mod_A}
$$
where $P_{\infty}\otimes A\{X\otimes A\}_{Mod_A}$ is the mapping space of dg props in $A$-modules $Map(P_{\infty}\otimes A,End_{X\otimes A}^{Mod_A})$ 
and $End_{X\otimes A}^{Mod_A}$ is the endomorphism prop of $X\otimes A$ taken in the category of $A$-modules.
\end{defn}
In other words, $\underline{P_{\infty}\{X\}}(A)$ is the \emph{simplicial moduli space} of $P_{\infty}$-algebra structures on $X\otimes A$
in the category of $A$-modules. 
Indeed, since $Mod_A$ is tensored over $Ch_{\mathbb{K}}$, one can make $P_{\infty}$ act on $A$-modules either by morphisms of dg props in $A$-modules 
from $P_{\infty}\otimes A$ to the endomorphism prop defined by the internal hom of $Mod_A$, 
or by morphisms of dg props from $P_{\infty}$ to the endomorphism prop defined by the external hom of $Mod_A$. See for instance \cite[Lemma 3.4]{Yal2}.

By Proposition~\ref{P: classif}, the simplicial set $\underline{P_{\infty}\{X\}}(A)$ classifies $P_{\infty}\otimes A$-algebra structures on $X\otimes A$. 
However, the simplicial presheaf $\underline{P_{\infty}\{X\}}$ is not a formal moduli problem, since $\underline{P_{\infty}\{X\}}(\mathbb{K})$ 
is in general not contractible.

\begin{defn}\label{D:fmpofAlgStruct} The \emph{formal moduli problem $\underline{P_{\infty}\{X\}}^{_{\mathbb{K}}}$
controlling (a certain type of) formal deformations of a $P_{\infty}$-algebra structure $_{\mathbb{K}}:P_{\infty}\rightarrow End_X$ on $X$ is} defined, 
on any augmented Artinian cdga $R$, as \emph{the homotopy fiber}
\begin{equation}\label{eq:defFMPAlgStructure}
\underline{P_{\infty}\{X\}}^{\psi}(R)=hofib(\underline{P_{\infty}\{X\}}(R)\rightarrow \underline{P_{\infty}\{X\}}(\mathbb{K}))
\end{equation}
taken over the base point $\psi$, the map being induced by the augmentation $R\rightarrow\mathbb{K}$.
\end{defn}
The moduli spaces of algebraic structures  and its associated formal moduli problem are encoded by $L_\infty$-algebras according to
Lurie - Pridham Theorem.
We now explain how those $L_\infty$-structures can be described explicitly using \emph{dg-properads} following \cite{MV1}
and~\cite{Yal2}. From now and until the end of Section $1$, we restrict to properads instead of props. 

\smallskip

Cofibrant resolutions of a properad $P$ can always be obtained as a cobar construction $\mathcal{B}^{co}(C)$
on some coproperad $C$ (which is usually the bar construction or the Koszul dual if $P$ is Koszul).
Given a cofibrant resolution $P_{\infty}:=\mathcal{B}^{co}(C)\stackrel{\sim}{\rightarrow}P$ of $P$ and another properad $Q$,
one constructs the convolution dg Lie algebra $Hom_{\Sigma}(\overline{C},Q)$:
\begin{defn}\label{D:convolutionLieofProperad}
 Let $C$ be a coaugmented coproperad and $Q$ be a properad. Their associated \emph{convolution} dg Lie algebra is the dg $\mathbb{K}$-module
  $$Hom_{\Sigma}(\overline{C},Q)$$ of morphisms
of $\Sigma$-biobjects from the augmentation ideal of $C$ to $Q$ endowed with the differential induced by the internal ones of $C$ and $Q$.
It is equipped with the  Lie bracket given by the antisymmetrization
of the convolution product.
\end{defn}
 This convolution product is defined similarly to the convolution product of
morphisms from a coalgebra to an algebra, using the infinitesimal coproduct of $C$ and the infinitesimal
product of $Q$. 

The total complex $Hom_{\Sigma}(\overline{C},Q)$ is a complete dg Lie algebra.
More generally, if $P$ is a properad with minimal model $(\mathcal{F}(s^{-1}\overline{C}),\partial)\stackrel{\sim}{\rightarrow}P$
for a certain homotopy coproperad $C$ (see \cite[Section 4]{MV1} for the definition of homotopy coproperads), and $Q$ is any properad, 
then \emph{the complex $Hom_{\Sigma}(\overline{C},Q)$ is a complete dg $L_{\infty}$ algebra} (which is not a dg-Lie algebra in general).

The  simplicial mapping space of morphisms $P_{\infty}\rightarrow Q$ is computed by the convolution 
$L_{\infty}$-algebra $Hom_{\Sigma}(\overline{C},Q)$ thanks to the following theorem:
\begin{thm}(cf. \cite[Theorem 2.10,Corollary 4.21]{Yal2})\label{T:Yal2}
Let $P$ be a dg properad equipped with a semi-free resolution $P_{\infty}:=(\mathcal{F}(s^{-1}\overline{C}),\partial)\stackrel{\sim}{\rightarrow}P$ and $Q$ be 
a dg properad. The simplicial presheaf
\[
\underline{Map}(P_{\infty},Q):R\in dgArt_{\mathbb{K}}^{aug}\mapsto Map_{Prop}(P_{\infty},Q\otimes R)
\]
is equivalent to the simplicial presheaf
\[
\underline{MC_{\bullet}}(Hom_{\Sigma}(\overline{C},Q)):R\in dgArt_{\mathbb{K}}^{aug}\mapsto MC_{\bullet}(Hom_{\Sigma}(\overline{C},Q)\otimes R)
\]
associated to the complete $L_{\infty}$-algebra $Hom_{\Sigma}(\overline{C},Q)$.
\end{thm}
Note that by \cite[Corollary 2.4]{Yal1}, the tensor product $MC_{\bullet}(Hom_{\Sigma}(\overline{C},Q)\otimes R)$ does not need to be completed because $R$  is Artinian. In order to get a formal moduli problem, we also consider the simplicial presheaf
\[
\underline{MC_{\bullet}^{fmp}}(Hom_{\Sigma}(\overline{C},Q)):R\in dgArt_{\mathbb{K}}^{aug}\mapsto MC_{\bullet}(Hom_{\Sigma}(\overline{C},Q)\otimes m_R),
\]
where $m_R$ is the maximal ideal of $R$. This presheaf is a formal moduli problem associated to $Hom_{\Sigma}(\overline{C},Q)$. 
In the case where $Q=End_X$, Theorem~\ref{T:Yal2} implies that the 
 the simplicial presheaf $\underline{MC_{\bullet}}(Hom_{\Sigma}(\overline{C}, End_X))$  is equivalent to $\underline{P_{\infty}\{X\}}$
 (definition~\ref{D:spresheafofAlgstructure}).
 
This theorem applies in particular to the case of a Koszul properad, 
which includes for instance Frobenius algebras, Lie bialgebras and their variants such as involutive Lie bialgebras in equivariant string topology \cite{CS2}.
It applies also to more general situations such as the properad $Bialg$ encoding associative and coassociative bialgebras, 
which is homotopy Koszul \cite[Proposition 41]{MV1}.

 \smallskip

We now describe the $L_\infty$-algebra structure encoding this formal moduli problem. It is given by twisting the convolution Lie algebra as follows.
The \emph{twisting} of $Hom_{\Sigma}(\overline{C},End_X)$ by a properad morphism $\varphi:P_{\infty}\rightarrow End_X$
is often called the \emph{deformation complex}\footnote{Proposition~\ref{P:TangentLinfty} belows justifies the name, although of course one has to  be careful
about which kind of deformation it encodes} of $\varphi$, and we have an isomorphism
\[
g_{P,X}^{\varphi} = Hom_{\Sigma}(\overline{C},End_X)^{\varphi} \cong Der_{\varphi}(\Omega(C),End_X)
\]
where the right-hand term is the complex of derivations with respect to $\varphi$ \cite[Theorem 12]{MV2}.
\begin{prop}\label{P:TangentLinfty}
The tangent $L_{\infty}$-algebra of the formal moduli problem $\underline{P_{\infty}\{X\}}^{\varphi}$ is given by
\[
g_{P,X}^{\varphi} = Hom_{\Sigma}(\overline{C},End_X)^{\varphi}.
\]
\end{prop}
\begin{proof}
Let $R$ be an augmented Artinian cdga.
By Theorem~\ref{T:Yal2}, we have the homotopy equivalences
\begin{eqnarray*}
\underline{P_{\infty}\{X\}}^{\varphi}(R) & \simeq & hofib(\underline{MC_{\bullet}}(g_{P,X})(R)\rightarrow \underline{MC_{\bullet}}(g_{P,X})(\mathbb{K})) \\
 & = & hofib(MC_{\bullet}(g_{P,X}\otimes R)\rightarrow MC_{\bullet}(g_{P,X}))
\end{eqnarray*}
where $hofib(MC_{\bullet}(g_{P,X}\otimes R)\rightarrow MC_{\bullet}(g_{P,X}))$ is the homotopy fiber, over the Maurer-Cartan element $\varphi$, of the simplicial map induced by the $L_{\infty}$-algebra morphism $p:g_{P,X}\otimes R\rightarrow g_{P,X}$ given by the tensor product of the augmentation $R\rightarrow\mathbb{K}$ with $g_{P,X}$. The unit of $R$ induces a morphism $i:g_{P,X}\rightarrow g_{P,X}\otimes R$ such that $p\circ i=Id_{g_{P,X}}$ (the composite of the unit of $R$ by its augmentation is $id_R$). According to the proof of \cite[Lemma 1.18]{MS1}, this homotopy fiber is $MC_{\bullet}(ker(\tilde{p}))$, where $\tilde{p}:(g_{P,X}\otimes R)^{\varphi\otimes 1_R}\rightarrow g_{P,X}^{\varphi}$ is the morphism of $L_{\infty}$-algebras still defined by $p$ but between the $L_{\infty}$-algebras twisted respectively by $\varphi\otimes 1_R$ and $p(\varphi\otimes 1_R)=\varphi$ (note that the enunciate of \cite[Lemma 1.18]{MS1} is a bit misleading as written, since they consider $ker(p)^{i(x)}$ even though $i(x)$ is not an element of $ker(p)$ for a non trivial $x$, hence the little correction here). We have $(g_{P,X}\otimes m_R)^{\varphi\otimes 1_R}=g_{P,X}^{\varphi}\otimes m_R$ by definition of the $L_{\infty}$-algebra structure on $g_{P,X}\otimes m_R$ (the reader may notice that given any $L_{\infty}$-algebra $\mathfrak{g}$, any Maurer-Cartan element $\varphi$ of $g$ and any cdga $R$, one has $(\mathfrak{g}\otimes R)^{\varphi\otimes 1_R}=\mathfrak{g}^{\varphi}\otimes R$). The kernel of $\tilde{p}$ is the tensor product of $g_{P,X}^{\varphi}$ with the kernel of the augmentation, which is the maximal ideal $m_R$ of $R$, so $ker(\tilde{p})=g_{P,X}^{\varphi}\otimes m_R$.

Our homotopy fiber is $MC_{\bullet}(g_{P,X}^{\varphi}\otimes m_R)$, so there is an equivalence of formal moduli problems
\[
\underline{P_{\infty}\{X\}}^{\varphi} \simeq \underline{MC_{\bullet}^{fmp}}(g_{P,X}^{\varphi}).
\]
By Lurie's equivalence theorem, this means that $g_{P,X}^{\varphi}$ is the Lie algebra of the formal moduli problem $\underline{P_{\infty}\{X\}}^{\varphi}$.
\end{proof}

\section{Derived formal groups of algebraic structures and associated formal moduli problems}\label{S:formalgeom}

In this section, we explain how the theory of  formal moduli problems is related to derived formal groups,
and how this allows to state the correspondence between formal groups and  Lie algebras at a higher (and derived) level of generality. 
This correspondence is suitable for us to define a natural deformation problem of homotopy $P$-algebras structures on a complex $X$ up to quasi-isomorphisms and to understand how it relates to those associated moduli space of algebraic structures from section~\ref{SS:DefALgStructClassic}.

\begin{defn}\label{D:derivedgpobkects} Let $\mathcal{C}$ be a stable $\infty$-category. We denote by 
  $Mon_{E_1}^{gp}(\mathcal{C})$ the $\infty$-category of grouplike $E_1$-monoids in $\mathcal{C}$, that is the subcategory of 
  grouplike objects in the $\infty$-category of $E_1$-algebras in $\mathcal{C}$ equipped with the cartesian monoidal structure. 
  here an $E_1$-monoid $G$ is said to be \emph{grouplike} if the two canonical maps $ (\mu, \pi_i):G\times G \to G\times G$  (induced by the multiplication $\mu:G\times G\to G$
  and the two canonical projections $\pi_1,\, \pi_2:G\times G\to G$) are equivalences.
  
A \emph{group object} of $\mathcal{C}$ is an object of   $Mon_{E_1}^{gp}(\mathcal{C})$.
\end{defn}

\begin{example}
Loop spaces provide the main source of examples of group objects in topology (which are also called $H$-groups in this particular setting). A topological monoid $M$ is said to be grouplike if $\pi_0M$ is a group, and since any grouplike topological monoid is equivalent to a loop space, grouplike topological monoids model group objects in the $\infty$-category of topological spaces in the sense of Definition~\ref{D:derivedgpobkects}. The same holds true for grouplike simplicial monoids, which model group objects in the $\infty$-category of simplicial sets and will be especially useful for us to study homotopy automorphisms of algebras.
\end{example}

\subsection{Generalities on derived formal groups}

First, let us remark that the category of  formal moduli problems is pointed. 
In fact, we have :
\begin{lem}\label{L:pointedSPsh}
Let $SPsh^{pt}((dgArt_{\mathbb{K}}^{aug})^{op})$ be the full sub-$\infty$-category of $SPsh((dgArt_{\mathbb{K}}^{aug})^{op})$ consisting of those $\infty$-functors $F$ 
from augmented dg Artinian algebras to simplicial sets 
such that $F(\mathbb{K})$ is contractible. This $\infty$-category is pointed.
\end{lem}
\begin{proof}
Let us 
first note $pt$ the $\infty$-functor sending any augmented Artinian cdga to the simplicial set generated by a single vertex. 
Now let $R$ and $R'$ be two augmented Artinian cdgas, let us write $\eta_R,\eta_{R'}$ for their respective unit morphisms and $\epsilon_R,\epsilon_{R'}$ their respective augmentations. Let $f:R\rightarrow R'$ be a morphism of augmented Artinian cdgas. A morphism of augmented Artinian cdgas commutes with units, so the diagram
\[
\xymatrix{
pt(R)\ar[r]^-{\sim}\ar[d]_-= & F(\mathbb{K})\ar[d]^-=\ar[r]^-{F(\eta_R)} & F(R)\ar[d]^-f \\
pt(R')\ar[r]^-{\sim} & F(\mathbb{K})\ar[r]^-{F(\eta_{R'})} & F(R')
}
\]
commutes as well, hence we get a unique morphism of $\infty$-functors $pt\rightarrow F$.
A morphism of augmented Artinian cdgas commute with augmentations, so the diagram
\[
\xymatrix{
F(R)\ar[d]^-f\ar[r]^-{F(\epsilon_R)} & F(\mathbb{K})\ar[d]^-=\ar[r]^-{\sim} & pt(R)\ar[d]_-=  \\
F(R')\ar[r]^-{F(\epsilon_{R'})} & F(\mathbb{K})\ar[r]^-{\sim} & pt(R')
}
\]
commutes as well, hence a unique morphism $F\rightarrow pt$.
\end{proof}
Consequently, one can form the \emph{pointed loop space functor} as the homotopy pullback \begin{equation}\label{eq:defLoopofFMP} \Omega_*F \,:=\, pt\times^h_F pt \end{equation} in $SPsh^{pt}((dgArt_{\mathbb{K}}^{aug})^{op})$.
Let us note that since $dgArt_{\mathbb{K}}^{aug}$ and $sSet$ are presentable, 
the $\infty$-category of pointed $\infty$-functors is presentable as well.
Therefore $SPsh^{pt}((dgArt_{\mathbb{K}}^{aug})^{op})$ is a presentable pointed $\infty$-category, and $FMP_{\mathbb{K}}$ is a presentable pointed sub-$\infty$-category of it. Therefore the universal property of homotopy pullbacks makes $\Omega_*F$ into a group-like $E_1$-monoid in simplicial presheaves (see~\ref{P:PropertiesofOmega}.(1) below). 

Moreover, the inclusion 
\[
i:FMP_{\mathbb{K}}\hookrightarrow  SPsh^{pt}((dgArt_{\mathbb{K}}^{aug})^{op})
\]
commutes with small homotopy limits, and since small homotopy limits in $SPsh^{pt}((dgArt_{\mathbb{K}}^{aug})^{op})$ are determined pointwise,   we have proved
\begin{lem}\label{L:LoopsofFMPisloops}
For any derived formal moduli problem $F$ and any augmented Artinian cdga $R$, we have $(\Omega_*F)(R) \, \cong \,\Omega_{\eta_R}F(R)$ where the base point of $F(R)$ is given by the morphism $F(\eta_R):pt\simeq F(\mathbb{K})\rightarrow F(R)$ induced by the unit $\eta_R$ of $R$.
\end{lem}

The base point of $F(R)$ corresponds to the \lq\lq{}trivial $R$-deformation\rq\rq{} of the unique point of $F(\mathbb{K})$. It is important to mention that $FMP_{\mathbb{K}}$ is presentable \cite[Remark 1.1.17]{Lur0} and that the inclusion of $FMP_{\mathbb{K}}$ in pointed $\infty$-functors admits a left adjoint (applying the $\infty$-categorical adjoint functor theorem)
\[
L: SPsh^{pt}((dgArt_{\mathbb{K}}^{aug})^{op})\rightarrow FMP_{\mathbb{K}}
\]
making a simplicial presheaf canonically into a formal moduli problem.
When $F$ is a formal moduli problem, then $LF \cong F$, otherwise $LF$ is the (best) formal moduli problem approximating the pointed $\infty$-functor $F$.
The functor $L$ is hard to understand explicitely in general, but is related to the (standard) pointwise classifying space functor, in the sense that we have a natural equivalence \begin{equation}L(B\Omega_*F) \, \cong \, LF \end{equation}
where $B$ is given by applying objectwise the classifying space functor from $E_1$-monoids in spaces to spaces. 

\smallskip

The loop space functor enjoys the following properties (as a consequence of Lurie's work \cite{Lur0}, see for example \cite[Proposition 2.15]{BKP} for a proof):
\begin{prop}\label{P:PropertiesofOmega}
\begin{itemize}
\item[(1)] Let $\mathcal{C}$ be a pointed presentable $\infty$-category. The pointed loop space $\infty$-functor lifts to a ($\infty$-categorical) limit preserving functor
\[
\Omega_*:\mathcal{C}\rightarrow Mon_{E_1}^{gp}(\mathcal{C})
\]
where $Mon_{E_1}^{gp}(\mathcal{C})$ (\ref{D:filteredandcomplteLinfty}) is the $\infty$-category of grouplike $E_1$-monoids in $\mathcal{C}$.

\item[(2)] In the case $\mathcal{C}=FMP_{\mathbb{K}}$, the loop space functor is an equivalence.
\end{itemize}
\end{prop}
\begin{defn}\label{D:derivedformalgroup}
 A \emph{derived formal group} is an object of $Mon_{E_1}^{gp}(FMP_{\mathbb{K}})$, that is a group object in the pointed $\infty$-category of formal moduli problems.
\end{defn}
By proposition~\ref{P:PropertiesofOmega}.(2), the functor $\Omega_*$
has a left adjoint \begin{equation}
                  \label{eq:defBfmp}   
                  B_{fmp}: Mon_{E_1}^{gp}(FMP_{\mathbb{K}}) \longrightarrow FMP_{\mathbb{K}}.
                   \end{equation}
The functor  $B_{fmp}$ is obtained as a generalized bar construction given by the realization of a simplicial object in derived formal moduli problems, 
hence a homotopy colimit corresponding to a classifying space $\infty$-functor for derived formal groups (see \cite[Lemma 2.16]{BKP} and \cite[Remark 5.2.2.8]{Lur2}). 
Composing equivalence (2) with Lurie's equivalence theorem \cite{Lur0} result into the equivalence 
\[ Mon_{E_1}^{gp}(FMP_{\mathbb{K}}) \, \cong \, FMP_{\mathbb{K}} \, \cong \, L_\infty-Alg\]
 between the $\infty$-category of dg-Lie algebras and derived formal groups.

 This equivalence is an  analogue to the classical correspondence between formal/algebraic/Lie groups and Lie algebras. This equivalence holds true not only in the commutative case but also for iterated loop spaces and noncommutative moduli problems, see \cite[Proposition 2.15]{BKP}.
\begin{rem}\label{R:connectfmp}
Note that $B_{fmp}$ is \emph{not} defined pointwise by the standard classifying space. If it was so, then, given a formal moduli problem $F$, for any Artinian cdga $R$, there would be an equivalence $B\Omega_{\eta_R}F(R)\simeq F(R)$, which would imply that $F(R)$ is connected. This is not the case, since $F(R)$ is equivalent to the nerve of the dg Lie algebra $\mathfrak{L}_F\otimes m_R$ (where $\mathfrak{L}_F$ is the dg Lie algebra of $F$ via Lurie-Pridham correspondence), and the connected components of the later are the equivalence classes of Maurer-Cartan elements of $\mathfrak{L}_F$.
\end{rem}

Let us recall from Definition \ref{D:TgtcplxFMP} the structure of the $\Omega$-spectrum $\mathcal{T}_F$ whose associated complex is $\mathbb{T}_F$:
\[
F(\mathbb{K}\oplus\mathbb{K}[n])\overset{\sim}{\rightarrow}pt\times_{\mathbb{K}\oplus\mathbb{K}[n+1]}^hpt\simeq \Omega_*F(\mathbb{K}\oplus\mathbb{K}[n+1]).
\]
Now, recall that the pointed loop space functor for formal moduli problems is determined pointwise by the standard pointed loop space, 
so $\mathcal{T}_{\Omega_*F}\simeq \Omega \mathcal{T}_F$, which means that  
\begin{equation}\label{eq:TofOmega}\mathbb{T}_{\Omega_*F}\;\cong \; \mathbb{T}_F[-1]\end{equation} for the corresponding complexes.

\begin{rem}\label{R:nontrivdef}
Since $(\Omega_*F)(R)\simeq \Omega_{\eta_R}F(R)$ (lemma~\ref{L:LoopsofFMPisloops}), the derived formal group of a formal moduli problem $F$ seems to retain, for any $R$, only the informations about the connected component of the trivial $R$-deformation. However, all the information of the deformation problem is in fact contained here, since its tangent complex gives the dg Lie algebra controling it. To understand  how this is possible, let us remind that by infinitesimal cohesiveness of $F$ we have, for example, equivalences
\[
F(\mathbb{K}\oplus\mathbb{K}[n])\simeq \Omega_*F(\mathbb{K}\oplus\mathbb{K}[n+1])
\]
which means that the space of $\mathbb{K}\oplus\mathbb{K}[n]$-deformations is equivalent to the space of self-equivalences of the trivial $\mathbb{K}\oplus\mathbb{K}[n+1]$-deformation. For example, deformations over the algebra of dual numbers $\mathbb{K}[t]/(t^2)$ are recovered as loops over the trivial $\mathbb{K}[\epsilon]/(\epsilon^2)$-deformation with $\epsilon$ of degree $1$.
\end{rem}
More generally, if $F$ is a pointed $\infty$-functor such that $\Omega_*F$ is a derived formal group (e.g. a $1$-proximate moduli problem in the sense of \cite[Definition 5.1.5]{Lur0}, see also \cite[Lemma 2.11]{BKP}), then $$\mathbb{T}_{\Omega_*F}\,\simeq \,\mathbb{T}_{LF}[-1].$$ This comes from \cite[Lemma 5.1.12]{Lur0}. In other words, \emph{The derived formal group $\Omega_*F$ controls the deformations parametrized by the formal moduli completion of $F$}.

Note  that many functors are not representable by a derived stack via Lurie's representability theorem \cite{Lur02}, but produce nethertheless derived formal moduli problems when restricted to Artinian cdgas, so one can associate a Lie algebra to them without any representability condition. 

\begin{example}\label{Ex:infcohesive}A case of interest for us is when $F$ is an infinitesimally cohesive (in the sense of \cite[Definition 2.1.1]{Lur02}) simplicial $\infty$-presheaf over $(dgArt_{\mathbb{K}}^{aug})^{op}$.
That is a  simplicial presheaf preserving weak equivalences and satisfying the derived Schlessinger condition (~\ref{D:FMP}),
but such that $F(\mathbb{K})$ is not (necessarily) contractible. Then one can nethertheless attach to any $\mathbb{K}$-point $x\in F(\mathbb{K})$ a derived formal moduli problem $\widehat{F_x}$ by setting
\[
\widehat{F_x}(R) = hofib_x(F(R)\rightarrow F(\mathbb{K})),
\]
where the map is induced by the augmentation $R\rightarrow \mathbb{K}$ of the Artinian cdga $R$ (see the proof of~\ref{L:formalofderivedGp}).
Thus, one attaches to any $x\in F(\mathbb{K})$ a derived formal group by taking the pointed loop space of the construction above. 
Hence, such a $F$ parametrizes a family of derived formal moduli problems over $F(\mathbb{K})$. 
\end{example}

\subsection{Derived prestack group and their tangent $L_\infty$-algebras}\label{SS:DerivedPreStackGp}

We will now study families of derived formal groups, which we call derived prestack groups. These are analogues of Lie groups but in the context 
of  infinitesimally cohesive prestacks instead of manifolds. 
In particular, they have an associated $L_\infty$-algebra given by their tangent space at the 
neutral element.

\smallskip 

Let us denote by $SPsh^{infcoh}_{\infty}((dgArt_{\mathbb{K}}^{aug})^{op})$ 
the $\infty$-category of infinitesimally cohesive $\infty$-functors on $dgArt_{\mathbb{K}}^{aug}$ 
with values in simplicial sets. Here we denote by $SPsh$ the simplicial presheaves and $infcoh$ the infinitesimal cohesiveness of the corresponding $\infty$-functors.
We can consider  its $\infty$-category of group objects (Definition~\ref{D:derivedgpobkects}); that is we introduce the following definition:
\begin{defn}\label{D:derivedgroup}
A \emph{derived prestack group} is a group object in the $\infty$-category $SPsh^{infcoh}_{\infty}((dgArt_{\mathbb{K}}^{aug})^{op})$.
More precisely, the $\infty$-category 
of derived prestack groups is $Mon_{E_1}^{gp}(SPsh^{infcoh}_{\infty}((dgArt_{\mathbb{K}}^{aug})^{op}))$.
\end{defn}
The relevance of the definition is given by the following
\begin{lem}\label{L:formalofderivedGp}
Let $G$ be a derived prestack group. For any $x\in G(\mathbb{K})$, the completion \[\widehat{G_x}:=\Big(R\mapsto hofib_x\big(G(R)\rightarrow G(\mathbb{K})\big)\Big)\] is a formal derived group. 
\end{lem}
\begin{proof}The map $G(R)\rightarrow G(\mathbb{K})$ is induced by the augmentation $R\to \mathbb{K}$ of $R$.
 Since the homotopy fiber is an $\infty$-limit, it preserves the infinitesimally cohesive condition and weak equivalences. 
By definition, the homotopy fiber computed for $R=\mathbb{K}$ is a point and therefore $\widehat{G_x}$ is a formal moduli problem according to definition~\ref{D:FMP}.
\end{proof}
In other words, a derived prestack group $G$ is a family of derived formal groups parametrized by $G(\mathbb{K})$. 
In what follows, we will by especially interested in the formal completion at the neutral element.
The pointed loop space construction commutes with homotopy fibers, so for any $F\in SPsh^{infcoh}_{\infty}((dgArt_{\mathbb{K}}^{aug})^{op})$ 
and any $x\in F(\mathbb{K})$, we have
\[
\widehat{(\Omega_xF)_e} = \Omega_*\widehat{F_x}.
\]
Hence the derived formal group associated to the derived prestack group $\Omega_xF$ by completion at the constant loop
is the derived formal group corresponding 
to the formal moduli problem $\widehat{F_x}$.
Moreover, since the loop space functor preserves weak equivalences as a homotopy limit, it leads to a homotopy invariance of derived formal groups with respect to their underlying derived prestacks:
\begin{lem}\label{L:invarianceformalgroups}
Let $f:F\rightarrow G$ be a morphism in $SPsh^{infcoh}_{\infty}((dgArt_{\mathbb{K}}^{aug})^{op})$. If $f$ is a weak equivalence, then for any $x\in F(\mathbb{K})$ it induces:
\begin{itemize}
\item[-] a weak equivalence of derived prestack groups $\Omega_xF\stackrel{\sim}{\rightarrow} \Omega_{f(x)}G$;
\item[-] a weak equivalence of derived formal moduli problems $\widehat{F_x}\stackrel{\sim}{\rightarrow}\widehat{G_{f(x)}}$;
\item[-] a weak equivalence of derived formal groups  $\Omega_*\widehat{F_x}\stackrel{\sim}{\rightarrow}\Omega_*\widehat{G_{f(x)}}$.
\end{itemize}
\end{lem}
\begin{rem}
One cannot expect the formal completion of any derived stack at a point to produce a derived formal group and a corresponding tangent Lie algebra, 
because of the lack of cohesiveness. However, any derived Artin stack (that is, geometric for smooth morphisms) is in particular cohesive, 
see for instance \cite[Corollary 6.5]{Lur01} and \cite[Lemma 2.1.7]{Lur02}.
\end{rem}
To legitimate constructions we are going to use in the next section, it is  worth mentionning the following properties 
of infinitesimally cohesive simplicial presheaves:
\begin{lem}\label{L:infcohppties}

(1) The $\infty$-category $SPsh^{infcoh}_{\infty}((dgArt_{\mathbb{K}}^{aug})^{op})$ is stable under small limits.

(2) If $\mathcal{C}$ and $\mathcal{D}$ are two equivalent $\infty$-categories,  the $\infty$-categories $SPsh^{infcoh}_{\infty}(\mathcal{C}^{op})$ 
and $SPsh^{infcoh}_{\infty}(\mathcal{D}^{op})$ are equivalent as well.
\end{lem}
\begin{proof}

(1) Follows from the definition of infinitesimally cohesive $\infty$-functors \cite[Remark 2.1.11]{Lur02}.

(2) This is just a particular case of an equivalence of $\infty$-categories of sheaves induced by an equivalence of their $\infty$-sites, here with the discrete Grothendieck topology.
\end{proof}

\begin{defn} \emph{(Tangent Lie algebra of derived groups)}\label{D:TgtofGpPreStack}
\begin{itemize}
 \item Let $\widehat{G}$ be a derived formal group (\ref{D:derivedformalgroup}). Its \emph{tangent homotopy Lie algebra} is 
 \[ Lie\big(\widehat{G}\big):= \mathfrak{L}_{B_{fmp}\big(\widehat{G}\big)} \in Lie_\infty-Alg\]
 where $B_{fmp}$ is the equivalence~\eqref{eq:defBfmp} and $\mathfrak{L}_{(-)}$ the one of~\ref{C:LieFMP}.
 \item 
Let $G$ be a derived prestack group (\ref{D:derivedgroup}). Its \emph{tangent homotopy Lie algebra} is 
\[ Lie(G):= Lie( \widehat{G}_{1} )\] where  $\widehat{G}_1$ is the formal completion at the unit of $G$ (\ref{L:formalofderivedGp}).
\end{itemize}
\end{defn}
The following result shows that  the tangent at the identity of a derived prestack group inherits a canonical structure of homotopy Lie algebra (which completely determines it if it is actually a derived formal group).
\begin{prop}\label{P:LieofGroup} Let $G$ be a derived prestack group. 
\begin{enumerate}\item There is an equivalence of underlying complexes $Lie(G)  \, \cong \, (\mathbb{T}_{G})_1$ between its Lie algebra and its tangent space at $1$. 
\item If $F$ is a formal moduli problem and $G\cong \Omega F$, then $Lie(G) = \mathfrak{L}_F$.
\item For any point $x$ in $G(\mathbb{K})$, $(\mathbb{T}_{G})_x \, \cong \, (\mathbb{T}_{G})_1$.\end{enumerate}
\end{prop}
\begin{proof}
By Proposition~\ref{P:PropertiesofOmega} and~\eqref{eq:TofOmega} we  have equivalences of complexes
\[\mathbb{T}_{\widehat{G}_{1}} \, \cong \, \mathbb{T}_{\Omega B_{fmp}\widehat{G}_{1}} \, \cong\, \mathbb{T}_{B_{fmp}\widehat{G}_{1}}[-1]\, \cong\, 
\mathfrak{L}_{B_{fmp}\widehat{G}_{1}}\]
where the first equivalence follows from the fact that $\Omega B_{fmp}$ is equivalent to the identity, the second equivalence from \ref{eq:TofOmega} and the third equivalence from Lurie's result ~\cite{Lur02} asserting that the underlying complex of the Lie algebra $\mathfrak{L}_F$ of a formal moduli problem  $F$ is equivalent to $\mathbb{T}_F[-1]$. The first claim follows then from Definition~\ref{D:TgtofGpPreStack}.

The second claim follows from the fact that $\Omega B_{fmp}$ is the identity and Lemma~\ref{L:LoopsofFMPisloops}, using the sequence of equivalences
\[
Lie(\Omega F) \, = \, \mathfrak{L}_{B_{fmp}\Omega F} \, \cong \, \mathfrak{L}
\]
and that $B_{fmp}\Omega$ is equivalent to the identity.

To conclude, since $G$ is a grouplike monoid object, the map $G\to G$ induces by multiplication by $x$ is an equivalence which proves the last statement.
\end{proof}

Moreover, by Lemma \ref{L:invarianceformalgroups}:
\begin{lem}\label{L:invariancetangentLie}
Let $f:F\rightarrow G$ be a morphism in $SPsh^{infcoh}_{\infty}((dgArt_{\mathbb{K}}^{aug})^{op})$. If $f$ is a weak equivalence, then for any $x\in F(\mathbb{K})$ it induces a weak equivalence of tangent homotopy Lie algebras
\[
Lie(\Omega_xF)\stackrel{\sim}{\rightarrow} Lie(\Omega_{f(x)}G).
\]
\end{lem}

\begin{example}\label{ex:OmegaSpaces}Easy examples of  derived prestack groups $G$ are given  by infinitesimally cohesive $\infty$-functors
\[
G:dgArt_{\mathbb{K}}^{aug}\rightarrow \Omega\textrm{-}Spaces
\]
where $\Omega\textrm{-}Spaces$ is the $\infty$-category $Mon_{E_1}^{gp}(Top)$ of grouplike $E_1$-monoids in spaces, i.e., group objects in topological spaces. 
 By May's recognition principle, the latter are (weakly) equivalent to loop spaces, hence the terminology. 
 Our examples of interests will take place in the $\infty$-category of grouplike simplicial monoids $sMon^{gl}$ as a model for $Mon_{E_1}^{gp}(Top)$ (i.e. we use 
 the equivalence between the model categories of topological spaces and simplicial sets and strictification to model $\Omega\textrm{-}Spaces$).
 As we explained, a derived prestack group $G$ gives rises to a family of derived formal groups parametrized by $G(\mathbb{K})$. 
 \end{example}

 In the next section we will focus on the formal neighourhood of the identity in homotopy automorphism groups, 
 and see how this formalism applies to homotopy automorphisms of algebras over props.

\subsection{Prestacks of algebras and derived groups of homotopy automorphisms}\label{SS:DKSpace}
We now define our second type of moduli of algebraic structures build on automorphisms of the structure.

\smallskip

First we recall that the self equivalences of an object in an $\infty$-category are canonically a group object in spaces (as in example~\ref{ex:OmegaSpaces}).
When the $\infty$-category comes from a model category, strict models for those self equivalences are given by simplicial monoids of homotopy automorphisms.
We refer the reader to \cite[Section 2.2]{Fre5} for a detailed account on simplicial monoids of homotopy automorphisms in model categories 
and to \cite{DK1,DK2,DK3} for the generalization to homotopy automorphisms in the simplicial localization of any relative category. 
\begin{defn}\label{D:haut} Let $X$ be a chain complex. Let $P$ be a prop, 
 $P_{\infty}$  a cofibrant resolution of $P$, and  $(X,\varphi: P_\infty\to End_X)$ be a $P_{\infty}$-algebra structure on $X$. 
\begin{itemize}\item We denote $\underline{haut}(X)$  
the derived prestack group of homotopy automorphisms of the underlying complex $X$ 
taken in the model category of chain complexes\footnote{Precisely we consider the projective model structure}. It is defined by
\[
dgArt_{\mathbb{K}}^{aug} \ni \,R \,\mapsto \, haut_{Mod_R}(X\otimes R),
\]
where $haut_{Mod_R}$ is the simplicial monoid of homotopy automorphisms in the category of $R$-modules.
\item 
We define $\underline{haut}_{P_{\infty}}(X,\varphi)$ to be the derived prestack group associated 
to the automorphisms of $(X,\varphi)$\footnote{that is, the  automorphisms or weak self-equivalences of $(X,\varphi)$ in this $\infty$-category}
in the $\infty$-category $P_{\infty}-Alg[W_{qiso}^{-1}]$: 
\[dgArt_{\mathbb{K}}^{aug} \ni \,R \,\mapsto \, \textrm{Iso}_{P_{\infty}-Alg(Mod_R) [W_{qiso}^{-1}]}\big(X\otimes R, X\otimes  R\big) \]
where, for any $\infty$-category $\mathcal{C}$, we write $\textrm{Iso}_{\mathcal{C}}$ for the space of maps in the underlying maximum $\infty$-groupoid of 
$\mathcal{C}$. 
\end{itemize}
\end{defn}
Note that, since $X$ is cofibrant (like any chain complex over a field) and $(-)\otimes R$ is a left Quillen functor, 
the homotopy automorphisms $\underline{haut}(X)$   above are exactly the self quasi-isomorphisms of $X\otimes R$. 
We prove in Theorem~\ref{T:hautderivedgroup} below that $\underline{haut}_{P_{\infty}}(X,\varphi)$ is indeed a derived prestack 
group.

Let us describe more precisely this derived group: 
consider the presheaf of $\infty$-categories 
over $cdga_{\mathbb{K}}^{op}$ defined by
\begin{eqnarray*}
\underline{P_{\infty}-Alg}:cdga_{\mathbb{K}} & \rightarrow & Cat_{\infty} \\
R & \longmapsto & P_{\infty}-Alg(Mod_R^{cof})[W_{qiso}^{-1}]
\end{eqnarray*}
where $Cat_{\infty}$ is the $\infty$-category of $\infty$-categories. Here $Mod_R^{cof}$ is the subcategory of cofibrant $R$-modules in the projective model structure. 
Let us take then the maximal sub-$\infty$-groupoid of 
$P_{\infty}-Alg(Mod_R^{cof})[W_{qiso}^{-1}]$ for each $R$, getting an $\infty$-groupoid valued presheaf. 
Then, the based loop space at a point $(X,\varphi)$ is exactly $\underline{haut}_{P_{\infty}}(X,\varphi)$. 
An explicit construction for this is given by, for any cdga $R$, the Dwyer-Kan simplicial loop groupoid \cite{DK4} of the quasi-category 
$P_{\infty}-Alg(Mod_R)[W_{qiso}^{-1}]$. Then the Kan complex of paths from $(X\otimes R,\varphi\otimes R)$ to itself in this simplicial loop groupoid is a model 
for  $\underline{haut}_{P_{\infty}}(X,\varphi)(R)$ (this is similar to example~\ref{ex:OmegaSpaces}).

\smallskip

 We now describe a \lq\lq{}point-set\rq\rq{} model for the construction of  those ($\infty$-categorical) 
derived groups of 
$P_\infty$-algebras automorphisms. 
First, we introduce a related and useful construction. 

\smallskip

\noindent
\textit{The presheaf of Dwyer-Kan classification spaces.} 
The assignment
\[
R\mapsto wP_{\infty}-Alg(Mod_R^{cof}),
\]
where the $w(-)$ stands for the subcategory of weak equivalences and $cof$ for cofibrant $R$-modules, defines a weak presheaf of categories in the sense of 
\cite[Definition I.56]{Ane}. It sends a morphism $A\rightarrow B$ to the functor $-\otimes_A B$. which is symmetric monoidal,
hence lifts at the level of $P_{\infty}$-algebras. This is not a strict presheaf, since the composition of morphisms $A\rightarrow B\rightarrow C$ 
is sent to the functor $(-)\otimes_AB\otimes_BC$, which is naturally isomorphic (but not equal) to $(-)\otimes_AC$. 
This weak presheaf can be strictified into a presheaf of categories. 
Applying the nerve functor to this  then defines an $\infty$-groupoid.
\begin{defn}\label{D:NwP} We denote
\begin{equation}\underline{\mathcal{N}wP_{\infty}-Alg}: R\mapsto \mathcal{N}wP_{\infty}-Alg(Mod_R^{cof}) \end{equation} for the \emph{simplicial presheaf of Dwyer-Kan classification spaces} given by the above construction, that is the coherent nerve of the presheaf of categories induced by the subcategory of weak-equivalences in $P_{\infty}-Alg$.

We also denote
\begin{equation}
 \underline{\mathcal{N}wCh_{\mathbb{K}}}: R\mapsto Mod_R^{cof}
\end{equation}
the simplicial presheaf  of quasi-coherent modules of \cite[Definition 1.3.7.1]{TV}.
\end{defn}
The loop space on $\underline{\mathcal{N}wP_{\infty}-Alg}$ based at a $P_{\infty}$-algebra $(X,\varphi)$ 
is then the strictification of the weak simplicial presheaf
\[
\Omega_{(X,\varphi)}\underline{\mathcal{N}wP_{\infty}-Alg}:R\mapsto \Omega_{(X\otimes R,\varphi\otimes R)}\mathcal{N}wP_{\infty}-Alg(Mod_R^{cof}).
\]
\begin{lem}
The pointwise loop space defined above is pointwise equivalent to the loop space functor in the projective model category of simplicial presheaves, 
where we consider simplicial presheaves with values in pointed simplicial sets. 
\end{lem}
\begin{proof}
The pointed loop space functor on the projective model category of simplicial presheaves $SPsh(\mathcal{C})$ on a model category $\mathcal{C}$ is defined on any simplicial presheaf $F$ as the homotopy pullback $pt\times_F^hpt$.
In the model category setting, a homotopy pullback is computed as the limit of a fibrant resolution of the pullback diagram in $SPsh(\mathcal{C})^I_{inj}$, where $I$ is the small category $\{\bullet\rightarrow\bullet\leftarrow\bullet\}$ and $inj$ means that we consider this diagram category equipped with the injective model structure. Moreover, we have a Quillen equivalence
\[
SPsh(\mathcal{C})^I_{proj}\leftrightarrows SPsh(\mathcal{C})^I_{inj}
\]
where $SPsh(\mathcal{C})^I_{proj}$ is the projective model category of $I$-diagrams and the adjunction is given by the identity functors. In particular, this implies that every fibrant resolution in $SPsh(\mathcal{C})^I_{inj}$ is a fibrant resolution in $SPsh(\mathcal{C})^I_{proj}$. In the projective model structure $SPsh(\mathcal{C})^I_{proj}$, fibrations are the same as in the projective model category of functors $Fun(\mathcal{C}\times I, sSet)_{proj}$. So a fibrant resolution in $SPsh(\mathcal{C})^I_{inj}$ is pointwise a fibrant resolution in $sSet^I_{inj}$. Moreover, limits in $SPsh(\mathcal{C})$ are determined pointwise. This implies that the pullback of a fibrant resolution of a pullback diagram in simplicial presheaves is given, pointwise, by the pullback of a fibrant resolution of a pullback diagram in simplicial sets. That is, the homotopy pullback defining the loop space functor for simplicial presheaves, when valued at a given object of $\mathcal{C}$, gives the homotopy pullback defining the loop space functor for pointed simplicial sets.
\end{proof}

\smallskip

\noindent
\textit{Homotopy automorphism presheaves as loops over the presheaf of Dwyer-Kan classification spaces.} 
In the case of an operad $O$, there is an easy model for $\underline{haut}_{O_{\infty}}$. Indeed, in that case, $O_{\infty}$-algebras inherits a canonical
model category structure and $\underline{haut}_{O_{\infty}}$
 is the $\infty$-functor associated to a simplicial presheaf given  by 
 the simplicial monoid of homotopy automorphisms of $(X,\varphi)$ in the model category of $O_{\infty}$-algebras. That is
  the simplicial sub-monoid of self weak equivalences in the usual homotopy mapping space $Map_{O_{\infty}-Alg}(X,X)$ 
  (see for instance \cite[Chapter 17]{Hir}). Thus
this weak simplicial presheaf is 
\[
R\longmapsto haut_{O_{\infty}}(X\otimes R,\varphi\otimes R)_{Mod_R}
\]
where $haut_{O_{\infty}}(X\otimes R,\varphi\otimes R)_{Mod_R}$ is the simplicial monoid of homotopy automorphisms
of $(X\otimes R,\varphi\otimes R)\in O_{\infty}-Alg(Mod_R^{cof})$. Note that by definition, this homotopy automorphism are computed by taking 
a cofibrant resolution of
$(X\otimes R,\varphi\otimes R)$ to get a cofibrant-fibrant object (all algebras are fibrant), and then considering weak self-equivalences of it.
Our simplicial presheaf is then its strictification (see \cite[Section I.2.3.1]{Ane}). 

In the case of a general prop $P$ (and for properads as well), 
there is no model category structure anymore on the category of $P_{\infty}$-algebras. However, 
we can still define the simplicial monoid $\underline{L^HwP_{\infty}-Alg}(X,\varphi)$ of homotopy automorphisms 
in the simplicial or hammock localization (with respect to quasi-isomorphisms) of $P_{\infty}$-algebras, following Dwyer-Kan \cite{DK2,DK3}.
Note that by \cite{DK3}, in the case when $P_{\infty}-Alg$ is a model category (that is, $P$ is an operad), we have a homotopy equivalence 
\[\underline{haut}_{P_{\infty}}(X,\varphi)\, \simeq \, \underline{L^HwP_{\infty}-Alg}(X,\varphi)\]
(taking the model category construction for the left side of this equivalence), so the two constructions agree.
In both cases, these are models of the pointed loop space $\Omega_{(X,\varphi)}\underline{\mathcal{N}wP_{\infty}-Alg}$
on the simplicial presheaf of Dwyer-Kan classification spaces:
\begin{lem}\label{L:hautasLoop}
 Let $P_\infty$ be a cofibrant prop.  Then  $\underline{haut}_{P_{\infty}}(X,\varphi)$ is equivalent to \[
R\in dgArt_{\mathbb{K}}^{aug}\longmapsto \Omega_{(X\otimes R,\varphi\otimes R)}\Big(\mathcal{N}wP_{\infty}-Alg(Mod_R^{cof})\Big).
\]
Further, $\underline{haut}(X)$ is equivalent to \[
R\in dgArt_{\mathbb{K}}^{aug}\longmapsto \Omega_{X\otimes R} \Big(\mathcal{N}wMod_{R}^{cof}\Big).
\]
\end{lem}
\begin{proof}
This comes from the fact that, for any relative category $(C,W)$
and any object $X$ of $C$, the connected component of $X$ in $\mathcal{N}W$ is equivalent to the classifying space $BLW(X,X)$. Therefore
there is an equivalence $LW(X,X)\simeq \Omega_X\mathcal{N}W$ of simplicial monoids.
Hence we can define the presheaf of homotopy automorphisms, or self-weak equivalences, $\underline{haut}_{P_{\infty}}(X,\varphi)$ is equivalent to the 
following simplicial presheaf
\[
\underline{haut}_{P_{\infty}}(X,\varphi):R\in dgArt_{\mathbb{K}}^{aug}\longmapsto \Omega_{(X\otimes R,\varphi\otimes R)}\mathcal{N}wP_{\infty}-Alg(Mod_R^{cof}).
\]
The proof for $\underline{haut}(X)$ is similar.
\end{proof}

\smallskip

\noindent
\textit{Prestacks of algebras.} We will now prove that what we called the derived  group of automorphisms of an algebra is indeed a derived prestack group. 
As a first step, we need the following version of Rezk's homotopy pullback theorem \cite{Rez} :
\begin{prop}\label{P: Rezkgen}
Let $P_{\infty}$ be a cofibrant prop and $X$ be a chain complex. The forgetful functor $P_{\infty}-Alg\rightarrow Ch_{\mathbb{K}}$ induces a homotopy fiber sequence
\[
\underline{P_{\infty}\{X\}}\rightarrow \underline{\mathcal{N}wP_{\infty}-Alg} \rightarrow \underline{\mathcal{N}wCh_{\mathbb{K}}}
\]
of simplicial presheaves over augmented Artinian cdgas (see~\ref{D:NwP} for the notations).
\end{prop}

\begin{proof}
We explain briefly how \cite[Theorem 0.1]{Yal3} can be transposed in the context of simplicial presheaves of cdgas. The identification of the homotopy fiber of the forgetful map
\[
\underline{\mathcal{N}wP_{\infty}-Alg}\rightarrow \underline{\mathcal{N}wCh_{\mathbb{K}}}
\]
with the simplicial presheaf $\underline{P_{\infty}\{X\}}$ follows from the two following facts. 
First, we can identify it pointwise, for any cdga $A$, with $Map(P_{\infty}\otimes A, End_{X\otimes A}^{Mod_A})$, 
where $End_{X\otimes A}^{Mod_A}$ is the endomorphism prop of $X\otimes A$ in the category of $A$-modules.
This comes from the extension of \cite[Theorem 0.1]{Yal3} to $A$-linear $P_{\infty}$-algebras, 
which holds true trivially by replacing chain complexes of $\mathbb{K}$-modules by dg $A$-modules 
as target category in the universal functorial constructions of \cite[Section 2.2]{Yal3} 
($A$-modules are equipped with exactly the same operations than chain complexes which are needed in this construction: 
directs sums, suspensions, twisting cochains). 

Second, for any morphism of cdgas $f:A\rightarrow B$, the tensor product $(-)\otimes_A B$ induces a morphism of simplicial sets
\[
Map(P_{\infty}\otimes A, End_{X\otimes A}^{Mod_A}) \rightarrow Map(P_{\infty}\otimes B, End_{X\otimes B}^{Mod_B})
\]
fitting in a commutative square
\[
\xymatrix{
Map(P_{\infty}\otimes A, End_{X\otimes A}^{Mod_A})\ar[r]^-{\cong}\ar[d]_-{(-)\otimes_AB} & \underline{P_{\infty}\{X\}}(A)\ar[d]^-f \\
Map(P_{\infty}\otimes B, End_{X\otimes B}^{Mod_B})\ar[r]^-{\cong} & \underline{P_{\infty}\{X\}}(B) }
\]
(see for instance \cite[Section 3]{Yal2}) so that we get a morphism of homotopy fiber sequences
\[
\xymatrix{
\underline{P_{\infty}\{X\}}(A)\ar[r]\ar[d] & \underline{\mathcal{N}wP_{\infty}-Alg}(Mod_A)\ar[r]\ar[d]_-{(-)\otimes_AB} & \mathcal{N}wMod_A\ar[d]_-{(-)\otimes_AB} \\
\underline{P_{\infty}\{X\}}(B)\ar[r] & \underline{\mathcal{N}wP_{\infty}-Alg}(Mod_B)\ar[r] & \mathcal{N}wMod_B .}
\]
The right commutative squares shows that we have a morphism of simplicial presheaves $\underline{\mathcal{N}wP_{\infty}-Alg} \rightarrow \underline{\mathcal{N}wCh_{\mathbb{K}}}$. In the projective model structure, a model for the homotopy fiber of this morphism is given by the strict fiber of a fibration resolving this morphism. Fibrations are determined componentwise and the model category of simplicial sets is right proper, so this is a homotopy fiber componentwise. So, by the argument of the first paragraph, we know the homotopy fiber is given componentwise by $\underline{P_{\infty}\{X\}}(A)$. The left vertical map above is then given by construction by the base change $(-)\otimes_AB$, so the homotopy fiber is indeed the simplicial presheaf $\underline{P_{\infty}\{X\}}$.
\end{proof}

\begin{thm}\label{T:hautderivedgroup}
The simplicial presheaf $\underline{haut}_{P_{\infty}}(X,\varphi)$ is a derived prestack group in the sense of Definition~\ref{D:derivedgroup}\footnote{that
is an object of $Mon_{E_1}^{gp}(SPsh^{infcoh}_{\infty}((dgArt_{\mathbb{K}}^{aug})^{op}))$}.

In particular $\widehat{\underline{haut}_{P_{\infty}}(X,\varphi)_{id}}$ is a derived formal group.
\end{thm}
\begin{proof}
First, recall that $\underline{haut}_{P_{\infty}}(X,\varphi)$ is equivalent to $\Omega_{(X,\varphi)}\underline{\mathcal{N}wP_{\infty}-Alg}$, and that we already know it is a presheaf with values in grouplike simplicial monoids, hence a group object in simplicial presheaves.
Second, we use the simplicial presheaf version of Rezk's pullback theorem \cite{Rez} for algebras over properads, that is, 
the  homotopy fiber sequence
\[
\underline{P_{\infty}\{X\}}\rightarrow \underline{\mathcal{N}wP_{\infty}-Alg} \rightarrow \underline{\mathcal{N}wCh_{\mathbb{K}}}
\]
of simplicial presheaves over augmented Artinian cdgas, taken over the base point $X$ given by Proposition~\ref{P: Rezkgen}.
Considering the associated long fibration sequence of (iterated) loop spaces, this homotopy fiber sequence induces in particular a homotopy fiber sequence
\[
\Omega_{(X,\varphi)}\underline{\mathcal{N}wP_{\infty}-Alg} \rightarrow \Omega_X\underline{\mathcal{N}wCh_{\mathbb{K}}}\rightarrow  \underline{P_{\infty}\{X\}}
\]
hence the fiber sequence 
\[
\underline{haut}_{P_{\infty}}(X,\varphi) \rightarrow \underline{haut}(X)\rightarrow  \underline{P_{\infty}\{X\}}.
\] (by Lemma \ref{L:hautasLoop}).
Now we combine this result with Lemma~\ref{L:infcohppties}(1) to deduce that $\underline{haut}_{P_{\infty}}(X,\varphi)$ preserves weak equivalences 
and is infinitesimally cohesive. 
For this, we just have to check that the two right-hand terms of the fiber sequence satisfy these properties 
and use that this homotopy fiber sequence is in particular a pointwise homotopy fiber sequence. 
Concerning $\underline{haut}(X)$ this is already known from see \cite[Section 5.2]{Lur0}, 
and concerning $\underline{P_{\infty}\{X\}}$ this follows from its isomorphism with the Maurer-Cartan simplicial presheaf in Theorem~\ref{T:Yal2}.

In particular, $\widehat{\underline{haut}_{P_{\infty}}(X,\varphi)_{id}}$ is a derived formal group. 
Note that we could have directly proved this last statement by taking the formal completion of the fiber sequence above 
(that is, the componentwise homotopy fiber of this diagram over each appropriate base point), and then apply Lemma~\ref{L:infcohppties}(1) 
to $\widehat{\underline{haut}(X)}_{id}$ (which is a formal derived group) and $\widehat{\underline{P_{\infty}\{X\}}^{\varphi}}$ 
(which is a derived formal moduli problem).
\end{proof}
\begin{rem}\label{R:homotopyaction}
The classification space $\underline{\mathcal{N}wP_{\infty}-Alg}$ decomposes as a coproduct of the classifying spaces of homotopy automorphisms of $P_{\infty}$-algebras
\[
\underline{\mathcal{N}wP_{\infty}-Alg} \cong \coprod_{[Y,\phi]\in\pi_0\mathcal{N}wP_{\infty}-Alg}B\underline{haut}_{P_{\infty}-Alg}(Y,\phi)
\]
where $[Y,\phi]$ ranges over quasi-isomorphism classes of $P_{\infty}$-algebras. Restricting the homotopy pullback of Proposition~\ref{P: Rezkgen} to the connected component $B\underline{haut}(X)$ of the base space, we get a homotopy pullback
\[
\xymatrix{
\underline{P_{\infty}\{X\}}\ar[d]\ar[r] & \coprod_{[Y,\phi],Y\simeq X}B\underline{haut}_{P_{\infty}-Alg}(Y,\phi) \ar[d] \\
pt\ar@{^{(}->}[r] & B\underline{haut}(X)
}
\]
where the coproduct $\coprod_{[Y,\phi],Y\simeq X}$ ranges over $P_{\infty}$-algebras so that $Y\simeq X$ as complexes. So Rezk homotopy pullback theorem and its version above tells us that $\coprod_{[Y,\phi],Y\simeq X}B\underline{haut}_{P_{\infty}-Alg}(Y,\phi)$ can be seen as a homotopy quotient of $\underline{P_{\infty}\{X\}}$ by the action of $\underline{haut}(X)$. From a deformation theoretic perspective, this means that at a ``tangent level'', the deformation theory of $\varphi:P_{\infty}\rightarrow End_X$ corresponds to deformations of the $P_{\infty}$-algebra $(X,\varphi)$ which preserves the differential of the underlying complex $X$, whereas the deformations associated to $\underline{haut}_{P_{\infty}-Alg}(X,\varphi)$ deform the differential as well (that is, it takes into account the action of $\underline{haut}(X)$ on $\underline{P_{\infty}\{X\}}$). We are going to see in Section 3 how to formalize properly this idea.
\end{rem}
\begin{rem}\label{R:proximatemoduli}
Let us explain the relationship between the classifying presheaf of algebras and the derived formal group of homotopy automorphisms in the neighbourhood of the identity. Recall the construction
\[
\underline{haut}_{P_{\infty}}(X,\varphi):R\in dgArt_{\mathbb{K}}^{aug}\longmapsto \Omega_{(X\otimes R,\varphi\otimes R)}\mathcal{N}wP_{\infty}-Alg(Mod_R^{cof}),
\]
from which we deduce
\[
\widehat{\underline{haut}_{P_{\infty}}(X,\varphi)_{id}} = \Omega_*\widehat{\underline{\mathcal{N}wP_{\infty}-Alg}}_{(X,\varphi)}
\]
where $\Omega_*$ is the loop space for pointed functors as explained in Section~\ref{S:formalgeom}. Using the decomposition of the nerve of weak equivalences into classifying spaces of homotopy automorphisms pointed out in Remark~\ref{R:homotopyaction}, we see moreover that for any augmented Artinian cdga $R$, there is a decomposition
\[
\widehat{\underline{\mathcal{N}wP_{\infty}-Alg}}_{(X,\varphi)}(R) \cong \coprod_{[Y,\phi]|(Y,\phi)\otimes_R\mathbb{K}\simeq (X,\varphi)}Bhaut_{P_{\infty}-Alg(Mod_R)}(Y,\phi).
\]
Equivalently, $\widehat{\underline{\mathcal{N}wP_{\infty}-Alg}}_{(X,\varphi)}(R)$ is homotopy equivalent to the maximal $\infty$-subgroupoid of the $\infty$-category $P_{\infty}-Alg(Mod_R)[W_{qiso}^{-1}]$ generated by $R$-linear $P_{\infty}$-algebras $(Y,\phi)$ such that $(Y,\phi)\otimes_R\mathbb{K}\simeq (X,\varphi)$, that is
\[
\widehat{\underline{\mathcal{N}wP_{\infty}-Alg}}_{(X,\varphi)}(R) \cong P_{\infty}-Alg(Mod_R)[W_{qiso}^{-1}]\times^h_{P_{\infty}-Alg(Ch_{\mathbb{K}})[W_{qiso}^{-1}]}\{(X,\varphi)\}.
\]

The space $\widehat{\underline{\mathcal{N}wP_{\infty}-Alg}}_{(X,\varphi)}$ encapsulates the whole deformation theory of $(X,\varphi)$ in the $\infty$-category $P_{\infty}-Alg[W_{qiso}^{-1}]$ as we can think of it, that is, an $R$-deformation of $(X,\varphi)$ is an $R$-linear $P_{\infty}$-algebra whose restriction modulo $R$ is quasi-isomorphic to $(X,\varphi)$, and equivalences between $R$-deformations are defined by compatible $R$-linear quasi-isomorphisms whose restriction modulo $R$ is homotopic to $Id_{(X,\varphi)}$. This is the natural generalization, to the differential graded setting, of classical deformations of degree zero algebras. Although it is not clear that such a construction provides a derived formal moduli problem, one can however associates to it the derived formal group 
$\underline{haut}_{P_{\infty}}(X,\varphi)$ via a loop space construction, and by the general formalism explained in Section~\ref{S:formalgeom} we have
\[
\mathbb{T}_{\widehat{\underline{haut}_{P_{\infty}}(X,\varphi)_{id}}} = Lie(L(\widehat{\underline{\mathcal{N}wP_{\infty}-Alg}}_{(X,\varphi)}))
\]
where $L$ is the completion of $\widehat{\underline{\mathcal{N}wP_{\infty}-Alg}}_{(X,\varphi)}$ in a formal moduli problem.
Another way to state this is that in general $\widehat{\underline{\mathcal{N}wP_{\infty}-Alg}}_{(X,\varphi)}$ is $1$-proximate in the sense of \cite{Lur0}.
\end{rem}
\begin{rem}\label{R:VectStack}
In the special case of operads acting on algebras concentrated in degree $0$, we can say more.
Let $A$ be a $P$-algebra in vector spaces whose underlying vector space is of finite dimension, then by \cite[Prop.2.2.6.8]{TV},
the classifying presheaf $\underline{\mathcal{N}wP-Alg}$ is actually a derived $1$-geometric stack, 
which implies by  \cite[Corollary 6.5]{Lur01} and \cite[Lemma 2.1.7]{Lur02} that its restriction to $dgArt_{\mathbb{K}}^{aug}$ is infinitesimally cohesive.
Consequently $\widehat{\underline{\mathcal{N}wP-Alg}}_A$ is already a derived formal moduli problem in this case 
and \[Lie \big({\widehat{\underline{haut}_P(A)_{id}}}\big) \, \cong \, Lie(\widehat{\underline{\mathcal{N}wP-Alg}}_A)\] by Proposition~\ref{P:LieofGroup}.

In the special case where $P$ is a non-positively graded dg operad and $A$ a non-positively graded dg algebra, this classifying presheaf is not known to be a derived geometric stack, nevertheless it is homotopy equivalent to the nerve of the tangent complex of $A$ according to \cite[Theorem 2.3.4]{Hin3}.
\end{rem}
This Lie algebra recovers in particular various known deformation complexes in the litterature, once one has an explicit formula to compute it, 
as we are going to detail in Section~\ref{S:Examples}.

\subsection{The fiber sequence of deformation theories}
We now relate precisely the two moduli problems of algebraic structures, that is those  governed by the mapping space $\underline{P_{\infty}\{X\}}$
and the homotopy automorphisms space $\underline{haut_{P_{\infty}-Alg}}(X,\varphi)$ (Definitions~\ref{D:haut} and~\ref{D:spresheafofAlgstructure}).
\begin{thm}
\label{P: hofibgroups}
There is a homotopy fiber sequence of derived prestack groups
\[
\Omega_{\varphi}\underline{P_{\infty}\{X\}}\rightarrow \underline{haut_{P_{\infty}-Alg}}  (X,\varphi) \rightarrow \underline{haut}(X),
\]
hence a homotopy fiber sequence of derived formal groups
\[
\widehat{\Omega_{\varphi}\underline{P_{\infty}\{X\}}}\rightarrow \widehat{\underline{haut_{P_{\infty}-Alg}}}(X,\varphi)_{id} \rightarrow \widehat{\underline{haut}(X)_{id}},
\]
and equivalently of their associated $L_{\infty}$-algebras
\[
g_{P,X}^{\varphi}\rightarrow Lie(\underline{haut_{P_{\infty}-Alg}}(X,\varphi))\rightarrow Lie(\underline{haut}(X)).
\]
\end{thm}
\begin{proof}
Recall (see~\eqref{eq:defLoopofFMP}) that the pointed loop space functor is defined 
on any simplicial presheaf $F$ as the homotopy pullback $pt\times_F^hpt$. It thus commutes with homotopy fibers, 
and in particular the loop space $\infty$-functor commutes with fibers in the $\infty$-category of simplicial presheaves.
In the fiber sequence of Proposition~\ref{P: Rezkgen}, we choose $\varphi$ as the base point on the left,  $(X,\varphi)$ in the middle,
and $X$ on the right.
Since fibers in the $\infty$-category of presheaves valued in simplicial monoids 
are determined in the underlying $\infty$-category of simplicial presheaves,  
applying the pointed loop space $\infty$-functor with respect to these base points, we deduce a fiber sequence of derived prestack groups
\[
\Omega_{\varphi}\underline{P_{\infty}\{X\}}\rightarrow \underline{haut_{P_{\infty}-Alg}}(X,\varphi) \rightarrow \underline{haut}(X)
\]
(using that $\Omega\circ B\simeq Id$). Hence, we get 
 a fiber sequence of derived formal groups
\[
\widehat{\Omega_{\varphi}\underline{P_{\infty}\{X\}}}\rightarrow \widehat{\underline{haut_{P_{\infty}-Alg}}(X,\varphi)_{id}} \rightarrow \widehat{\underline{haut}(X)_{id}}
\]
(using that $\Omega\widehat{F_x}\simeq \widehat{\Omega_xF}$ for an infinitesimally cohesive $\infty$-functor $F$ and $x\in F(\mathbb{K})$, 
and in particular that $\widehat{G_e}\simeq \Omega\widehat{BG}$ for a derived prestack group $G$). The corresponding fiber sequence of Lie algebras aasociated 
to this formal derived problems
identifies with the desired one 
\[
g_{P,X}^{\varphi}\rightarrow Lie(\underline{haut_{P_{\infty}-Alg}}(X,\varphi))\rightarrow Lie(\underline{haut}(X))
\]by Lemma~\ref{L:hautasLoop} and using equivalence~\eqref{eq:TofOmega}
\end{proof}

\subsection{Equivalent deformation theories for equivalent (pre)stacks of algebras}

In derived algebraic geometry,
an equivalence between two derived Artin stacks $F$ and $G$ induces a weak equivalence between the tangent complex over a given point of $F$  and the tangent complex over its image in $G$ \cite{TV}. 
We now prove similar statement about the tangent Lie algebras of our formal moduli problems of algebraic structures.

Recall the presheaf of categories given by the
 $\infty$-functor
\begin{eqnarray*}
\underline{P_{\infty}-Alg}:cdga_{\mathbb{K}} & \rightarrow & Cat_{\infty} \\
R & \longmapsto & P_{\infty}-Alg(Mod_R)[W_{qiso}^{-1}]
\end{eqnarray*}
where $Cat_{\infty}$ is the $\infty$-category of $\infty$-categories. 

The idea is that \cite[Theorem 0.1]{Yal3} implies that the formal moduli problem $\underline{P_{\infty}\{X\}}^{\varphi}$ is \lq\lq{}tangent\rq\rq{} 
over $(X,\varphi)$ to the Dwyer-Kan classification space of the $\infty$-category of $P_{\infty}$-algebras (see~\ref{SS:DKSpace}). 

\smallskip

More precisely, recall that $F:\underline{P_{\infty}-Alg}\stackrel{\sim}{\longrightarrow}\underline{Q_{\infty}-Alg} $ being an an equivalence 
of presheaves of $\infty$-categories means that, for every augmented Artinian cdga $R$, 
\[
F(R):P_{\infty}-Alg(Mod_R)[W_{qiso}^{-1}]\stackrel{\sim}{\longrightarrow} Q_{\infty}-Alg(Mod_R)[W_{qiso}^{-1}].
\]
is an equivalence of $\infty$-categories. 
Relying on our previous results, we prove:
\begin{thm}\label{T:equivhautLie}
Let $F$ be an equivalence of presheaves of $\infty$-categories
\[
F:\underline{P_{\infty}-Alg}\stackrel{\sim}{\longrightarrow}\underline{Q_{\infty}-Alg}.
\]
Then $F$ induces an equivalence of derived formal moduli problems
\[
B_{fmp}\widehat{\underline{haut}_{P_{\infty}-Alg}(X,\varphi)}_{Id_{(X,\varphi)}}\stackrel{\sim}{\rightarrow} B_{fmp}\widehat{\underline{haut}_{Q_{\infty}-Alg}(F(X,\varphi))}_{Id_{(X,\varphi)}},
\]
equivalently an equivalence of the associated $L_{\infty}$-algebras
\[
Lie(\underline{haut}_{P_{\infty}-Alg}(X,\varphi))\stackrel{\sim}{\rightarrow}Lie(\underline{haut}_{Q_{\infty}-Alg}(F(X,\varphi))).
\]
\end{thm}
\begin{proof}
The equivalence $F$ induces an equivalence of loop spaces
\[
\Omega_{(X,\varphi)}\underline{P_{\infty}-Alg} \stackrel{\sim}{\rightarrow} \Omega_{F(X,\varphi)}\underline{Q_{\infty}-Alg}
\]
hence an equivalence
\[
\underline{haut}_{P_{\infty}-Alg}(X,\varphi)\rightarrow \underline{haut}_{Q_{\infty}-Alg}(F(X,\varphi)).
\]
These presheaves are derived prestack groups by Theorem \ref{T:hautderivedgroup}, so this induces in turn an equivalence between their formal completions at the identity and equivalently between the $L_{\infty}$-algebras corresponding to these derived formal groups.
\end{proof}

In the case where $F$ commutes with the maps induced by forgetful functors, we can insert this equivalence in an equivalence of fiber sequences:
\begin{thm}\label{T:equivfiberseq}
Let $F$ be an equivalence of presheaves of $\infty$-categories
\[
F:\underline{P_{\infty}-Alg}\stackrel{\sim}{\longrightarrow}\underline{Q_{\infty}-Alg}.
\]
Let us suppose that $F$ commutes with the maps induced by the forgetful functors of these $\infty$-categories of algebras.
Then $F$ induces an equivalence of fiber sequences of derived formal moduli problems
\[
\xymatrix{
\underline{P_{\infty}\{X\}}^{\varphi}\ar[d]^-{\sim}\ar[r] & B_{fmp}\widehat{\underline{haut}_{P_{\infty}-Alg}(X,\varphi)}_{Id_{(X,\varphi)}}\ar[d]^-{\sim}\ar[r] & B_{fmp}\widehat{\underline{haut}(X)}_{Id_X}\ar[d]^-=\\
\underline{Q_{\infty}\{F(X)\}}^{F(\varphi)}\ar[r]  & B_{fmp}\widehat{\underline{haut}_{Q_{\infty}-Alg}(F(X,\varphi))}_{Id_{(X,\varphi)}}\ar[r] & B_{fmp}\widehat{\underline{haut}(X)}_{Id_X}
}
\]
where $F(\varphi)$ is the $Q_{\infty}$-algebra structure on the image of  $(X,\varphi)$  under $F$ (and $B_{fmp}$ is given by~\ref{eq:defBfmp}). Equivalently,
$F$ induces an equivalence of fiber sequences of the associated $L_{\infty}$-algebras
\[
\xymatrix{
g_{P,X}^{\varphi}\ar[d]^-{\sim}\ar[r] & Lie(\underline{haut}_{P_{\infty}-Alg}(X,\varphi)) \ar[d]^-{\sim}\ar[r] & Lie(\underline{haut}(X))\ar[d]^-=\\
g_{Q,F(X)}^{F(\varphi)}\ar[r]  & Lie(\underline{haut}_{Q_{\infty}-Alg}(F(X,\varphi)))\ar[r] & Lie(\underline{haut}(X))
}.
\]
\end{thm}
\begin{proof}
Let $F:\underline{P_{\infty}-Alg}\rightarrow \underline{Q_{\infty}-Alg}$ be an equivalence of presheaves of $\infty$-categories.
We have a commutative triangle
\[
\xymatrix{
\underline{P_{\infty}-Alg} \ar[rr]^-{F}\ar[dr]_-{U} & & \underline{Q_{\infty}-Alg} \ar[dl]_-{U} \\
 & \underline{Ch_{\mathbb{K}}} & }.
\]
Applying the loop space functor~\eqref{eq:defLoopofFMP} at the appropriate base points we get the commutative triangle
\[
\xymatrix{
\Omega_{(X,\varphi)}\underline{P_{\infty}-Alg} \ar[rr]^-{\sim}\ar[dr] & & \Omega_{F(X,\varphi)}\underline{Q_{\infty}-Alg} \ar[dl] \\
 & \Omega_X\underline{Ch_{\mathbb{K}}} & }.
\]
But a based loop space at a point of an $\infty$-category is the homotopy automorphims grouplike monoid of this point, so that this triangle is actually the triangle of derived prestack groups
\[
\xymatrix{
\underline{haut}_{P_{\infty}-Alg}(X,\varphi) \ar[rr]^-{\sim}\ar[dr] & & \underline{haut}_{Q_{\infty}-Alg}(F(X,\varphi)) \ar[dl] \\
 & \underline{haut}(X) & }.
\]
By Theorem~\ref{P: hofibgroups}, we get the equivalence of homotopy fiber sequences of derived prestack groups
\[
\xymatrix{
\Omega_{\varphi}\underline{P_{\infty}\{X\}} \ar[d]\ar[rr]^-{\sim} & & \Omega_{F(\varphi)}\underline{Q_{\infty}\{F(X)\}}\ar[d] \\
\underline{haut}_{P_{\infty}-Alg}(X,\varphi) \ar[rr]^-{\sim}\ar[dr] & & \underline{haut}_{Q_{\infty}-Alg}(F(X,\varphi)) \ar[dl] \\
 & \underline{haut}(X) & }.
\]
hence an equivalence of homotopy fiber sequences of the corresponding derived formal groups obtained by completion at the appropriate base points. This equivalence of fiber sequences gives an equivalence of fiber sequences of the corresponding Lie algebras by the Lurie-Pridham equivalence theorem.
\end{proof}
\begin{rem}
There is also a \lq\lq{}strict\rq\rq{} version of this theorem. Let us consider a morphism of weak presheaves of relative categories, that is, given for each cdga $A$ by a morphism of relative categories
\[
F(A):(P_{\infty}-Alg(Mod_A),W_{qiso})\rightarrow (Q_{\infty}-Alg(Mod_A),W_{qiso}).
\]
Let us suppose that $F$ induces an equivalence of presheaves of classification spaces
\[
\mathcal{N}wF:\underline{\mathcal{N}wP_{\infty}-Alg}\stackrel{\sim}{\longrightarrow}\underline{\mathcal{N}wQ_{\infty}-Alg}.
\]
By \cite[Section 3.3]{Yal}, this means that $F$ induces an equivalence of weak presheaves of $\infty$-categories as in Theorem~\ref{T:equivfiberseq}. 
Then, we can mimick the proof of  Theorem~\ref{T:equivfiberseq} as follows: 
we replace the presheaves of $\infty$-categories by these presheaves of classification spaces, take based loop spaces which gives back the homotopy automorphisms as well, and apply Theorem~\ref{P: hofibgroups}. 
\end{rem}
\begin{rem}
The commutation with the forgetful functors in Theorem \ref{T:equivfiberseq} is what happens for example when one takes for $F$ the restriction functor associated to a weak equivalence of cofibrant props $P_{\infty}\stackrel{\sim}{\rightarrow}Q_{\infty}$.
\end{rem}
\begin{rem}
In the general case, when $F$ is any equivalence of presheaves of $\infty$-categories, then the argument line of the proof of Theorem \ref{T:equivfiberseq} cannot hold true anymore. A priori, we only have two fiber sequences
\[
\underline{P_{\infty}\{X\}}^{\varphi}\rightarrow B_{fmp}\widehat{\underline{haut}_{P_{\infty}-Alg}(X,\varphi)}_{Id_{(X,\varphi)}}\rightarrow B_{fmp}\widehat{\underline{haut}(X)}_{Id_X}
\]
and
\[
\underline{Q_{\infty}\{F(X)\}}^{F(\varphi)}\rightarrow B_{fmp}\widehat{\underline{haut}_{Q_{\infty}-Alg}(F(X,\varphi))}_{Id_{(X,\varphi)}}\rightarrow B_{fmp}\widehat{\underline{haut}(F(X))}_{Id_{F(X)}}
\]
whose middle terms are weakly equivalent by Theorem \ref{T:equivhautLie}. However, the relation between $B_{fmp}\widehat{\underline{haut}(X)}_{Id_X}$ and $B_{fmp}\widehat{\underline{haut}(F(X))}_{Id_{F(X)}}$ requires a more careful analysis depending on $F$.
\end{rem}

\section{The tangent Lie algebra of homotopy automorphims}\label{S:TgtLie} 

The goal of this section and the next one is to make Theorem~\ref{P: hofibgroups} more explicit.

In Section \ref{S:TgtLie}, we prove that $Lie(\underline{haut}_{P_{\infty}-Alg}(X,\varphi))$ is the tangent Lie algebra of a homotopy quotient induced by the homotopy action of $haut(X)$ on $P_{\infty}\{X\}$ that we mentionned in Remark~\ref{R:homotopyaction}. This tangent Lie algebras fits in a fiber sequence whose later term is nothing but $End(X)=Hom_{Ch_{\mathbb{K}}}(X,X)$ equipped with the commutator of the composition product. In a few words, the tangent Lie algebra of homotopy automorphisms takes into account the action of the automorphisms of the complex $X$ on the Maurer-Cartan elements of $g_{P,X}^{\varphi}$, that is, on the space of $P_{\infty}$-algebra structures on $X$. 

In Section~\ref{S:Plus}, we further provide an explicit description in properadic terms of that homotopy Lie algebra in terms of a certain plus construction $P_{\infty}^+$, which parametrizes twistings of the differential of the complex $X$ compatible with the $P_{\infty}$-algebra structures on $X$. This is an explicit combinatorial way of handling simultaneously compatible deformations of the $P_{\infty}$-algebra structure and the differential, that can be described moreover as a twisted semi-direct product of the two extremal terms of the aforementioned fiber sequence.

More precisely, we are going to construct the diagram of equivalence of fiber sequences below
\[
\xymatrix{
g_{P,X}^{\varphi}\ar[r]^-{\sim}\ar[d] & g_{P,X}^{\varphi}\ar[r]^-{\sim}\ar[d] & g_{P,X}^{\varphi}\ar[d] \\
g_{P^+,X}^{\varphi^+}\ar[r]^-{\sim}\ar[d] & Lie(\Omega_{[\varphi]}(\underline{P_{\infty}\{X\}}//\underline{haut}(X)))\ar[r]^-{\sim}\ar[d] & Lie(\underline{haut}_{P_{\infty}}(X,\varphi))\ar[d] \\
End(X)\ar[r]^-{\sim} & End(X)\ar[r]^-{\sim} & End(X)
}
\]
where the middle term is defined in the next section~\ref{SS:ifntyaction}.
It turns out that this is the tangent incarnation of the non trivial action of $\underline{haut}(X)$ on the moduli space $P_{\infty}\{X\}$ at a topological level. Taking this action into account in the deformation theory of $(X,\varphi)$ means, on the one hand deforming the $P_{\infty}$-algebra structure with compatible deformations of the differential (equivalence of the middle fiber sequence with the left one), on the other hand deforming $(X,\varphi)$ \emph{in the $\infty$-category of $P_{\infty}$-algebras} (equivalence of the middle fiber sequence with the right one).

The plan is as follows. First, we construct two equivalences of fiber sequences of derived groups fitting in the diagram
\[
\xymatrix{
\Omega_{\varphi}\underline{P_{\infty}\{X\}}\ar[r]^-{\sim}\ar[d] & \Omega_{\varphi}\underline{P_{\infty}\{X\}}\ar[r]^-{\sim}\ar[d] & \Omega_{\varphi}\underline{P_{\infty}\{X\}}\ar[d] \\
\Omega_{\varphi^+}\underline{P^+_{\infty}\{X\}}\ar[r]^-{\sim}\ar[d] & \Omega_{[\varphi]}(\underline{P_{\infty}\{X\}}//\underline{haut}(X))\ar[r]^-{\sim}\ar[d] & \underline{haut}_{P_{\infty}}(X,\varphi))\ar[d] \\
\underline{haut}(X)\ar[r]^-{\sim} & \underline{haut}(X)\ar[r]^-{\sim} & \underline{haut}(X)
},
\]
where $\underline{P_{\infty}\{X\}}//\underline{haut}(X)$ is the appropriate homotopy quotient in the $\infty$-category of infinitesimally cohesive $\infty$-presheaves over $(dgArt_{\mathbb{K}}^{aug})^{op}$.
Secondly, the desired fiber sequence of $L_{\infty}$-algebras is induced by this diagram (taking, as usual in this paper, completions at identities to get equivalences of fiber sequences of derived formal groups).

\subsection{$\infty$-actions in infinitesimally cohesive presheaves}
\label{SS:ifntyaction}
In this section, the ambient $\infty$-category  is 
$SPsh^{infcoh}_{\infty}((dgArt_{\mathbb{K}}^{aug})^{op})$ and derived prestack groups are precisely the group objects (see \ref{D:derivedgpobkects}) in it.
This is a particular case of infinitesimally cohesive $\infty$-topos, where the theory of principal $\infty$-bundles developped in \cite{NSS,NSS2} 
fully applies.
In this setting, the general notion of $\infty$-action of a group object $G$ in an $\infty$-category on another object $X$ 
provides a homotopy quotient\footnote{which is the same as the quotient in $\infty$-stacks} $X//G$.
This homotopy quotient comes naturally equipped with a homotopy fiber sequence
\[
X\rightarrow X//G\rightarrow BG
\]
with  fiber $X$  associated to the universal $G$-principal $\infty$-bundle $\bullet\rightarrow BG$. The map $X//G\rightarrow BG$ 
is the classifying morphism of the action of $G$ on $X$. This is an analogue of the usual quotient stack by a group stack action of~\cite{GiNo}.
\begin{rem} When $G$ is presented by a simplicial presheaf in grouplike simplicial monoids (which is a model for derived prestack groups), 
the homotopy quotient  $X//G$ is computed by 
  the geometric realization of the simplicial action groupoid
\[
\cdots G\times G\times X\triplerightarrow G\times X\rightrightarrows X.
\]
(see for example \cite{Jar} in the case of group actions in simplicial presheaves).
\end{rem}

Let $G$ be a group object in $SPsh^{infcoh}_{\infty}((dgArt_{\mathbb{K}}^{aug})^{op})$, i.e., a derived prestack group.
\begin{prop}\label{P:fiberseqofquotientstack}
For any $X$ equipped with an $\infty$-action of a derived prestack group object $G$, there is a fiber sequence of homotopy Lie algebras
\[
Lie({\Omega_*X})\rightarrow Lie({\Omega_*(X//G)})\rightarrow Lie(G)
\]
\end{prop}
\begin{proof}
Since the loop space is an homotopy pullback, the fiber sequence
\[
X\rightarrow X//G\rightarrow BG
\]
yields a fiber sequence of derived groups
\[
\Omega_*X\rightarrow \Omega_*(X//G)\rightarrow G
\]
hence the desired fiber sequence of tangent Lie algebras by~\ref{D:TgtofGpPreStack}.
\end{proof}
We still consider an object $X$ with an $\infty$-action of a derived prestack group  $G$. Recall the completion of a derived prestack group~\ref{L:formalofderivedGp}.
\begin{lem}\label{L:splityieldssemidirect} 
 Assume that there exists a section of the (induced) projection map $\widehat{\Omega_*(X//G)}_{x}\stackrel{\pi}\rightarrow \widehat{G}_{1}$, 
 that is a  derived formal group morphism 
 $s:\widehat{G}_1\to \widehat{\Omega_* (X//G)}_{x}$   such that $\pi \circ s$ is equivalent to the identity. Then 
 there is an equivalence of $L_\infty$ algebras 
 \[Lie\big(\Omega_*(X//G)\big) \; \cong \; Lie(\Omega_* X) \rtimes  Lie(G) \]
 and the fiber sequence of proposition~\ref{P:fiberseqofquotientstack} identifies with the semi-direct product one.
\end{lem}
\begin{proof}
 The Lie algebra functor (\ref{D:TgtofGpPreStack}) depends only on the associated formal group at the base point. Therefore it gives a $L_\infty$-morphim $Lie(G)\stackrel{Lie(s)}\longrightarrow Lie\big(\Omega_*(X//G)\big)$. Composing with the adjoint action of the latter,
 we obtain a morphism $ad\circ Lie(s):Lie(G)\to Der\Big(Lie\big(\Omega_*(X//G)\big)\Big)$. 
 Since this is a map of Lie algebras, and $s$ is a section of $\pi$,
 the induced action of $Lie(G)$ on $Lie\big(\Omega_*(X//G)\big)$ restricts to $ker(\pi)\cong Lie(\Omega_* X)$. Therefore 
 we get an induced $L_\infty$-algebra map $Lie(G)\to Der\big(Lie(\Omega_* X)\big)$
 which defines the homotopy Lie algebra semi-direct product (\ref{ex:semidirectofLinfty}).
It follows that the morphism \[Lie(\Omega_* X) \rtimes  Lie(G) \ni (x,y)\,\stackrel{\tau}\mapsto \, x+s(y) \in Lie\big(\Omega_*(X//G)\big) \] 
is a $L_\infty$-algebra map
and that we have a commutative diagram of fiber sequences
\[\xymatrix{Lie(\Omega_* X)  \ar[r] \ar@{=}[d] &  Lie(\Omega_* X) \rtimes  Lie(G) \ar[r]^{\quad\qquad Lie(\pi)} \ar[d]^{\tau}& Lie(G) \ar@{=}[d]\\
Lie(\Omega_* X)  \ar[r]  &  Lie\big(\Omega_*(X//G)\big) \ar[r]^{\qquad Lie(\pi)} & Lie(G) } \]
 of  $L_\infty$-algebras. The equivalence now follows from the 2 out 3 property.
\end{proof}

\subsection{The Lie algebra of homotopy automorphisms from a derived group action}

The goal of this section is to prove (and make sense of) the formula below:
\[
Lie(\underline{haut}_{P_{\infty}-Alg}(X,\varphi)) \simeq Lie(\Omega_{[\varphi]}(\underline{P_{\infty}\{X\}}//\underline{haut}(X))).
\]
To do so, we interpret $Lie(\underline{haut}_{P_{\infty}-Alg}(X,\varphi))$ as the tangent Lie algebra of a homotopy quotient 
of $\underline{P_{\infty}\{X\}}$ by the $\infty$-action of $\underline{haut}(X)$. 

An explicit model of the homotopy quotient is given by a homotopy version of the well known Borel construction,
suitably adapted for simplicial presheaves over the opposite category of cdgas.
In \cite{Jar}, the Borel construction is given by the classifying space of the translation groupoid associated to the action of a sheaf of groups $G$
on a sheaf $X$, that is $EG\times_GX$. We adapt this construction to the case of an $\infty$-action.

Let $P_{\infty}$ a cofibrant prop. Let $\mathcal{N} wCh_{\mathbb{K}}^{\Delta[-]\otimes P_{\infty}}$ be
the bisimplicial set defined by $(\mathcal{N} wCh_{\mathbb{K}}^{\Delta[-]\otimes P_{\infty}})_{m,n}=(\mathcal{N} wCh_{\mathbb{K}}^{\Delta[n]\otimes P_{\infty}})_{m}$,
where the $w$ denotes the subcategory of morphisms which are quasi-isomorphisms in $Ch_{\mathbb{K}}$ and $\Delta[-]\otimes P_{\infty}$ is a cosimplicial resolution of $P_{\infty}$.
We recall from \cite{Yal3} the diagram
\[
\xymatrix{P_{\infty}\{X\}\ar[r]\ar[d] & diag\mathcal{N} fwCh_{\mathbb{K}}^{\Delta[-]\otimes P_{\infty}}\ar[d]\ar[r]^{\sim} & diag\mathcal{N} wCh_{\mathbb{K}}^{\Delta[-]\otimes P_{\infty}} & \mathcal{N} wCh_{\mathbb{K}}^{P_{\infty}}\ar[l]_-{\sim}\ar[d]\\
pt\ar[r] & \mathcal{N}(fwCh_{\mathbb{K}})\ar[rr]^{\sim} &  & \mathcal{N}(wCh_{\mathbb{K}})
},
\]
where the $fw$ denotes the subcategory of morphisms which are acyclic fibrations in $Ch_{\mathbb{K}}$, and where $diag(-)$ denotes the diagonal simplicial set of a bisimplicial set.
The crucial point here is that the left-hand commutative square of this diagram is a homotopy pullback, implying that
we have a homotopy pullback of simplicial sets (see \cite[Theorem 0.1]{Yal3})
\[
\xymatrix{P_{\infty}\{X\}\ar[d]\ar[r] & \mathcal{N}(wCh_{\mathbb{K}}^{P_{\infty}})\ar[d]\\
\{X\}\ar[r] & \mathcal{N}wCh_{\mathbb{K}}.
}
\]
Therefore $P_{\infty}\{X\}$ can be identified with the homotopy fiber
\[
\xymatrix{P_{\infty}\{X\}\ar[r]\ar[d] & diag\mathcal{N} fwCh_{\mathbb{K}}^{P\otimes\Delta^{\bullet}}\ar[d]\\
\{X\}\ar[r] & \mathcal{N}fwCh_{\mathbb{K}} \sim  \mathcal{N}wCh_{\mathbb{K}.}
}.
\]
As pointed out in the proof of \ref{P: Rezkgen}, these constructions hold true for more general symmetric monoidal model categories tensored over chain complexes, such as dg modules over a cdga (our case of interest).

The main goal of this section is to prove the following result which identifies the homotopy quotient of $\underline{P_{\infty}\{X\}}$ by the derived prestack group of automophisms of the underlying cochain complex:
\begin{thm}\label{P:BorelvsRezknerve}
Let $X$ be a cochain complex and $\phi:P_{\infty}\rightarrow End_X$ be a prop morphism.
There exists a commutative square
\[
\xymatrix{
E\underline{haut}_{\mathbb{K}}(X)\times_{\underline{haut}_{\mathbb{K}}(X)} \underline{P_{\infty}\{X\}} \ar[r]^-{\sim} \ar@{->>}[d]_{\pi}
& diag\underline{\mathcal{N} fwCh_{\mathbb{K}}}^{P\otimes\Delta^{\bullet}}\ar[d]\\
B\underline{haut}_{\mathbb{K}}(X) \ar[r]^-{\sim} & \underline{\mathcal{N}wCh_{\mathbb{K}}}|_X
}
\]
where $\pi$ is a Kan fibration obtained by simplicial Borel construction
and  the horizontal maps are weak equivalences of simplicial sets,
inducing an equivalence 
\[
\Omega_{[\varphi]}\left(E\underline{haut}_{\mathbb{K}}(X)\times_{\underline{haut}_{\mathbb{K}}(X)} \underline{P_{\infty}\{X\}}\right)\stackrel{\sim}{\rightarrow}
\Omega_{(X,\varphi)}\left(diag\underline{\mathcal{N} fwCh_{\mathbb{K}}}^{P\otimes\Delta^{\bullet}}\right)
\simeq \Omega_{(X,\varphi)}\underline{\mathcal{N}wCh_{\mathbb{K}}}^{P_{\infty}}
\]of derived prestack groups.
\end{thm}
For this aim, we need to define the action of $\underline{haut}_{\mathbb{K}}(X)$ on $\underline{P_{\infty}\{X\}}$.
First, let us recall some homotopical properties of props under base change of cdgas:
\begin{lem}\label{L:basechange}
\begin{itemize}
\item[(1)] Any cdga $A$ induces a Quillen adjunction
\[
(-)\otimes A:Prop(Ch_{\mathbb{K}})\rightleftarrows Prop(Mod_A):U
\]
between the model category of $\mathbb{K}$-linear dg props and the model category of $A$-linear dg props, where $(-)\otimes A$ is the aritywise tensor product by $A$ and $U$ is the forgetful functor.

\item[(2)] Any morphism of cdgas $u:A\rightarrow B$ induces a Quillen adjunction
\[
u_*:=(-)\otimes_AB:Prop(Mod_A)\rightleftarrows Prop(Mod_B):u^*
\]
between the model category of $A$-linear dg props and the model category of $B$-linear dg props, where $u^*$ is the classical restriction functor sending any $B$-module to the same underlying complex with the $A$-module structure induced by $u$ and the tensor product defining $u_*$ is the aritywise tensor product of a prop by a cdga.

\item[(3)] For any complex $X$ and for any chain morphism $f:X\rightarrow Y$, the functor $u_*$ induces well defined prop morphisms
\[
End_{X\otimes A}^{Mod_A}\rightarrow End_{X\otimes B}^{Mod_B}
\]
\[
End_{f\otimes A}^{Mod_A}\rightarrow End_{f\otimes B}^{Mod_B}
\]
where $End_{(-)}^{Mod_A}$ and $End_{(-)}^{Mod_B}$ denotes respectively the endomorphism prop in the category of dg $A$-modules and in the category of dg $B$-modules.

\item[(4)] Given a cofibrant dg prop in $\mathbb{K}$-modules $P_{\infty}$ and a complex $X$, the functor $u_*$ induces also a well defined simplicial map of \emph{Kan complexes}
\[
P_{\infty}\otimes A\{X\otimes A\}\rightarrow P_{\infty}\otimes B\{X\otimes B\} 
\]
where the left hand side and right hand side mapping spaces are taken respectively in the model category of $A$-linear dg props and in the model category of $B$-linear dg props.
\end{itemize}
\end{lem}
\begin{proof}
Recall that $Mod_A$ is a cofibrantly generated symmetric monoidal model category whose fibrations and weak equivalences are those induced by the forgetful functor to complexes (so degreewise surjections and quasi-isomorphisms) and tensor product is defined for any pair of $A$-modules $M$ and $N$ by the coequalizer $$A\otimes M\otimes N\rightrightarrows M\otimes N\rightarrow M\otimes_AN$$ whose pair of arrows are defined by the $A$-module structures of $M$ and $N$.

Claim (1) is \cite[Lemma 3.4]{Yal2}.

Let $u:A\rightarrow B$ be a morphism of cdgas. It induces an adjunction
\[
u_*:=(-)\otimes_AB:Mod_A\rightleftarrows Mod_B:u^*
\]
in which the right adjoint $u^*$ preserves obviously fibrations and weak equivalences, so that it forms actually a Quillen adjunction.
The left adjoint is a strong symmetric monoidal functor via the natural isomorphisms
\[
M\otimes_AN\otimes_AB\cong (M\otimes_AB)\otimes_B(N\otimes_AB)
\]
for $M,N\in Mod_B$
and the right adjoint is a lax monoidal functor via the natural maps $u^*(M)\otimes_A u^*(N)\rightarrow  u^*(M\otimes_BN)$. An adjunction between a strong symmetric monoidal left adjoint and a lax monoidal right adjoint suffices to lift this adjunction at the level of props
\[
u_*:=(-)\otimes_AB:Prop(Mod_A)\rightleftarrows Prop(Mod_B):u^*.
\]
Fibrations and weak equivalences of dg props are defined by the forgetful functor from props to $\mathbb{N}\times\mathbb{N}$-indexed collections of complexes, so these are respectively aritywise surjections and aritywise quasi-isomorphisms. The right adjoint preserves such morphisms, so the adjunction above forms actually a Quillen adjunction  finishing claim (2).

The first morphism of claim (3) follows directly from the fact that $u_*$ is strong symmetric monoidal. The second one is induced by the first one, considering that the endomorphism prop of a morphism $f:X\rightarrow Y$ is given by the pullback
\[
\xymatrix{
End_{f\otimes A}^{Mod_A}\ar[r]\ar[d] & End_{Y\otimes A}^{Mod_A}\ar[d]^-{(f\otimes A)^*} \\
End_{X\otimes A}^{Mod_A}\ar[r]_-{(f\otimes A)_*} & Hom_{XY}^{Mod_A}
}
\]
where $Hom_{XY}^{Mod_A}$ is the sequence of complexes $\{Hom_{Mod_A}((X\otimes A)^{\otimes_Am},(Y\otimes A)^{\otimes_An})\}_{m,n\in\mathbb{N}}$ and the morphisms $(f\otimes A)^*$ and $(f\otimes A)_*$ are defined in each arity $(m,n)$ respectively by postcomposing with $(f\otimes A)^{\otimes n}$ and precomposing with $(f\otimes A)^{\otimes m}$. The morphism
\[
End_{f\otimes A}\rightarrow End_{f\otimes B}
\]
then follows by applying the pullback functor to the morphism of diagrams
\[
\xymatrix{
End_{X\otimes A}^{Mod_A}\ar[d] & Hom_{XY}^{Mod_A}\ar[l]\ar[r]\ar[d] & End_{Y\otimes A}^{Mod_A}\ar[d] \\
End_{X\otimes B}^{Mod_B} & Hom_{XY}^{Mod_B}\ar[l]\ar[r] & End_{Y\otimes B}^{Mod_B}
}
\]
induced by $u_*$.

To prove (4), we first need to observe that $u_*$ is compatible with the simplicial structure of the mapping spaces, thus defines a simplicial functor between the corresponding simplicially enriched categories. For this, we recall from the general definition of simplicial mapping spaces in model categories combined with \cite[Proposition 2.5]{Yal2} that in the model category of $A$-linear dg props, the simplicial mapping space between two props $P$ and $Q$ is given by
\[
Map_{Prop(Mod_A)}(P,Q) = Mor_{Prop(Mod_A)}(P,Q\otimes_A(A\otimes \Omega_{\bullet}))
\]
with a simplicial structure induced by the simplicial cdga $A\otimes \Omega_{\bullet}$.
We apply this to the particular case $P=P_{\infty}\otimes A$ and $Q=End_{X\otimes A}^{Mod_A}$. The functor $u^*$ induces, for every $n$, an application
\begin{eqnarray*}
 & & Mor_{Prop(Mod_A)}(P_{\infty}\otimes A,End_{X\otimes A}^{Mod_A}\otimes_A(A\otimes \Omega_{n})) \\
 & \rightarrow & Mor_{Prop(Mod_B)}((P_{\infty}\otimes A)\otimes_AB,(End_{X\otimes A}^{Mod_A}\otimes_A(A\otimes \Omega_{n}))\otimes_AB) \\
 & = & Mor_{Prop(Mod_B)}(P_{\infty}\otimes B,(End_{X\otimes A}^{Mod_A}\otimes_AB)\otimes_B(A\otimes \Omega_{n}))\otimes_AB) \\
 & = & Mor_{Prop(Mod_B)}(P_{\infty}\otimes B,(End_{X\otimes B}^{Mod_B}\otimes_B(\Omega_{n}\otimes B))
\end{eqnarray*}
which is compatible with the simplicial structures of $A\otimes \Omega_{\bullet}$ and $B\otimes \Omega_{\bullet}$, hence defines a simplicial map
\[
P_{\infty}\otimes A\{X\otimes A\}\rightarrow P_{\infty}\otimes B\{X\otimes B\}.
\]
To conclude, for these mapping spaces to be Kan complexes it is sufficient to prove that the source is cofibrant and the target is fibrant. The source of each mapping space is cofibrant thanks to (1) (since $P_{\infty}$ is cofibrant in $\mathbb{K}$-linear dg props). Dg $A$-modules are all fibrant because fibrations of $A$-modules are defined by the forgetful functor and every complex over a field is fibrant. Moreover, fibrant $A$-linear dg props are the aritywise fibrant ones, so every $A$-linear dg prop is fibrant. This concludes the proof.
\end{proof}
These compatibilities allow us to get the functoriality, up to homotopy, of the $\infty$-action we are going to consider below.

First, let us point out that we have a natural equivalence $haut(X)\simeq fhaut(X)$, where $fhaut(X)$ is the simplicial submonoid of $haut(X)$
whose vertices are the self acyclic fibrations $X\stackrel{\sim}{\rightarrow}X$. 
This equivalence is simply given by the functorial factorization properties of the underlying model category (which replace functorially any weak equivalence by a weakly equivalent acyclic fibration),
implying that every self-weak equivalence of $X$ is in the connected component of a self acyclic fibration of $X$ (recall that every complex over a field is both fibrant and cofibrant, so $X$ is already a fibrant-cofibrant object).

Second, we can transfer $P_{\infty}$-algebra structures along acyclic fibrations, hence a map
\begin{eqnarray*}
haut(X)\times P_{\infty}\{X\} & \rightarrow & P_{\infty}\{X\} \\
(f,\psi) & \longmapsto & Rf_*\psi
\end{eqnarray*}
as follows. We associate to $f$ its equivalent acyclic fibration $Rf$.
Let $End_{Rf}$ be the dg prop associated to the morphism $Rf$. It is defined by the coreflexive equalizer 
\[End_{Rf}(n,m)= Eq\Big(Hom(X^{\otimes n}, X^{\otimes m}) \times Hom(Y^{\otimes n}, Y^{\otimes m})\rightrightarrows 
Hom(X^{\otimes n}, Y^{\otimes m})\Big) \] where the maps are given by either postcomposition or precomposition by $Rf$. Note that 
$End_{Rf}$ has  a natural prop structure and two canonical prop maps $Rf_*, Rf^*$ to $End_X$ and $End_Y$.
We get a lifting
\begin{equation}\label{eq:liftofPsi}
\xymatrix{
0\ar[r]\ar[d] & End_{Rf}\ar[d]^-{Rf^*} \ar[r]^{Rf_*}& End_X\\
P_{\infty}\ar[r]_{\psi} \ar@{-->}[ur]& End_X &
}
\end{equation}
since, by \cite[Lemma 7.2]{Frep},  the right vertical morphism is an acyclic fibration in the model category of dg properads, and $P_{\infty}$ is cofibrant. 

With a slight abuse of notation  \emph{we denote by} $Rf_*\psi$ \emph{the composite}
\begin{equation}
\label{eq:defRfPsi} P_\infty \longrightarrow   End_{Rf} \stackrel{Rf_*}\longrightarrow End_X.
\end{equation}

Moreover, given two homotopy automorphisms $f$ and $g$, the lifts obtained by $R(g\circ f)_*\psi=(Rg\circ Rf)_*\psi$ and $Rg_*(Rf_*\psi)$ are homotopy equivalent by contractibility of the space of lifts in the commutative square above, giving the compatibility up to homotopy with the composition in $haut(X)$.
\begin{lem}\label{L:inftyactionprop}
The map~\eqref{eq:defRfPsi} induces an $\infty$-action of the derived prestack group $\underline{haut}(X)$ on $\underline{P_{\infty}\{X\}}$.
\end{lem}
\begin{proof}
We have to explain why such lifts gives an action at the simplicial level, then why this action is functorial in order to induce the desired action at the level of simplicial presheaves. For this, first recall that we use the fibrant replacement functor $f\mapsto Rf$ to only consider acyclic fibration.  We use the base change properties of Lemma \ref{L:basechange}. Indeed, the lift $P_\infty \longrightarrow   End_{Rf}$ induces for any Artinian cdga $A$ a lift $P_\infty \otimes A\longrightarrow   End_{Rf}^{Mod_A}$ such that for any morphism of cdgas $u:A\to B$, we have a commutative diagram 
\begin{equation}\label{eq:comlift} \xymatrix{ P_\infty \otimes A\ar[r] \ar[d]_{u_*} & End_{Rf}^{Mod_A} \ar[d]^{u_*}\\ P_\infty \otimes B \ar[r] & End_{Rf}^{Mod_B}  }\end{equation}
Note that for any cdga $A$, the category $Mod_A$  is a cofibrantly generated symmetric monoidal model category satisfying the limit monoid axioms \cite[Section 6.6]{Frep},  so that \cite[Lemma 7.2]{Frep} still applies to $A$-linear $P_{\infty}$-structures in $Mod_A^{cof}$. Further,  
 the extension $Rf: X\otimes A\to X\otimes A$ is an acyclic fibration between cofibrant-fibrant objects of $Mod_{A}$. And by the assertion (1) of \ref{L:basechange}, the endomorphism prop maps $End_{Rf}^{Mod_A}\stackrel{Rf^*}\longrightarrow End_{X\otimes A}^{Mod_A} $ are   still a cofibrations (all modules are already fibrant). 

Now, for any $n\in \mathbb{N}$, we work in the model category $Prop(Mod_{\Omega_n})$ and apply this Lemma to the particular case $A=\Omega_n$. Recall that:
\begin{itemize}
\item[-] the $n$-simplices of $fhaut(X)$ are determined by self acyclic fibrations of $X\otimes\Omega_n$ in $\Omega_n$-modules;

\item[-] the set of $n$-simplices of $P_{\infty}\{X\}$ satisfies the natural isomorphism
\[
P_{\infty}\{X\}_n=Mor_{Prop}(P_{\infty},End_X\otimes\Omega_n)\cong Mor_{Prop(Mod_{\Omega_n})}(P_{\infty}\otimes\Omega_n,End_{X\otimes\Omega_n}^{Mod_{\Omega_n})})
\]
identifying $n$-simplices with $\Omega_n$-linear $P_{\infty}$-structures on $X\otimes\Omega_n$,
\end{itemize}
so building our action at the level of $n$-simplices amounts to make the set of self acyclic fibrations of $X\otimes\Omega_n$ (as an $\Omega_n$-module) act on $Mor_{Prop(Mod_{\Omega_n})}(P_{\infty}\otimes\Omega_n,End_{X\otimes\Omega_n}^{Mod_{\Omega_n}})$ by the same formula as in the case $n=0$ above.

We are now interested by the compatibility of our lifts with base changes along morphisms of cdgas. The construction above, written in the particular case $A=\Omega_n$, works exactly the same for any cdga $A$. Let $A\rightarrow B$ be a any morphism of cdgas. Let $f_A:X\otimes A\rightarrow X\otimes A$ an acyclic fibration in the model category of $A$-modules. Under the base change functor $(-)\otimes_AB$ and the fibrant replacement functor, this gives a self acyclic fibration $R(f_A\otimes_AB)$ of $X\otimes B$ in the model category of $B$-modules and the diagram~\eqref{eq:comlift} above ensures the compatibility of those lifts with  the base change functor (lifted to the model categories of props). 
Further, by contractibility of the space  of lifts in the diagram~\eqref{eq:liftofPsi}, we get equivalences  $R(g\circ f)_*\psi=(Rg\circ Rf)_*\psi$ and $Rg_*(Rf_*\psi)$.
The compatibility of theses lifts 
with base changes along morphisms of cdgas explains both :
\begin{itemize}
\item[-] faces and degeneracies of the simplicial structures under consideration are given by base changes along morphisms of cdgas defined by the cosimplicial structure of $\Omega_{\bullet}$ ;

\item[-] the naturality of the action with respect to morphisms of cdgas is its compatibility with the corresponding base change functors.
\end{itemize}
Thus, we get an $\infty$-action of $\underline{haut}(X)$ on $\underline{P_{\infty}\{X\}}$ in simplicial presheaves over artinian cdgas.
\end{proof}
With this $\infty$-action we start the proof of Theorem~\ref{P:BorelvsRezknerve}.
\begin{proof}[Proof of Theorem~\ref{P:BorelvsRezknerve}]
First we construct the commutative diagram
\[
\xymatrix{
Ehaut_{\mathbb{K}}(X)\times_{haut_{\mathbb{K}}(X)} P_{\infty}\{X\} \ar[r] \ar@{->>}[d]_{\pi}
& diag\mathcal{N} fwCh_{\mathbb{K}}^{P\otimes\Delta^{\bullet}}|_X\ar[d]\\
Bhaut_{\mathbb{K}}(X) \ar[r] & \mathcal{N}wCh_{\mathbb{K}}|_X
}
\]
in the following way:
\[
\xymatrix{
((f_k,...,f_0),\varphi:P_{\infty}\otimes\Delta^k\rightarrow End_X) \ar[r] \ar[d]_{\pi=p^*_{haut_{\mathbb{K}}(X)}} &
((X\phi)\stackrel{\sim}{\rightarrow}(X,f_k.\varphi)...\stackrel{\sim}{\rightarrow}(X,(f_k\circ...\circ f_1).\varphi)) \ar[d]^{forget} \\
(f_{k-1},...,f_0) \ar[r] & (X\stackrel{f_{k-1}}{\rightarrow}...\stackrel{f_0}{\rightarrow}X)
}
\]
where the left vertical map is the projection associated to the Borel construction
and the right vertical map forgets the $P_{\infty}\otimes\Delta^k$-algebra structure.
The top horizontal map transfers the $P_{\infty}\otimes\Delta^k$-algebra structure on $X$
along the sequence of quasi-isomorphisms given by $f_k,...,f_0$ and the bottom horizontal map
is just an inclusion.
It is clear by definition of faces and degeneracies in the simplicial structures involved
that these four maps are simplicial.

It remains to prove that the two horizontal maps are weak equivalences. For the bottom arrow, it follows from the work of Dwyer-Kan \cite{DK3} which identifies the connected components of the classification space of a model category with the classifying complexes of homotopy automorphisms $Bhaut(X)$.

For the top arrow, we have a morphism of homotopy fibers over $X$
\[
\xymatrix{
P_{\infty}\{X\} \ar[r]\ar[d]^{=} & Ehaut_{\mathbb{K}}(X)\times_{haut_{\mathbb{K}}(X)} P_{\infty}\{X\} \ar[d] \ar[r] &
Bhaut_{\mathbb{K}}(X)\ar[d]^{\sim} \\
P_{\infty}\{X\} \ar[r] & diag\mathcal{N} fwCh_{\mathbb{K}}^{P\otimes\Delta^{\bullet}} \ar[r] &
\mathcal{N}wCh_{\mathbb{K}}|_X \\
}
\]
inducing another morphism of homotopy fibers
\[
\xymatrix{
\Omega_{[\varphi]}\left(Ehaut_{\mathbb{K}}(X)\times_{haut_{\mathbb{K}}(X)}P_{\infty}\{X\}\right) \ar[r]\ar[d] & haut_{\mathbb{K}}(X) \ar[d]^{\sim}\ar[r] & P_{\infty}\{X\} \ar[d]^{=} \\
\Omega_{(X,\varphi)}\left(diag\mathcal{N} fwCh_{\mathbb{K}}^{P\otimes\Delta^{\bullet}}\right) \ar[r] & \Omega_X\mathcal{N}wCh_{\mathbb{K}}|_X \ar[r] &
P_{\infty}\{X\} \\
}
\]
taken over the base point $\varphi$.
\end{proof}

\begin{prop}\label{P:hautcplx}
For any chain complex $X$, the tangent Lie algebra $Lie(\underline{haut}(X))$ of $\underline{haut}(X)$ is equivalent to 
$End(X)=Hom_{Ch_{\mathbb{K}}}(X,X)$ 
equipped with the commutator of the composition product as Lie bracket.
\end{prop}
\begin{proof}
This follows from Lurie-Pridham correspondence applied to the formal moduli problem $B\underline{haut}(X)$. 
Recall (Definition~\ref{D:TgtofGpPreStack}) that $Lie(\underline{haut}(X)) = Lie(\widehat{\underline{haut}(X)}_{id})$
Further, $\widehat{\underline{haut}(X)}_{id}$  is the based loop space (i.e. the automorphsims) of $ObjDefo_{X}:dgArt_{\mathbb{K}}^{aug}\rightarrow sSet$, 
the deformation object functor of $X$ 
 of~\cite[Section 5.2]{Lur0}.  The latter is a $1$-proximate formal moduli problem in the sense of~\cite[Section 5.1]{Lur0} with associated 
 formal moduli problem denoted $L(ObjDefo_{X})$.
 By~\cite[Lemma 2.11]{BKP}, there is an equivalence of formal moduli problems  
 \begin{equation}\label{eq:LieofHaut} \widehat{\underline{haut}(X)}_{id} \, \cong \, \Omega \big(ObjDefo_{X} \, 
 \big)\stackrel{\cong}\longrightarrow\, \Omega \big(L(ObjDefo_{X}) \big)
 .\end{equation}
Using this equivalence~\eqref{eq:LieofHaut} with Proposition~\ref{P:LieofGroup} then shows that 
\[Lie(\underline{haut}(X)) \, \cong\, \mathfrak{L}_{L(ObjDefo_{X}) }. \]
By \cite[Theorem 5.2.8, Theorem 3.3.1]{Lur0}, the  Lie algebra associated to $L(ObjDefo_{X})$ is precisely 
$Hom_{Ch_{\mathbb{K}}}(X,X)$ with its dg Lie algebra structure.
\end{proof}

We deduce:
\begin{prop}\label{P:BorelvsRezktangent}
There is an equivalence of homotopy fiber sequences of Lie algebras
\[
\xymatrix{
g_{P,X}^{\varphi}\ar[r]\ar[d]^{\sim} &Lie(\underline{haut}_{P_{\infty}-Alg}(X,\varphi)) \ar[r]\ar[d]^{\sim} & End(X)\ar[d]^{\sim} \\
g_{P,X}^{\varphi}\ar[r] & Lie(\Omega_{[\varphi]}(\underline{P_{\infty}\{X\}}//\underline{haut}(X)))\ar[r] & End(X)
}
\]
\end{prop}
\begin{proof}
By Theorem~\ref{P:BorelvsRezknerve}, we have a morphism of homotopy fiber sequences
\[
\xymatrix{
P_{\infty}\{X\} \ar[r]\ar[d]^{=} & Ehaut_{\mathbb{K}}(X)\times_{haut_{\mathbb{K}}(X)} P_{\infty}\{X\} \ar[d] \ar[r] &
Bhaut_{\mathbb{K}}(X)\ar[d]^{\sim} \\
P_{\infty}\{X\} \ar[r] & diag\mathcal{N} fwCh_{\mathbb{K}}^{P\otimes\Delta^{\bullet}} \ar[r] &
\mathcal{N}wCh_{\mathbb{K}}|_X \\
}.
\]
Applying the based loop functor we get an equivalence of homotopy fiber sequences of derived groups, since we also have, 
by Theorem~\ref{P:BorelvsRezknerve},  that the map
\begin{multline*}
\Omega_{[\varphi]}\left(Ehaut_{\mathbb{K}}(X)\times_{haut_{\mathbb{K}}(X)} P_{\infty}\{X\}\right)\stackrel{\sim}{\rightarrow}
\Omega_{(X,\varphi)}\left(diag\mathcal{N} fwCh_{\mathbb{K}}^{P\otimes\Delta^{\bullet}}\right) \\
\simeq \Omega_{(X,\varphi)}\mathcal{N}wCh_{\mathbb{K}}^{P_{\infty}}
\end{multline*}
is an equivalence.  
Taking the formal completions at the appropriate base points and applying the $Lie$-algebra $\infty$-functor combined with Proposition~\ref{P:hautcplx}, 
we obtain the desired equivalence of homotopy fiber sequences of Lie algebras.
\end{proof}

\section{An explicit model via the operad of differentials}\label{S:Plus}
 We now  provide an explicit formula to express the Lie algebra structure of the homotopy automorphisms of a $P_\infty$-algebra, which is crucial to consider deformation complexes of algebraic structures
which also encode compatible deformations of the differential. For this, we use a construction originally due to Merkulov \cite{Mer2},
which gives a conceptual explanation of how one can express the deformation theory inside $P_{\infty}-Alg$ as deformations of a $P_{\infty}$-algebra structure 
in the properadic sense plus compatible deformations of the differential, and formalizes properly Remark~\ref{R:homotopyaction}.

Precisely, 
Theorem~\ref{T:twistedsemidirprod} expresses the Lie algebra structure of the homotopy automorphisms of a $P_\infty$-algebra as a twisted semi-direct product involving the standard operadic deformation complex
of $P_\infty$-algebras. In the next two subsections we actually express the semi-direct product 
explicitly as an $L_{\infty}$-algebra $g_{P^+,X}^{\varphi^+}$, 
obtained as a Maurer-Cartan twisting of a convolution $L_{\infty}$-algebra involving a ``plus construction'' for properads. The use of the $+$ construction to deform usual deformation complex of morphism of properads also appear as a crucial part in Merkulov-Willwacher study of quantization functors~\cite{MW}.

From now, instead of props we work again with properads.

\subsection{The operad of differentials}

We start by recalling the following definition of Merkulov~\cite{Mer2}.
\begin{defn}\label{D:plusconstruction}
Let $P$ be any dg properad  with presentation $P=\mathcal{F}(E)/(R)$ and differential $\delta$.
We define $P^+$ to be the dg-properad with presentation $\mathcal{F}(E^+)/(R)$ and differential  $\delta^+$  where 
the $\Sigma$-biobject $E^+$ is defined by
\[
E^+(1,1)=E(1,1)\oplus\mathbb{K}[1] \mbox{ and } E^+(m,n)=E(m,n).
\]
In other word we add to $E$ a generating operation $u$ of degree $-1$, with one input and one output. The differential $\delta^+$ is the following modification of $\delta$:
\[
\delta^+(u)=u\otimes u\in E(1,1)\otimes E(1,1)
\]
and
\[
\delta^+(g_{m,n}) = \delta (g_{m,n}) -\sum_{i=0}^{m-1}(-1)^kg_{m,n}\circ_{i+1} u + \sum_{i=0}^{n-1}u\circ_{i+1}g_{m,n}
\]
where $g_{m,n}\circ_{i+1} u$ is the insertion of $u$ in the $i+1$ input of $g_{m,n}$, $u\circ_{i+1}g_{m,n}$ the insertion of $u$ in the $i+1$ output of $g_{m,n}$, and $g_{m,n}$ is a degree $k$ generator of $E(m,n)$ for $(m,n)\neq (1,1)$.
\end{defn}
The role of the generator $u$ is thus to twist the differential of a complex $X$ when we consider a $P^+$-algebra structure on $X$. 
The following is proved in~\cite{Mer2} (and also follows from the argument of~\ref{L:Di}).
\begin{lem}
 The construction $P\mapsto (P)^+$ is an  endofunctor $(-)^+:Prop\rightarrow Prop$ of 
 the category of dg-properads. 
 
 Furthermore, properad morphisms $\varphi^+:P^+\rightarrow End_{(X,d)}$ for a given complex $X$ with differential $d$ corresponds to properad morphisms $P\rightarrow End_{(X,d-\varphi^+(u))}$ for $X$ equipped with the twisted differential $d-\varphi^+(u)$.
\end{lem}
In particular, if $X$ is a graded vector space then $P^+$-algebra structures on $X$ equip $X$ simultaneously with a $P$-algebra structure and a compatible differential.

\smallskip

Let us reinterpret this construction by defining the following operad:
\begin{defn}\label{Def:Di}
The operad of differentials $Di$ is the quasi-free operad $Di=(\mathcal{F}(E),\partial)$, where $E(1)=\mathbb{K}\delta$ with $\delta$ a generator of degree $-1$, $E(n)=0$ for $n\neq 1$ and $\partial(\delta)=\delta\circ \delta$ is the operadic composition $\circ:Di(1)\otimes Di(1)\rightarrow Di(1)$.
\end{defn}
We will do an abuse of notation and still note $Di$ the properad freely generated by this operad.
\begin{lem}\label{L:Di}
Let $(V,d_V)$ be a complex.

(1) A $Di$-algebra structure $\phi:Di\rightarrow End_V$ on $V$ is a twisted complex $(V,d_V-\delta_V)$ where $\delta_V$ is the image of the operadic generator $\delta$ under $\phi$.

(2) A morphism of $Di$-algebras $f:(V,d_V-\delta_V)\rightarrow (W,d_W-\delta_W)$ is a chain morphism $f:(V,d_V)\rightarrow (W,d_W)$ which satifies moreover $f\circ (d_V-\delta_V) = (d_W-\delta_W)\circ f$ (it is a morphism of twisted complexes).
\end{lem}
\begin{proof}
(1) The morphism $\phi$ is entirely determined by the image of the generator $\delta$. Since
\[
Di(1)\rightarrow Hom(V,V)
\]
is a morphism of complexes, its compatibility with the differentials reads
\[
\phi(\partial(\delta)) = d_V\circ\delta_V  + \delta_V\circ d_V
\]
which gives the equation of twisting cochains
\[
\delta_V^2 = d_V\circ \delta_V + \delta_V\circ d_V,
\]
hence
\[
(d_V-\delta_V)^2 = d_V^2+\delta_V^2 - d_V\circ\delta_V - \delta_V\circ d_V = 0.
\]

(2) A $Di$-algebra structure on $V$ is given by a morphism $Di(V)\rightarrow V$, and a $Di$-algebra morphism $f:V\rightarrow W$
is a chain morphism fitting in the commutative square
\[
\xymatrix{
Di(V)\ar[r]^-{Di(f)}\ar[d] & Di(W)\ar[d] \\
V\ar[r]_-f & W
}.
\]
Since a $Di$-algebra structure is determined by the image of the generator $\delta$ via $Di(1)\otimes V\rightarrow V$, this amounts to the commutativity of the square
\[
\xymatrix{
Di(1)\otimes V\ar[r]^-{Di(1)\otimes f}\ar[d] & Di(1)\otimes W\ar[d] \\
V\ar[r]_-f & W
},
\]
which is exactly saying that $f$ is a morphism of twisted complexes.
\end{proof}
\begin{rem}\label{R:Diiscofibrant}Let us note that $Di$ is a non-positively graded quasi-free operad, 
but not a \emph{cofibrant} operad. Indeed, to be a retract of a relative cell complex in the model category of dg operads, it needs a filtration $Di$ lacks of (and the same holds true for the corresponding dg properad).
By \cite[Corollary 40]{MV2}, quasi-free properads $(\mathcal{F}(E), \partial)$ endowed with a Sullivan filtration are cofibrant, and any properad admits a resolution by such. A Sullivan filtration (inspired by the Sullivan filtrations of rational homotopy theory) is an exhaustive increasing filtration $(E_i)_{i\geq 0}$ such that $E_0=\{0\}$, the maps $E_i\to E_{i+1}$ are split dg-monomorphisms of $\Sigma$-modules and $\partial (E_i) \subset \mathcal{F}(E_{i-1})$. In the case of the properad generated by the operad $Di$, the $\Sigma$-module of generators is reduced to $E(1,1)=\mathbb{K}\delta$, so there is no other possibility of filtration than the trivial one given by $E_0=\{0\}$ and $E_1=E$, which is not a Sullivan filtration since $\partial(E)\neq \{0\}$. Neither the operad nor the properad $Di$ are cofibrant, they are only $\Sigma$-cofibrant.
\end{rem}
 
\begin{lem}\label{L:Dicontractible}
The (pr)operad $Di$ is contractible in the sense that the initial morphism $I\rightarrow Di$ and the projection $Di\rightarrow I$ are quasi-isomorphisms.
\end{lem}
\begin{proof}
To define our contracting homotopy, let us analyze a bit more the structure of $Di$. We have $Di(n)=0$ for $\neq 1$ and
\[
Di(1) = \mathbb{K}\oplus \bigoplus_{n\geq 1}\mathbb{K}\delta^{\circ n}
\]
where $\circ$ is the operadic composition and $\delta$ is of degree $-1$ (so $\delta^n$ is of degree $-n$). The differential $\partial$ is the extension to $\mathcal{F}(E)$ of the map $E\rightarrow \mathcal{F}(E),\delta\mapsto \delta\circ\delta$. For sign reasons, we have $\partial (\delta\circ\delta) = 0$. By recursion, we deduce that for any natural integer $n$, we have
$\partial (\delta^{\circ 2n})=0$ and $\partial (\delta^{\circ 2n+1})=\delta^{\circ 2n+2}$. Elements of odd degrees are not cycles and elements of even degrees are boundaries except for the identity operation in degree zero, which defines the only non trivial class in homology, so the homology of $Di$ reduces to $I$.

An equivalent way to state this is that there is a chain homotopy between $Id_{Di(1,1)}$ and the composite $Di(1,1)\rightarrow I(1,1)\rightarrow Di(1,1)$ (the projection determined by sending $\delta$ to $0$ followed by the inclusion). This chain homotopy $h:Di(1,1)_{-n}\rightarrow Di(1,1)_{-n-1}$ (where $n$ is a natural integer and the subscript is the homological degree) is defined by $h(\delta^{\circ n})=\delta^{n+1}$. All complexes over $\mathbb{K}$ are fibrant and cofibrant, so chain homotopies are equivalent to homotopies in the sense of model category theory. Given that the converse composite, the inclusion followed by the projection, equals the identity, this means that these form a homotopy equivalence. All chain complexes here being fibrant and cofibrant, a homotopy equivalence is a quasi-isomorphism. Moreover, $Di(m,n)=I(m,n)=0$ for $(m,n)\neq (1,1)$, so the initial morphism and the projection are indeed homotopy inverse quasi-isomorphisms of properads.
\end{proof}

This operad $Di$ is a model for the moduli problem associated to derived homotopy self-equivalences $\underline{haut}(X)$.
Indeed, the operadic moduli space $\underline{Di\{X\}}$ of a $Di$-algebra $X$ controls the homotopy automorphism of the underlying complex:
\begin{prop}\label{P:Di=haut}
There is an isomorphism of dg Lie algebras
\[
g_{Di,X}^{triv}\simeq Lie(\underline{haut}(X)).
\]
\end{prop}
\begin{proof}
The operad $Di$ is a quasi-free resolution of $I$ in the sense of Theorem~\ref{T:Yal2}. It is of the form $Di=(\mathcal{F}(s^{-1}C),\partial)$, 
where $C$ is a cooperad generated by a single generator $u$ of degree $0$ with a coproduct determined by $\Delta_C(u)=u\otimes u$.
Moreover, the trivial $Di$-algebra structure $triv$ sends $u$ to $0$, 
so the Lie bracket on $g_{Di,X}^{triv}=Hom_{\Sigma}(\overline{C},End_X)$ is just the convolution Lie bracket obtained by taking the graded commutator 
of the convolution product. At the level of complexes, we have
\begin{eqnarray*}
g_{Di,X}^{triv} & = & Hom_{\Sigma}(\overline{C},End_X) \\
 & = & Hom_{Ch_{\mathbb{K}}}(\mathbb{K}u,Hom(X,X)) \\
  & \cong & End(X).
\end{eqnarray*}
It remains to compare the Lie structures. Since both Lie brackets are graded commutators of associative products, we just have to compare these products.
The product on $End(X)$ is the composition of homomorphisms. The product on $g_{Di,X}^{triv}$ is the convolution product,
obtained on two elements $f,g:\mathbb{K}u\rightarrow End(X)$ by applying first the infinitesimal cooperadic coproduct $\Delta_{(1)}$ to $u$, 
then replacing the vertices by $f(u)$ and $g(u)$, and finally composing these maps in $End(X)$.
Under the identification of $Hom_{Ch_{\mathbb{K}}}(\mathbb{K}u,End(X))$ with $End(X)$, 
this gives exactly the composition product on $End(X)$ so the two structures agree.
\end{proof}
\begin{rem}
At the level of derived formal groups, this means that we have an equivalence $\widehat{\Omega_{triv}\underline{Di\{X\}}}\simeq \widehat{\underline{haut(X)}}_{id}$ (by taking the loops on the corresponding derived formal moduli problems).
\end{rem}

\subsection{Computing the tangent Lie algebra of homotopy automorphims as an explicit convolution complex}\label{SS:tgtasplus}

In this section, we relate $Lie\big(\underline{haut_{P_\infty}}(X,\varphi)\big)$ with the plus construction. 

\smallskip

\begin{lem}\label{L:pushoutDi} Let $P$ be a properad. 
There is a cofibre sequence of properads
\[
Di\rightarrow P_{\infty}^+\rightarrow P_{\infty}.
\]
where $Di\rightarrow P_{\infty}^+$ is an inclusion and $P_{\infty}^+\rightarrow P_{\infty}$ the forgetful map sending the generator of $Di$ to $0$.
\end{lem}
\begin{proof}
We compare first $P_{\infty}^+$ and $P_{\infty}\vee Di$, where $\vee$ stands for the coproduct of properads 
(see \cite[Appendix A.3]{MV2} for its definition). Since the free properad functor $\mathcal{F}$ is a left adjoint, 
it preserves coproducts and thus comes with natural isomorphisms $\mathcal{F}(M\oplus N)\cong\mathcal{F}(M)\vee\mathcal{F}(N)$. 
If we take the coproduct $P_{\infty}\vee Q_{\infty}$ of two quasi-free properads $P_{\infty}=(\mathcal{F}(M),\partial_P)$ 
and $Q_{\infty}=(\mathcal{F}(M),\partial_Q)$, then via the previous isomorphism we can define a differential on $\mathcal{F}(M\oplus N)$ 
by taking the derivation associated to
\[
\partial_P|_M\oplus \partial_Q|_N:M\oplus N\rightarrow \mathcal{F}(M)\oplus \mathcal{F}(N)\hookrightarrow \mathcal{F}(M\oplus N)
\]
by universal property of derivations and the fact that this morphism satisfies the twisting cochain equation.
In the case where $Q=Di$, it turns out that the free properad underlying $P_{\infty}^+$ is $\mathcal{F}(M\oplus\mathbb{K}d)$ 
and the differential on $P_{\infty}^+$ is the sum of the one above with a derivation coming from the action of $u$ in Definition \ref{D:plusconstruction}. We have in particular a \emph{graded} properad equality $P_{\infty}^+ = P_{\infty}\vee Di$. The canonical inclusion $Di\rightarrow P_{\infty}\vee Di$ is compatible with the differential of $P_{\infty}^+$ since the restriction of the later to $Di$ is exactly, by definition, the differential of $Di$. This canonical inclusion, in turn, is induced by the inclusion of $\Sigma$-bimodules $\mathbb{K}u\hookrightarrow M\oplus \mathbb{K}u$ whose cokernel is $M$. The free properad functor commutes with cokernels and the induced cokernel sequence is compatible with the differentials, hence the claim.
\end{proof}
\begin{rem}
The properads $Di$ and $P_{\infty}^+$ are not cofibrant. Moreover, it seems unlikely that the morphism $Di\rightarrow P_{\infty}^+$ forms a cofibration. Indeed, according to the characterization of cofibrations of properads given by \cite[Proposition 37]{MV2} (very close in spirit to the notion of relative Sullivan model in rational homotopy theory), we need to get a retract of a map of the form $Q\rightarrow Q\vee \mathcal{F}(S)$ where the generators $S$ of the quasi-free part are equipped with an increasing exhaustive filtration and the differential lowers the filtration rank of the elements of $S$. But as we see in Definition \ref{D:plusconstruction}, there is a part of the differential which cannot lower the filtration rank of any filtration one considers on the generators of $P_{\infty}$.

Moreover, the model category of properads is not left proper, so the sequence above cannot be a homotopy cofibre sequence of properads without cofibrancy of the properads involved.
\end{rem}
Although not being a homotopy cofibre sequence, it induces a homotopy fiber sequence of deformation complexes:
\begin{prop} \label{P:fiberseqplus}
The cofiber sequence of Lemma~\ref{L:pushoutDi} induces a homotopy fiber sequence of $L_{\infty}$-algebras
\[
g_{P,X}^{\varphi}\rightarrow g_{P^+,X}^{\varphi^+}\rightarrow g_{Di,X}^{triv}.
\]
\end{prop}
\begin{proof}
The cofibre sequence of Lemma~\ref{L:pushoutDi} induces a fibre sequence of convolution $L_{\infty}$-algebras
\[
g_{P,X}\rightarrow g_{P^+,X}\rightarrow g_{Di,X}.
\]
Indeed, by direct computation at the chain level $g_{P,X}$ is the kernel of the second map, and these maps are compatible with the convolution brackets.
Moreover, for any $P_{\infty}$-algebra $(X,\varphi)$, we have 
a commutative diagram of properad morphisms
\[
\xymatrix{
Di \ar@{^{(}->}[r] \ar[dr]^-{triv} & P_{\infty}^+ \ar@{->>}[r] \ar[d]^-{\varphi^+} & P_{\infty} \ar[dl]_-{\varphi} \\
 & End_X & }
\]
where the maps relating $Di$, $P_{\infty}$ and $P_{\infty}^+$ are the ones defined in the Lemma~\eqref{L:pushoutDi}.
Consequently, along this fiber sequence of $L_{\infty}$-algebras, the Maurer-Cartan element $\varphi$ of $g_{P,X}$ 
is sent to $\varphi^+$, which is in turn sent to $triv$. Twisting our $L_{\infty}$-algebras by these Maurer-Cartan elements produces a new fiber sequence
\[
g_{P,X}^{\varphi}\rightarrow g_{P^+,X}^{\varphi^+}\rightarrow g_{Di,X}^{triv}.
\]
To conclude, the second arrow is a surjection, hence a fibration in the model category of $L_{\infty}$-algebras, and all objects are fibrant, so this fiber sequence is a homotopy fiber sequence.
\end{proof}

To conclude, we compare this fiber sequence with the fiber sequence
\[
g_{P,X}^{\varphi}\rightarrow Lie(\underline{haut_{P_{\infty}-Alg}}(X,\varphi))\rightarrow Lie(\underline{haut}(X)).
\]
of Theorem~\ref{P: hofibgroups} to obtain:
\begin{thm}\label{T:Def+=hAut}
There is a quasi-isomorphism of $L_{\infty}$-algebras
\[
g_{P^+,X}^{\varphi^+}\simeq Lie(\underline{haut}_{P_{\infty}}(X,\varphi)).
\]
\end{thm}
\begin{proof}
We already have by Theorem~\ref{P:BorelvsRezknerve} an equivalence of homotopy fiber sequences
\[
\xymatrix{
\underline{P_{\infty}\{X\}}\ar[r]\ar[d]^{\sim} & \underline{\mathcal{N}wP_{\infty}-Alg}|_X\ar[r]\ar[d]^{\sim} & B\underline{haut}(X)\ar[d]^{\sim} \\
\underline{P_{\infty}\{X\}}\ar[r] & \underline{P_{\infty}\{X\}}//\underline{haut}(X)\ar[r] & B\underline{haut}(X)
}
\]
inducing a homotopy fiber sequence of $L_{\infty}$-algebras
\[
\xymatrix{
g_{P,X}^{\varphi}\ar[r]\ar[d]^{\sim} &Lie(\underline{haut}_{P_{\infty}-Alg}(X,\varphi)) \ar[r]\ar[d]^{\sim} & End(X)\ar[d]^{\sim} \\
g_{P,X}^{\varphi}\ar[r] & Lie(\Omega_{[\varphi]}(\underline{P_{\infty}\{X\}}//\underline{haut}(X)))\ar[r] & End(X)
}
\]
by Proposition \ref{P:BorelvsRezktangent}.
To conclude the proof, we have to compare  the lower fiber sequence with the fiber sequence
\[
g_{P,X}^{\varphi}\rightarrow g_{P^+,X}^{\varphi^+}\rightarrow g_{Di,X}^{triv}
\]
of Proposition \ref{P:fiberseqplus}. By the Lurie-Pridham correspondence, this fiber sequence gives a fiber sequence of formal moduli problems
\[
\widehat{\underline{P_{\infty}\{X\}}}\rightarrow \widehat{\underline{P_{\infty}^+\{X\}}}\rightarrow \widehat{\underline{Di\{X\}}}
\]
(as well as a fiber sequence of the corresponding derived formal groups of loops).

From the corresponding long fibration sequences we get two fiber sequences
\begin{equation}\label{eq:Def+haut1}
\Omega_{[\varphi]}\widehat{\underline{P_{\infty}\{X\}}//\underline{haut}(X)}\rightarrow \widehat{\underline{haut}(X)}\rightarrow \widehat{\underline{P_{\infty}\{X\}}}
\end{equation}
and
\begin{equation}
 \label{eq:Def+haut2}
\Omega_{\varphi^+}\widehat{\underline{P_{\infty}^+\{X\}}}\rightarrow \Omega_{triv}\widehat{\underline{Di\{X\}}}\rightarrow \widehat{\underline{P_{\infty}\{X\}}}.
\end{equation}
Our strategy is to define a commutative square
\begin{equation}\label{eq:Def+haut}
\xymatrix{
\Omega_{triv}\underline{Di\{X\}}\ar[r]\ar[d] & \underline{P_{\infty}\{X\}}\ar[d] \\
\underline{haut}(X)\ar[r]  & \underline{P_{\infty}\{X\}}
}
\end{equation}
hence inducing a morphism between the corresponding homotopy fiber sequences, so that the two vertical arrows induce equivalences of $L_{\infty}$-algebras 
at the tangent level, after completion of the derived groups at the appropriate base points.

The right hand vertical arrow of~\eqref{eq:Def+haut} is just the identity morphism.
Recall that, by \cite[Theorem 5.2.1]{Fre-cyl}, if $O_{\infty}$ is the cobar construction on a $\Sigma$-cofibrant dg cooperad, 
then the homotopies between two morphisms $\varphi,\psi:O_{\infty}\rightarrow End_X$ are in bijection with $\infty$-quasi-isotopies in $O_{\infty}-Alg$
between the corresponding $O_{\infty}$-algebras, that is, $\infty$-quasi-isomorphisms whose first level 
lies in the connected component of the identity in $haut(X)$. 

In particular, a loop in $\Omega_{\varphi}O_{\infty}\{X\}$, that is, a self-homotopy of $\varphi$, induces a self-$\infty$-isotopy of $(X,\varphi)$. 
In the particular case where the operad is augmented and $\varphi$ is the trivial $O_{\infty}$-algebra structure on $X$ 
(that is, it factorizes through the augmentation $O_{\infty}\rightarrow I$), 
then such a self-$\infty$-isotopy is just a self quasi-isomorphism in the connected component of the identity. Consequently, there is a natural map
\[
\Omega_{triv}\underline{Di\{X\}}\rightarrow \Omega_{triv}\underline{\mathcal{B}^{co}\mathcal{B}Di\{X\}}\rightarrow \underline{haut}(X)
\]
where $\mathcal{B}^{co}\mathcal{B}Di$ is the bar-cobar resolution of $Di$, the first arrow is induced by the operad morphism $\mathcal{B}^{co}\mathcal{B}Di\rightarrow Di$ and the second arrow is the one explained above which takes loops on the trivial algebra structure to self-$\infty$-isotopies of the trivial algebra structure, that is, to self-quasi-isomorphisms.
This natural map makes  the commutative square~\eqref{eq:Def+haut} above commutes, and becomes an isomorphism when restricting the target to the connected component of $id_X$. 
This means that, even though this map is not an equivalence, 
taking the tangent Lie algebras of the derived formal groups obtained after completions at the appropriate base points
(trivial loop on the left, $id_X$ on the right) leads to a quasi-isomorphism of Lie algebras
\[
g_{Di,X}^{triv}\stackrel{\simeq}{\rightarrow} Lie(\widehat{\underline{haut}(X)}_{id})=End(X).
\]

Now, the formal completion of this commutative square~\eqref{eq:Def+haut} induces a morphism between the homotopy fibers given by~\eqref{eq:Def+haut1} and~\eqref{eq:Def+haut2}
\[
\Omega_{\varphi^+}\widehat{\underline{P_{\infty}^+\{X\}}} \rightarrow \Omega_{[\varphi]}\widehat{\underline{P_{\infty}\{X\}}//\underline{haut}(X)}.
\]
Since this square becomes a square of quasi-isomorphisms of Lie algebras at the tangent level, the induced morphism between the Lie algebras of the fibers is a quasi-isomorphim as well
\[
g_{P^+,X}^{\varphi^+} = Lie(\widehat{\Omega_{\varphi^+}\underline{P_{\infty}^+\{X\}}}) \stackrel{\sim}{\rightarrow} Lie(\widehat{\Omega_{[\varphi]}\underline{P_{\infty}\{X\}}//\underline{haut}(X)}).
\]
Therefore, Proposition~\ref{P:BorelvsRezktangent} gives us the equivalence of the later Lie algebra with $Lie(\underline{haut}_{P_{\infty}}(X,\varphi))$.
\end{proof}
\begin{rem}
Although it is interesting to see the role of the Borel construction here, there is an alternate 
proof of Theorem~\ref{T:Def+=hAut} which makes no use of it. Let us sketch it; for this, we compare the fiber sequences
\[
\Omega_{[\varphi]}\underline{P_{\infty}\{X\}}//\underline{haut}(X)\rightarrow \underline{haut}(X)\rightarrow \underline{P_{\infty}\{X\}}
\]
and
\[
\underline{haut}_{P_{\infty}}(X,\varphi)\rightarrow \underline{haut}(X)\rightarrow \underline{P_{\infty}\{X\}}
\]
by checking that actually, in both cases we are considering the fibers of the same map from $\underline{haut}(X)$ to $\underline{P_{\infty}\{X\}}$, which is the map sending a homotopy automorphism to its action on $\varphi$.
Hence an equivalence of Lie algebras
\[
Lie(\underline{haut}_{P_{\infty}}(X,\varphi))\simeq Lie(\Omega_{[\varphi]}\underline{P_{\infty}\{X\}}//\underline{haut}(X)).
\]
Then, the argument line of the proof above provides the equivalence
\[
Lie(\Omega_{[\varphi]}\underline{P_{\infty}\{X\}}//\underline{haut}(X)).
\]
\end{rem}
Theorem~\ref{T:Def+=hAut} shows that the $+$ construction is crucial to study 
deformation of dg-algebras and not just deformations of algebraic structures on a fixed complex.

\subsection{Computing the tangent Lie algebra of homotopy automorphims as a twisted semi-direct product}

\subsubsection{Homotopy representations of $L_{\infty}$-algebras and semi-direct products}

Recall (\ref{D:Linfty}) that the structure of a $L_\infty$-algebra $g$ is encoded by a (cohomogical degree $-1$) coderivation $Q_g$ of square zero on 
$Sym^{\bullet \geqslant 1}(g[1])$. Dualizing this coderivation induces an augmented cdga structure on  \[{C}_{CE}^*(g):=Hom (Sym^\bullet(g[1]), k)\]
which is called 
the \emph{Chevalley-Eilenberg cochain algebra} of $g$; we denote $\varepsilon$ the augmentation. 
For any graded module $N$, $Hom (Sym^\bullet(g[1]), N)$ inherits similarly a structure of graded 
${C}_{CE}^*(g)$-module. 
\begin{defn}\label{Def:RepofLinfty} Let $(g, Q_g)$ be a $L_\infty$-algebra and $(M, d_M)\in Ch_{\mathbb{K}}$.

 A homotopy representation of an $L_{\infty}$-algebra $g$ on $M$  is a \emph{derivation} $D$ of \emph{square zero} and (cohomological) degree $1$ 
on $C_{CE}^*(g,M):=Hom_{dg}(Sym^\bullet(g[1]), M)$ such that $M\stackrel{D}\to C_{CE}(g,M)\stackrel{\varepsilon}\to M$ is equal to the inner differential 
$d_M$ of the complex $M$. The fact that $D$ is a derivation means precisely that it satisfies the following Leibniz relation: for $f \in {C}_{CE}^{n}(g)$, 
$\Phi \in C^*_{CE}(g,M)$, one has 
\begin{equation}
 D( f\cdot \Phi ) = Q_g(f) \cdot \Phi + (-1)^{n} f\cdot D(\Phi).
\end{equation}
\end{defn} 

\begin{example}A  particular case of homotopy representation is the standard notion of representation, given by a dg Lie algebra morphism $g\rightarrow End(M)$ and the 
 standard Chevalley-Eilenberg cochain complexes. This generalizes easily to any $L_\infty$-algebra $g$ and morphism of $L_{\infty}$-algebras $g\to End_{Ch_{\mathbb{K}}}(M,M)$.
 \end{example} 
Definition~\ref{Def:RepofLinfty} is equivalent to the data of a $L_\infty$-algebra structure on $g\oplus M$, that is  a coderivation of square zero 
on $Sym^{\bullet \geqslant 1}((g\oplus M)[1])$ that vanishes on the coideal spanned by $ Sym^{\bullet \geqslant 2}(M)$ and whose restriction 
to $Sym^{\bullet \geqslant 1}(g[1])$ and $M$ are respectively $Q_g$ and the inner differential of $M$
(followed by the canonical inclusions of these complexes 
in $Sym^{\bullet \geqslant 1}((g\oplus M)[1])$). In other words it is a square zero extension by $M$ of the $L_\infty$-algebra structure of $g$.

 \begin{example}[semi-direct product]\label{ex:semidirectofLinfty} If $h$ is a dg Lie algebra, any dg Lie algebra homomorphism
 $g \to Der(h)$ induces an action of $g$ onto $C^{CE}_* (h)$ as the coderivation extending the $g$-action on $h$. Similarly, if $h$ and $g$ are  $L_{\infty}$-algebras and given an $L_{\infty}$-algebra morphism $\varphi:g\rightarrow Der(h)$, we obtain a homotopy representation of $g$ on 
 $C^{CE}_* (h)$. 
 The coalgebra structure of $C^{CE}_* (h)$ then yields respectively a cocommutative dg-coalgebra and a cdga structure on 
 \begin{equation}  CE_*(g,h) := C^{CE}_* \big(g, C^{CE}_* (h)\big), \quad CE^*(g,h) := C_{CE}^* \big(g, C_{CE}^* (h)\big). \end{equation}
 The underlying graded vector space is $CE_*(g,h) = Sym \big( g[1] \oplus h[1]\big)$, and the \emph{semi-direct product $h\rtimes g$} is the direct sum $g\oplus h$ equipped with the $L_{\infty}$-algebra structure 
coming from the differential on the coalgebra $CE_*(g,h)$. 
In particular, the augmentations yield a
cofiber sequence of cdgas
\[
C^*_{CE}(g)\rightarrow CE^*(g,h)\rightarrow C^*_{CE}(h),
\]
which is dual to a fiber sequence of dg-cocommutative coalgebras 
\[
C_*^{CE}(h)\rightarrow CE_*(g,h)\rightarrow C_*^{CE}(g),
\] which is equivalent to a fiber sequence of
$L_{\infty}$-algebras
\[
h\rightarrow h\rtimes g \rightarrow g
\]
forming a split extension of $g$ by $h$. 

\end{example}
\begin{example}
 In particular the adjoint action  $ad:g\to Der(g)$ of a $L_\infty$-algebra $g$ on itself yields the semi-direct product $g\rtimes_{ad} g$.
\end{example}

\subsubsection{The case of the trivial algebra structure}

To consider the trivial algebra structure, we assume that $P_{\infty}$ is augmented for such a trivial structure to make sense, via the morphism $P_{\infty}\rightarrow I\rightarrow End_X$. This assumption is satisfied for a wide class of properads including Koszul and homotopy Koszul properads. Recall that $g_{P,X}^0$ is the $L_\infty$-algebra encoding the formal moduli problem $\underline{P_{\infty}\{X\}^0}$ and is nothing but the untwisted properadic convolution  $L_\infty$-algebra of Section $2$.

\begin{lem}\label{L:trivialsemidirect}
There is a filtered quasi-isomorphism of $L_{\infty}$-algebras
\[
g_{P^+,X}\simeq g_{P,X}^0\rtimes End(X).
\]
\end{lem}
\begin{proof}
We have a splitting of complexes
\[
Hom_{\Sigma}(\overline{\mathcal{C}}\oplus s\mathbb{K}u,End_X))\cong Hom_{\Sigma}(\overline{\mathcal{C}},End_X)\oplus g_{Di,X}^0 \cong g_{P,X}^0\oplus End(X)
\]
which is moreover compatible with the convolution products (there is no twist of the differential by a non trivial Maurer-Cartan element).
\end{proof}

\begin{rem}\label{L:sectionofhautPinfty}
We can also get, although more abstractly, the quasi-isomorphism above by constructing a section directly at the level of derived formal groups. As we have already seen, the forgetful functor mapping $P_\infty$-algebras to their underlying complexes induces a morphism of the homotopy automorphisms derived prestack groups of both categories.
Denoting by $(X,0)$ the trivial $P_\infty$-algebra with underlying complex $X$,
 the forgetful derived formal group morphism 
 \[\widehat{\underline{haut_{P_\infty -Alg}}}(X,0)_{id} \, \longrightarrow \, \widehat{\underline{haut}}(X)_{id} \]
 has a section in derived formal groups.

To prove this, by Lemma~\ref{L:hautasLoop}, since being group-like is a property, 
it is sufficient to construct a  $E_1$-monoid morphism 
\begin{equation}\label{eq:seekformonoidmap} \Omega_{X\otimes R}\Big(\mathcal{N}wCh_{{R}}\Big) \longrightarrow \Omega_{(X\otimes R,0)}
 \Big(
\mathcal{N}wP_{\infty}-Alg(Mod_R^{cof})\Big)\end{equation}
of simplicial presheaves which is a section of the forgetful morphism. Then, by applying the homotopy fiber functor over the identities we get a section of the forgetful derived formal group morphism. Since we are taking loop spaces at $X\otimes R$ and $(X\otimes R,0)$, we can restrict the considered spaces respectively to the connected component $(\mathcal{N}wCh_{{R}})_{X\otimes R}$ of $X\otimes R$ and the connected component $\mathcal{N}wP_{\infty}-Alg(Mod_R^{cof})_{(X\otimes R,0)}$ of $(X\otimes R,0)$. We search for a pointed map
\[
(\mathcal{N}wCh_{{R}})_{X\otimes R}\rightarrow \mathcal{N}wP_{\infty}-Alg(Mod_R^{cof})_{(X\otimes R,0)}
\]
(where the base points are respectively $X\otimes R$ and $(X\otimes R,0)$
whose composite with the forgetful map is the identity, so that applying the pointed loop space functor gives us the map \ref{eq:seekformonoidmap}.

For simplicity, we will use the description of props as strict small symmetric monoidal dg categories and morphisms of props as symmetric monoidal dg functors. Let us note $Fun_{dg}^{\otimes}(-,-)$ for the category of symmetric monoidal $R$-dg functors between dg categories with symmetric monoidal natural dg transformations. The category of $P_{\infty}\otimes R$-algebras whose underlying complex is $X\otimes R$ can then alternately be described as $Fun_{dg}^{\otimes}(P_{\infty}\otimes R,End_{X\otimes R})$ : one checks that such a natural transformation is defined by a collection of maps $\{\tau (n):(X\otimes R)^{\otimes n}\rightarrow (X\otimes R)^{\otimes n}\}_{n\in\mathbb{N}}$ and that, by strict monoidality and by its compatibility with the functors, it is uniquely determined by the data of the morphism $\tau (1): X\otimes R\rightarrow X\otimes R$ compatible with the $P_{\infty}$-algebra structures. The trivial prop morphism $0:P_{\infty}\otimes R\rightarrow I\otimes R\rightarrow End_{X\otimes R}$ then defines a functor, and the precomposition by $0$ induces a functor of $R$-dg categories
\[
0^*:Fun_{dg}^{\otimes}(End_{X\otimes R},End_{X\otimes R})\rightarrow Fun_{dg}^{\otimes}(P_{\infty}\otimes R,End_{X\otimes R})
\]
(precomposition of a natural transformation $\tau:F\rightarrow G$ by a functor $H$ gives still a natural transformation $\tau\circ H: F\circ H\rightarrow G\circ H$),
hence the simplicial map
\[
\mathcal{N}wFun_{dg}^{\otimes}(End_{X\otimes R},End_{X\otimes R})\rightarrow \mathcal{N}wFun_{dg}^{\otimes}(P_{\infty}\otimes R,End_{X\otimes R})\hookrightarrow \mathcal{N}wP_{\infty}-Alg(Mod_R^{cof}).
\]
This map sends $Id_{End_{X\otimes R}}$ to $(X\otimes R,0)$, so it sends the connected component of $Id_{End_{X\otimes R}}$ in the simplicial nerve into $\mathcal{N}wP_{\infty}-Alg(Mod_R^{cof})_{(X\otimes R,0)}$. The connected component of $Id_{End_{X\otimes R}}$ is $Bhaut(X)$. Indeed, by strict monoidality, a natural weak self-equivalence of the identity functor $Id_{End_{X\otimes R}}$ is uniquely determined by a self weak equivalence of $X\otimes R$ and vice-versa.
Finally, the composite with the forgetful map is obviously $Id_{Bhaut(X)}$ because $0^*$ does not change  nor the underlying complex $X\otimes R$ neither its chain self-weak equivalences.
\end{rem}
\begin{rem}
Another way to build a section is to consider the composite map
\[
\Omega_{id}\underline{End_X\{X\}}\rightarrow \Omega_{0}\underline{P_{\infty}\{X\}}\rightarrow\underline{haut_{P_{\infty}}(X,0)}
\]
where the first map is the looping of the precomposition by $0$ between the simplicial mapping spaces and the second map is the looping of the map induced by Rezk's homotopy pullback theorem. Then, one notices that $\underline{End_X\{X\}}$ is, for each cdga $R$, the nerve $\mathcal{N}wFun_{dg}^{\otimes}(End_{X\otimes R},End_{X\otimes R})$ considered in the proof above, so $\Omega_{id}\underline{End_X\{X\}}$ is  $\Omega_{id}\underline{\mathcal{N}wFun_{dg}^{\otimes}(End_{X},End_{X})}=\Omega_{id}B\underline{haut(X)}=\underline{haut(X)}$.
\end{rem}

\medskip

\subsubsection{Twisted semi-direct products and the general case}

We borrow the notion of \textbf{twisted semi-direct product of dg Lie algebras} from \cite{Tan,Ber}. We first give the definition  in details for the case of dg Lie algebras for the sake of clarity, before extending this notion to the case of $L_{\infty}$-algebras.
Let $h$ be a dg Lie algebra acting by derivations on another dg Lie algebra $g$ and  $\xi:h[-1]\rightarrow g$ be a map, called a twist, satisfying the axioms given \cite[Definition 3.13]{Ber}. One define their twisted semi-direct product  $g\rtimes_{\xi}h$ as follows. As a graded Lie algebra, this is the untwisted semi-direct product  $g\rtimes h$ with its Lie bracket defined for any $(a,x),(b,y)\in g\oplus h$
\[
[(a,x),(b,y)] = ([a,b]+x.b+a.y,[x,y])
\]
where $a.y$ is the action of $y$ on $a$ and $x.b=-(-1)^{|x||b|}b.x$. As a dg Lie algebra, its differential is the twisted differential
\[
\partial^{\xi}(a,x) = (da+\xi(x),dx).
\]
\begin{lem}\label{L:LiesemidirectMCtwist}
\begin{itemize}
\item[(1)] A pair $(a,x)\in g\oplus h$ is a Maurer-Cartan element in the semi-direct product $g\rtimes h$ if and only if it satisfies the equations
\[
dx+\frac{1}{2}[x,x] = 0
\]
(that is, $x$ is Maurer-Cartan element of $h$)
and
\[
da+\frac{1}{2}[a,a] = -a.x.
\]

\item[(2)] Given a Maurer-Cartan element $a$ of $g$, the pair $(a,0)$ is a Maurer-Cartan element of $g\rtimes h$ and the twisting $(g\rtimes h)^{(a,0)}$ is the twisted semi-direct product $g^a\rtimes_{\xi}h$ for $\xi(x)=a.x$.
\end{itemize}
\end{lem}
\begin{proof}
To prove (1), one just writes explicitely the Maurer-Cartan equation $d(a,x)+\frac{1}{2}[(a,x),(a,x)]=0$ for the Lie bracket of the semi-direct product $g\rtimes h$ :
\begin{eqnarray*}
d(a,x)+\frac{1}{2}[(a,x),(a,x)] & = & (da,dx)+\frac{1}{2}([a,a]+x.a+a.x,[x,x]) \\
 & = & (da+\frac{1}{2}([a,a]+2a.x),[x,x]) \\
 & = & (da+\frac{1}{2}[a,a]+a.x,[x,x]).
\end{eqnarray*}

To prove (2), let us start by noticing that, according to (1), a degree $1$ element of the form $(a,0)$ is a Maurer-Cartan element of $g\rtimes h$ if and only if $a$ is a Maurer-Cartan element of $g$. Moreover, the two dg Lie algebras $(g\rtimes h)^{(a,0)}$ and $g^a\rtimes_{\xi}h$ coincide as graded Lie algebras, so we just have to check that their differentials coincide as well. The new differential of the twisted Lie algebra $g^a$ is $d_a=d+[a,-]$ where $d$ is the differential of $g$. By definition of the Lie bracket of $g\rtimes h$, we have
\[
[(a,0),(b,y)] = ([a,b]+0.b+a.y,[0,y]) = ([a,b]+a.y,0)
\]
so the differential of $(g\rtimes h)^{(a,0)}$ is
\[
\partial_{(a,0)}(b,y) = d_{g\oplus h}(b,y)+ [(a,0),(b,y)] = (db+[a,b]+a.y,dy)=(db+[a,b]+\xi(y),dy).
\]
One recognizes the differential of $g^a\rtimes_{\xi}h$.
\end{proof}
\begin{rem}
As can be seen in the computations above, not any twisting by a Maurer-Cartan element can be expressed as a twisted semi-direct product in the sense of \cite[Definition 3.13]{Ber}.
\end{rem}

We focus now on similar constructions in the setting of $L_{\infty}$-algebras. As in Example \ref{ex:semidirectofLinfty}, we consider an $L_{\infty}$-algebra $h$ acting on another $L_{\infty}$-algebra $g$ by derivations, that is with a morphism of $L_{\infty}$-algebras $h\rightarrow Der(g)$. On the one hand, recall that in the untwisted case, we get a cofiber sequence of cdgas
\[
C^*_{CE}(h)\rightarrow CE^*(h,g)\rightarrow C^*_{CE}(g)
\]
where the differential of the middle term encodes the $L_{\infty}$-brackets of the semi-direct product $g\rtimes h$.Given a Maurer-Cartan element $a$ of $g$, the pair $(a,0)$ is a Maurer-Cartan element of $g\rtimes h$. Indeed, Maurer-Cartan elements correspond to augmentations of the corresponding Chevalley-Eilenberg algebras, so the Maurer-Cartan equation for $(a,0)$ follows from the following diagram of augmentations
\[
\xymatrix{
C^*_{CE}(h)\ar[r]\ar[dr] & CE^*(h,g)\ar[r]\ar[d] & C^*_{CE}(g)\ar[dl] \\
 & \mathbb{K} & 
}
\]
where the right vertical arrow corresponds to $a$, the left vertical arrow corresponds to $0$, and the middle vertical arrow is induced from these two by the formula defining $CE^*(h,g)$.

On the other hand, the notion of twisted semi-direct product still makes sense for $L_{\infty}$-algebras by defining it as follows:

\begin{defn}[Twisted semi-direct products of $L_{\infty}$-algebras]\label{D:semidirectMCtwist} Let $h$ be an $L_{\infty}$-algebra acting on an $L_{\infty}$-algebra $g$ by derivations.
A twisted semi-direct product of $g$ by $h$ is a $L_{\infty}$-algebra structure on the graded vector space $g\oplus h$ defined by a fiber sequence of dg coalgebras
\[
C_*^{CE}(g)\rightarrow CE_*^{\xi}(h,g)\rightarrow C_*^{CE}(h)
\]
where the middle term is the quasi-cofree coalgebra $Sym((g\oplus h)[1])$ equipped with a square zero derivation $D_{\xi}$ of degree $-1$ such that :
\begin{itemize}
\item[(1)] the part of weight $1$ is of the form $(d_g+\xi,d_h)$, where $\xi$ is a degree $-1$ map $\xi:h[-1]\rightarrow g$;

\item[(2)] we have $D_{\xi}=D+\xi$ where $D$ is the square zero coderivation on $Sym((g\oplus h)[1])$ defined by the action of $h$ on $g$.
\end{itemize}
We denote  $g\rtimes_{\xi} h$ the twisted semi-direct product given by $D_{\xi}$.
\end{defn}
The notion of twisted semi-direct product of dg Lie algebras from~\cite[Definition 3.13]{Ber} recalled above is a special case.
Indeed, the relation $D_{\xi}^2=0$ ensures the appropriate generalization, for $L_{\infty}$-algebras, of the compatibility of $\xi$ with the differentials of $g$ and $h$, with the higher $L_{\infty}$-brackets. Point (2) gives the compatibility with the action of $h$ on $g$. That is, the $L_{\infty}$-algebra $g\rtimes_{\xi} h$ has the same brackets as $g\rtimes h$ but a differential twisted by $\xi$.
It can be equivalently defined by the corresponding cofibre sequence of cdgas
\[
C^*_{CE}(h)\rightarrow CE^*_{\xi}(h,g)\rightarrow C^*_{CE}(g).
\]
As a consequence of the diagram of augmentations above and the definition of $\xi$,
we get a fiber sequence
\[
C^*_{CE}(h)\rightarrow (CE_{\xi}^*(h,g^a)\rightarrow C^*_{CE}(g^a).
\]
This means:
\begin{lem}\label{L:semidirectMCtwist}
Let $h$ and $g$ be two $L_{\infty}$-algebras such that $h$ acts on $g$ by derivations as above.
Given a Maurer-Cartan element $a$ of $g$, the Maurer-Cartan twisting $(g\rtimes h)^{(a,0)}$ is the twisted semi-direct product $g^a\rtimes_{\xi}h$ for $\xi(x)=a.x$.
\end{lem}

We conclude:
\begin{thm}\label{T:twistedsemidirprod}
There are equivalences of $L_{\infty}$-algebras
\[
Lie(\underline{haut_{P_{\infty}-Alg}}(X,\varphi)) \simeq g_{P^+,X}^{\varphi^+} \simeq g_{P,X}^{\varphi}\rtimes_{\xi} End(X),
\]
where the twist $\xi:End(X)[-1]\rightarrow g_{P,X}^{\varphi}$ is the tangent homotopy action of $End(X)$ on $\varphi$, given by:
\begin{itemize}
\item[(i)] the Maurer Cartan element $(\varphi,0)$ as in Lemma \ref{L:LiesemidirectMCtwist}.(2) when $P$ is a Koszul dg properad (in this case the convolutions $L_{\infty}$-algebras are dg Lie algebras);

\item[(ii)] the Maurer Cartan element $(\varphi,0)$ as in  Lemma \ref{L:semidirectMCtwist} when $P$ is any dg properad.
\end{itemize}
\end{thm}
\begin{proof}
The first equivalence is given by Theorem \ref{T:Def+=hAut}. Then, by definition, the properad morphism $\varphi^+$ sends the generator $u$ of Definition \ref{D:plusconstruction} to $0$, so under twisting by the Maurer-Cartan element $\varphi^+$, the filtered quasi-isomorphism of $L_{\infty}$-algebras of Lemma \ref{L:trivialsemidirect} becomes
\[
g_{P^+,X}^{\varphi^+} \simeq (g_{P,X}^0\rtimes End(X))^{(\varphi,0)}.
\]
Here we use the fact that Maurer-Cartan twistings of filtered quasi-isomorphisms of $L_{\infty}$-algebras are still quasi-isomorphisms, see for example \cite[Proposition 3.8]{Yal1}.
Then, the twisting of $g_{P,X}^0\rtimes End(X)$ by the Maurer-Cartan element $(\varphi,0)$ is the twisted semi-direct product $g_{P,X}^{\varphi}\rtimes_{\xi} End(X)$ where $\xi$ is defined by Lemma \ref{L:LiesemidirectMCtwist} in the dg Lie case and Lemma \ref{L:semidirectMCtwist} in the $L_{\infty}$-case.
\end{proof}

\begin{rem}
For a Koszul properad $P$, the dg Lie algebra $g_{P,X}^{\varphi}$ is obtained by twisting the convolution dg Lie algebra $g_{P,X}$ whose Lie bracket comes from a pre-Lie bracket. Moreover, the Lie bracket of $End(X)$ is the commmutator of an associative product so it comes also from a pre-Lie bracket. On the left hand side of Theorem \ref{T:twistedsemidirprod}, one considers the homotopy Lie structure associated to a grouplike $A_{\infty}$-monoïd. Put together, these observations lead to wonder whether Theorem \ref{T:twistedsemidirprod} could have a lift at the level of pre-Lie algebras.
\end{rem}

\begin{example}[Strict associative algebras]\label{ex:strictassocalgebra}
 Let $(A,\varphi)$ be a (necessarily strict) associative algebra concentrated in degree $0$. Then by proposition~\ref{P:hautcplx}, one has 
 $Lie(\underline{haut}(A)) \cong Hom(A,A)$ the Lie algebra (concentrated in degree $0$) of endomorphisms of the underlying vector space of $A$.
 
It is a standard computation (\cite{Mer1, Mer2}) 
that the Lie algebra $g_{Ass}^{\varphi}=\mathfrak{L}_{\underline{Ass_{\infty}\{A\}}^{\varphi}}$ is isomorphic to the subcomplex $C^{\bullet \geq 2}(A,A) [1]=
\bigoplus_{n\geq 2} Hom(A^{\otimes n}, A)[1-n]$ of the shifted Hochschild 
cochain complex $C^\bullet(A,A)[1] = \bigoplus_{n\geq 0} Hom(A^{\otimes n}, A)[1-n]$ with Lie bracket given by the restriction of the standard Gerstenhaber complex.

Using these equivalences and Theorem~\ref{T:twistedsemidirprod}, we have that 
\[Lie(\underline{haut}_{Ass_{\infty}}(A,m)) \, \cong \,  Hom ( A^{\otimes >1}, A)[1]\rtimes_{\xi} Hom(A,A) \;  \]
As a graded vector space, note that the right hand side is the same as the underlying graded vector space of $C^{\bullet \geq 1}(A,A)[1] =   \bigoplus_{n\geq 1} Hom(A^{\otimes n}, A)[1-n]$. The graded Lie bracket in  Theorem~\ref{T:twistedsemidirprod} is given by the properadic convolution product  hence it coincides with the restriction of the usual Gerstenhaber bracket to $ C^{\bullet \geq 1}(A,A)[1]$. It remains to identify the differential which by Lemma~\ref{L:semidirectMCtwist} and the above coincides with the usual Hochschild differential on $C^{\bullet \geq 2}(A,A) $. It remains to identifiy the differential on $ Hom(A,A)$ which is given, for any $f:A\to A$, by $f\mapsto \xi(f)$. By Lemma~\ref{L:trivialsemidirect}, the map $\xi$ is induced by the action of the Maurer Cartan $\varphi$ of
$\mathfrak{L}_{\underline{Ass_{\infty}\{A\}}}$ which in the case of a strict associative algebra is nothing more than the underlying multiplication of the algebra structure. Therefore $\xi(f) = [\varphi, f]$ is the convolution of the linear map $f$ with the product which is the usual Hochschild differential on linear maps.

Therefore we have an equivalence of underlying dg-Lie algebras: 
\[Lie(\underline{haut}_{Ass_{\infty}}(A,m)) \, \cong \,  Hom ( A^{\otimes >1}, A)[1]\rtimes_{\xi} Hom(A,A) \;  \, \cong\,  C^{\bullet \geq 1}(A,A)[1]\]

\smallskip 

In particular, the moduli space $\underline{Ass_{\infty}\{A\}}^{\varphi}(\mathbb{K}[[t]])$ controls the 
algebra structures on $A [[t]]$ whose reduction modulo $t$ is the given one.

However, the  moduli space $\underline{haut}_{Ass_\infty}(A,\varphi)(\mathbb{K}[[t]])$ controls the 
algebra structures on $A [[t]]$ whose reduction modulo $t$ is the given one, up to isomorphism of algebras which are the identity modulo $t$.

In other words, 
the set of connected components of such deformations in the first case is the set of all possible deformations, while the connected component of 
the derived prestack group $\underline{haut}_{Ass_\infty}(A,\varphi)$ are the set of all possible deformations modulo the standard gauge equivalences.

Example~\ref{ex:strictassocalgebra} can be generalized to algebras concentrated in degree 0 over other operads, see Section 6.1.

\smallskip 

Note that this result extends for dg-algebras (and $A_\infty$-algebras as well), see~\ref{SS:EnHoch}. 
\end{example}
If $P_\infty$ is a quasi-free resolution of a dg operad $P$, then there is another nice explicit description of $Lie(\underline{haut}_{P_\infty}(X,\varphi))$:
\begin{cor}\label{T:identificationBhautQFreeresol} Let $P_\infty=(\mathcal{F}(s^{-1}\overline{C}),\partial)\stackrel{\sim}{\rightarrow}P$ be 
a cofibrant quasi-free resolution of an operad $P$ where $C$ is a cooperad and $(X,\varphi)$ be a $P_\infty$-algebra.  
One has an equivalence of $L_{\infty}$-algebras
\[ Lie(\underline{haut}_{P_\infty}(X,\varphi)) \cong Coder( \overline{C}(X[1]))^{D_{\varphi}} \rtimes_{\xi} End(X,X)\]
where the last term is the $L_\infty$-algebra of coderivations of the cofree coalgebra on $X[1]$ 
twisted by the Maurer Cartan element $D_{\varphi}$ (the coderivation of square zero corresponding to the $P_{\infty}$-algebra structure $\varphi$) and the action of $End(X,X)$ is given by the composition of coderivations of $C(X[1])$. The twist $\xi$ is defined by Lemma \ref{L:semidirectMCtwist}. In the case where $P$ is Koszul and $C$ is its Koszul dual, we get a twisted semi-direct product of dg Lie algebras with the twist given by Lemma \ref{L:LiesemidirectMCtwist}.
\end{cor}
\begin{proof}
Similarly to \cite[Proposition 10.1.17]{LV} we have an isomorphism of $L_{\infty}$-algebras
\[
g_{P,X} \cong Coder( \overline{C}(X[1])).
\]
Then, the $P_{\infty}$-algebra structure $\varphi$ is a Maurer-Cartan element in $g_{P,X}$, whose image under the Lie algebra isomorphism above gives a Maurer-Cartan element $D_{\varphi}$ in $Coder( \overline{C}(X[1]))$, that is, a degree $1$ coderivation of square zero. Twisting this isomorphism by those Maurer-Cartan elements gives an isomorphism
\[
g_{P,X}^{\varphi} \cong Coder( \overline{C}(X[1]))^{D_{\varphi}}.
\]
Moreover, the tangent action of $End(X)$ on $g_{P,X}^{\varphi}$  gives under this isomorphism an action of $End(X)$ on $Coder( \overline{C}(X[1]))^{D_{\varphi}}$ defined by the composition of coderivations, and the twist $\xi$ depending on this action and on the Maurer-Cartan element $D_{\varphi}$ is given by Lemma \ref{L:LiesemidirectMCtwist}.
Therefore the equivalence between the r.h.s and l.h.s in the theorem follows from Theorem~\ref{T:twistedsemidirprod}.
\end{proof}

\begin{example}
Let $A$ be 
a dg-associative algebra. We can consider the truncated and full Hochschild complexes of $A$, respectively $CH^{\bullet >0}(A)$ and $CH_*(A)$ as in~\ref{ex:strictassocalgebra}. 
The \emph{full version} controls the deformation theory of the category of $A$-modules 
which is in general another kind of formal moduli problem. Equivalently, the full Hochschild complex controls the deformations of $A$ as a \emph{curved $A_{\infty}$-algebra}~\cite{Preygel}.

However, in the case where $A$ is concentrated in degree zero, we observe that, first, deformations of $A$ are deformations as a strict Poisson algebra or as a strict associative algebra, and second, curved and uncurved deformations are equivalent. Consequently, the space of Maurer Cartan elements are the same for the truncated and the untruncated versions of ¨ Hochschild complexes.

This observation is crucial in the study of formality theorems for Poisson algebras and deformation quantization  of Poisson structures on manifolds \cite{Ko2,Tam1}. Let us fix $A=\mathcal{C}^{\infty}(\mathbb{R}^d)$ the algebra of smooth functions on $\mathbb{R}^d$, and consider two complexes. First, the full Hochschild complex $CH^*(A,A)$, second, the complex of polyvector fields $T_{poly}(A)=\left( \bigoplus_{k\geq 0}\bigwedge^kDer(A)[-k] \right)[1]$ (recall that vector fields are derivations of the ring of smooth functions). The complex of polyvector fields also forms a (shifted) Lie algebra with a Lie structure induced by the bracket of vector fields. The classical Hochschild-Kostant-Rosenberg theorem (HKR for short) states that the cohomology of $CH^*(A,A)$ is precisely $T_{poly}(A)$. However, the HKR quasi-isomorphism is not compatible with their respective Lie algebra structures. In \cite{Ko2}, Kontsevich proved that the HKR quasi-isomorphism lifts to an $L_{\infty}$-quasi-isomorphism
\[
T_{poly}(A)[1] \stackrel{\sim}{\rightarrow} CH^*(A,A)[1]
\]
by building an explicit formality morphism. An alternative proof of the formality theorem is due to Tamarkin \cite{Tam1} and provides a formality quasi-isomorphism of homotopy Gerstenhaber algebras (that is $E_2$-algebras).
Here $T_{poly}(A)$ is actually the deformation complex of the trivial Poisson algebra structure. In general, the full Hochschild complex $CH^*(A,A)$ controls deformations of $A$ as a curved algebra, but since $A$ is in degree zero, the space of Maurer Cartan elements obtained from the full Hochschild complex is the same as the one from the truncated Hochschild complex. This is important, because the formality theorem holds for the full complex but not for the truncated one. This formality theorem implies the equivalence of the associated formal moduli problems. Then, applying these moduli problems to the ring of formal power series $\mathbb{K}[[t]]$, one gets that the the set of isomorphism classes Poisson algebra structures on $A[[\hbar]]$ without constant term is in bijection with gauge equivalence classes of $*_{\hbar}$-products (that is, associative formal deformations of the product of $A$).

There are similar phenomenon for other categories of algebraic structures such as (shifted) Poisson ones. See~\ref{S:Examples} and~\ref{S:Concluding}.
\end{example}

\section{Examples}\label{S:Examples}

\subsection{Deformations of $E_n$-algebras}\label{SS:EnHoch} We now 
generalize example~\ref{ex:strictassocalgebra} to  the homotopy setting and to higher algebras, that is $E_n$-algebras.
The latter are higher generalizations of homotopy associative algebras and form a hierarchy of ``more and more'' commutative and homotopy associative structures, interpolating between homotopy associative or $A_\infty$-algebras (that is $E_1$-algebras) and $E_{\infty}$-algebras. 

Algebras governed by $E_n$-operads and their deformation theory play a prominent role in a variety of topics such as the study 
of iterated loop spaces, Goodwillie-Weiss calculus for embedding spaces, deformation quantization of Poisson manifolds and Lie bialgebras, factorization homology and derived symplectic/Poisson geometry~\cite{Ko1, Ko2, Lur0, Lur2, CPTVV, FG, Fra, Fre5, GTZ, Hin, Kap-TFT, KoSo, May, Preygel, Tam1, Toen-ICM}. 


To define $E_n$-algebras, one first note that the  
configuration spaces of (rectilinear embeddings of) $n$-disks into a bigger $n$-disk gather into a topological operad $D_n$,
called the little $n$-disks operad. 
An $E_n$-operad (in chain complexes) is a dg-operad quasi-isomorphic to the singular chains $C_*(D_n)$ of the little $n$-disks operad. 
There is an $\infty$-functor from $E_n$-algebras to $L_\infty$-algebras whose composition with the forgetful functor to chain complexes 
is the shift $X\mapsto X[1-n]$.

Given an ordinary associative  algebra $A$, its endomorphisms $Hom_{biMod_A}(A,A)$ in the category $biMod_A$ 
of $A$-bimodules is isomorphic to the center $Z(A)$ of $A$. Deriving this hom object gives the Hochschild cochain complex 
$C^*(A,A)\cong \mathbb{R}Hom_{biMod_A}(A,A)$ of $A$,
and the associated Hochschild cohomology $HH^*(A,A)$ of $A$ satisfies $HH^0(A,A)=Z(A)$. For higher structures, one has the following definition 
(see~\cite{Fra, Lur2, GTZ}).
\begin{defn}\label{D:HochCochain}
The (full) Hochschild complex of an $E_n$-algebra $A$, computing its higher Hochschild cohomology,
is the derived hom $C^*_{E_n}(A,A)=\mathbb{R}Hom^{E_n}_A(A,A)$ in the category of (operadic) $A$-modules\footnote{note that the operadic 
$E_1$-module are precisely the bimodules} over $E_n$.
\end{defn} 
The Deligne conjecture endows the Hochschild cochain complex with an $E_{n+1}$-algebra structure~\cite[Theorem 6.28]{GTZ} 
or~\cite{Fra, Lur2}.  Associated to an $E_n$-algebra $A$, one also has its cotangent complex $L_A$, 
which classifies square-zero extensions of $A$~\cite{Fra, Lur2}. 
\begin{defn}[\cite{Fra}]\label{D:HochTangent} The tangent complex $T_A$ of an $E_n$-algebra $A$ is the dual
$T_A:= Hom^{E_n}_A(L_A,A)\cong \mathbb{R}Der(A,A)$.
\end{defn}
The latter isomorphism gives a $L_\infty$-structure to $T_A$ and 
Francis~\cite{Fra, Lur2} has proved that $T_A[-n]$ has a canonical structure of $E_{n+1}$-algebra (lifting the $L_\infty$-structure). 
He further proved 
that there is a fiber sequence \[T_A[-n]\to  CH^*_{E_n}(A,A) \to A\] where the first map is a map of $E_{n+1}$-algebras.

\smallskip 

A corollary of our theorem~\ref{T:Def+=hAut} is the following operadic identification of the tangent complex $T_A$ of
an $E_n$-algebra~(\ref{D:HochTangent}):
\begin{cor}\label{L:gE2+=TA}
The $E_n$-Hochschild tangent complex $T_A$ of an $E_n$-algebra $A$ is naturally weakly equivalent as an $L_\infty$-algebra   to $g_{E_n^+,A}^{\varphi^+}$:
\[
T_A\simeq Lie(\underline{haut}_{E_n}(A, \varphi))\simeq g_{E_n^+,A}^{\varphi^+},
\]
where $\varphi^+$ is the $E_n^+$-algebra structure on $A$ trivially induced by its $E_n$-algebra structure $\varphi:E_n\rightarrow End_A$ as above, and $\underline{haut}_{E_n}(A)$ is the derived prestack group of homotopy automorphisms of $A$ as an $E_n$-algebra.
\end{cor}
\begin{proof}
According to \cite[Lemma 4.31]{Fra}, the homotopy Lie algebra of homotopy automorphisms $Lie(\underline{haut}_{E_n}(A, \varphi))$
is equivalent to the tangent complex $T_A$ of $A$. Hence Theorem~\ref{T:Def+=hAut} implies the corollary.
\end{proof}
In particular, Theorem~\ref{T:Def+=hAut} implies that 
the tangent complex $T_A$ of an $E_n$-algebra splits as a semi-direct product of $End(A)$ with the operadic deformation complex of $A$ as an $E_n$-algebra.

\begin{rem}
The deformation problem of an $E_n$-algebra is an $E_{n+1}$-moduli problem and the $E_{n+1}$-algebra structure on $T_A[-n]$ is the one induced by Lurie's correspondence between augmented $E_{n+1}$-algebras and formal $E_{n+1}$-moduli problems. Shifting it to get an $L_{\infty}$-algebra structure on $T_A$, we obtain the $L_{\infty}$-algebra structure above, which corresponds to the formal commutative moduli problem associated to this $E_{n+1}$-moduli problem. Corollary \ref{L:gE2+=TA} gives an operadic explicit description of this formal commutative moduli problem.

The relation between $E_n$-moduli problems and commutative moduli problems is given by a commutative square of $\infty$-functors
\[
\xymatrix{
E_{n+1}-FMP_{\mathbb{K}}\ar[r]\ar[d] & FMP_{\mathbb{K}}\ar[d] \\
E_{n+1}-Alg_{\mathbb{K}}^{aug}\ar[r] & dgLie_{\mathbb{K}}
}
\]
where the vertical arrows are the equivalences of Lurie's correspondence theorem, the upper horizontal arrow is the forgetful functor induced by precomposition with the forgetful functor from commutative algebras to $E_n$-algebras, and the lower horizontal arrow associates to an augmented $E_{n+1}$-algebra $R$ with augmentation ideal $m_R$ a dg Lie algebra $m_R[n]$ \cite[Remark 4.0.10]{Lur0}.

By \cite[Theorem 4.0.8]{Lur0}, given an augmented $E_{n+1}$-algebra $R$, the complex $T[-n-1]$ (where ($T$ is the tangent complex functor) of the corresponding $E_{n+1}$-moduli problem is equivalent to the augmentation ideal $m_R$ of $R$. By the square above, one recovers that the dg Lie algebra of the associated formal commutative moduli problem is equivalent to the shifted tangent complex $T[-1]$.
\end{rem}

\subsection{Deformation complexes of $Pois_n$-algebras}
\label{SS:Poissn}
We now introduce  Tamarkin deformation complexes of a $Pois_n$-algebra~\cite{Ta-deformationofd-algebra}
and prove that these complexes \emph{do control deformations} of (dg-)$Pois_n$-algebras.

\smallskip

We denote by $Pois_n$ the operad of $Pois_n$-algebras and $uPois_n$ the operad of unital $Pois_n$-algebras.

\smallskip

Let $A$ be a dg $Pois_n$-algebra, with structure morphism $\varphi:Pois_n\rightarrow End_A$.  We denote by $CH_{Pois_n}^*(A,A)$ its $Pois_n$-Hochschild cochain complex, also referred to 
as its $Pois_n$-deformation complex as defined by Tamarkin~\cite{Ta-deformationofd-algebra} and Kontsevich~\cite{Ko1}.
Following
Calaque-Willwacher~\cite{CaWi}, we note that this complex is given by the suspension
\begin{equation}\label{eq:CHPoisDef}
 CH_{Pois_n}^{*}(A,A) \, := \, Hom_{\Sigma}(u {Pois_n}^* \{n\}, End_A)[-n] 
\end{equation}
of the underlying chain complex of the convolution Lie algebra. 
Here  $(-)^*$ is the linear dual and $\{n\}$ is the operadic $n$-iterated suspension.
The inclusion of $Pois_n$ in $uPois_n$ induces a splitting (as a graded space)
\begin{equation}\label{eq:CHPoisDefSplit}
CH_{Pois_n}^{*}(A,A) \, \cong \, A \oplus Hom_{\Sigma}({Pois_n}^* \{n\}, End_A)[-n]
\end{equation}
and also gives rise to 
the \emph{truncated} deformation complex
\begin{equation}\label{eq:CHPoisDefTruncated} CH_{Pois_n}^{(\bullet>0)}(A,A) =  Hom_{\Sigma}({Pois_n}^* \{n\}, End_A)[-n] 
\end{equation}
obtained by deleting the \lq\lq{}unit part\rq\rq{} $A$,
which is more relevant to deformations of $Pois_n$-algebras\footnote{as opposed to deformation of categories of modules}, 
see Lemma~\ref{L:gPoiss=Tam}. Note that both complexes are naturally bigraded with respect to the internal grading of $A$ 
and the \lq\lq{}operadic\rq\rq{} grading coming from $u {Pois_n}^*$.
The notation $CH_{Pois_n}^{(\bullet>0)}(A,A)$ is there to suggest that we are taking the 
subcomplex with positive weight with respect to the operadic grading.

\smallskip

The suspensions $CH_{Pois_n}^{*}(A,A) [n]$ and $CH_{Pois_n}^{(\bullet>0)}(A,A)[n] $ have \emph{canonical 
$L_\infty$-structures} since they are convolution algebras, and $CH_{Pois_n}^{(\bullet>0)}(A,A)[n]$ is canonically a sub $L_{\infty}$-algebra of $CH_{Pois_n}^{*}(A,A)[n]$.
Tamarkin~\cite{Ta-deformationofd-algebra}
(see also~\cite{Ko1, CaWi}) proved that \emph{the complex $CH_{Pois_n}^{*}(A,A)$ actually inherits a (homotopy) $Pois_{n+1}$-algebra 
structure} lifting this $L_{\infty}$-structure. 
Further, by~\eqref{eq:CHPoisDefSplit} we have an exact sequence of cochain complexes
\begin{equation}
 \label{eq:fibersequencePoisn}
 0\longrightarrow CH_{Pois_n}^{(\bullet>0)}(A,A)  \longrightarrow CH_{Pois_n}^{*}(A,A) \longrightarrow A \longrightarrow 0
\end{equation}
which yields after suspending the exact triangle
\begin{equation}
 \label{eq:fibersequencePoisnasLie}
  A[n-1] \stackrel{\partial_{Pois_n}[n-1]}\longrightarrow CH_{Pois_n}^{(\bullet>0)}(A,A)[n]  \longrightarrow CH_{Pois_n}^{*}(A,A)[n]. 
\end{equation}
\begin{rem}\label{R: hofibPois}
The map $\partial_{Pois_n}: A \subset CH_{Pois_n}^{*}(A,A)\to CH_{Pois_n}^{(\bullet>0)}(A,A)$ is the part of the 
differential in the cochain complex
$CH_{Pois_n}^{*}(A,A) =A \oplus CH_{Pois_n}^{(\bullet>0)}(A,A)$ which comes from the operadic structure. 
That is $\partial_{Pois_n}(x) \in Hom(A,A)$ is the map $a\mapsto \pm [x,a]$ where the bracket is the bracket of the $Pois_n$-algebra.
The Jacobi identity for the Lie algebra $A[n-1]$ implies that the sequence~\eqref{eq:fibersequencePoisnasLie} is a sequence of
$L_\infty$-algebras.
\end{rem}
\begin{rem}
 The operad $Pois_n$ is denoted $e_n$ in \cite{CaWi, Ta-deformationofd-algebra} and the complex
 $CH_{Pois_n}^{*}(A,A)$ is simply denoted $def(A)$ in Tamarkin~\cite{Ta-deformationofd-algebra}. We prefer to use the notations we have introduced by analogy with (operadic) Hochschild complexes.
\end{rem}
The next lemma compares the $L_\infty$-algebra structure of the truncated $Pois_n$ Hochschild complex and the 
one associated to the derived prestack group of homotopy automorphisms of a $Pois_n$-algebra:
\begin{lem}\label{L:gPoiss=Tam}
Let $A$ be a dg $Pois_n$-algebra with structure map $\varphi:Pois_n\rightarrow End_A$. There is a quasi-isomorphism dg Lie algebras
$$ g_{Pois_n^+,A}^{\varphi^+} \simeq CH_{Pois_n}^{(\bullet>0)}(A,A) $$
where the right hand side is the truncated cochain complex of a $Pois_n$-algebra defined by Tamarkin as above.
\end{lem}

\begin{proof}
The situation is completely similar to Example \ref{ex:strictassocalgebra}.
Here the resolution of $Pois_n$ we consider is the minimal model of of $Pois_n$ given by
\[
Pois_{n\infty}=\mathcal{B}^{co}(Pois_n^*\{n\})=(\mathcal{F}(\overline{Pois_n^*\{n+1\}}),\partial)
\]
where $(-)^*$ is the linear dual, $\{n\}$ is the operadic $n$-iterated suspension, $\mathcal{B}^{co}$ is the operadic cobar construction and $\overline{-}$ is the coaugmentation ideal of a coaugmented cooperad.
The result follows from the Koszul case in Theorem \ref{T:twistedsemidirprod}, with the twist defined by Lemma \ref{L:LiesemidirectMCtwist}. This result implies indeed that
\begin{eqnarray}
g_{Pois_n^+,A}^{\varphi^+} & \simeq &  g_{P,A}^{\varphi}\rtimes_{\xi} Hom(A,A)\\
& = & Hom_{\Sigma}(\overline{Pois_n^*\{n\}}\oplus I,End_A)^{\varphi} \\
 & = & Conv(Pois_n^*\{n\},End_A)
\end{eqnarray}
where $Conv(Pois_n^*\{n\},End_A)$ is the convolution Lie algebra of \cite[Section 2.2]{CaWi}.
In particular the last equality is an equality of dg Lie algebras with the convolution bracket defined by the action of the infinitesimal cooperadic coproduct on both sides.
\end{proof}

Lemma~\ref{L:gPoiss=Tam} together with Theorem~\ref{T:Def+=hAut} implies that
\begin{cor}\label{C:Tamarkin=Defo}
The truncated Tamarkin deformation complex $CH_{Pois_n}^{(\bullet>0)}(A,A)$ controls deformations
of $A$ into the $\infty$-category of dg $Pois_n$-algebras, in other words is the tangent Lie algebra of the derived prestack group $\underline{haut}_{Pois_{n\infty}}(A)$, where $Pois_{n\infty}$ is a cofibrant resolution of $Pois_n$. 
\end{cor}

\begin{rem}\label{R:VariousDefPoisn}
The proof of Lemma~\ref{L:gPoiss=Tam} also shows that the deformation complex 
$g_{Pois_n,A}^{\varphi}$ of the formal moduli problem $\underline{{Pois_{n}}_{\infty} \{A \}}^{\varphi}$ is given by the $L_\infty$-algebra 
$CH_{Pois_n}^{(\bullet>1)}(A,A)[n]$, which is the kernel  
\begin{equation}\label{eq:DDefPoisnPas+} CH_{Pois_n}^{(\bullet>1)}(A,A)[n] := \ker\Big(CH_{Pois_n}^{(\bullet>0)}(A,A)[n]\twoheadrightarrow Hom (A, A)[n]\Big) 
\end{equation}
and is thus a even further truncation of $CH_{Pois_n}^*(A,A)$.
The situation is thus similar to what happens in deformation theory of associative algebras.
\end{rem}

One can also wonder which deformation problem is controled by the full complex $CH_{Pois_n}^*(A,A)$. In  view of our results and classical results on deformation theory of $E_n$-algebras (\cite{KellerLowen, Preygel, Fra}), we make the following

\begin{conj}
Let $n\geq 2$ and let $A$ be an $n$-Poisson algebra. The $L_{\infty}$-algebra structure
of the full shifted Poisson complex $CH_{Pois_n}^*(A)[n]$ controls the deformations of $Mod_A$ into  $E_{n-1}$-monoidal dg categories.
\end{conj}

Here, when $n\leq 1$, some  shift is needed on the linear enrichment of the $E_{|n-1|}$-monoidal dg-category,  according to the red shift trick~\cite{Toen-ICM, Toe}.

This conjecture is deeply related to the deformation theory of shifted Poisson structures in derived algebraic geometry, in the sense of \cite{CPTVV}. Precisely, if $X$ is a derived Artin stack locally of finite presentation and equipped with an $n$-shifted Poisson structure, then its sheaf of principal parts (which controls the local deformation theory on $X$ and whose modules describe the quasi-coherent complexes over $X$) forms a sheaf of mixed graded $Pois_{n+1}$-algebras. The deformation theory of the category of quasi-coherent complexes should then be controled by a full shifted Poisson complex. 

\subsection{Bialgebras}

Let us conclude our series of examples with one of properadic nature. Here we are interested in associative and coassociative bialgebras, and refer the reader to Example~\ref{E:asscoass} for a precise definition as well as the construction of the corresponding properad $Bialg$.

What we call the \emph{Gerstenhaber-Schack complex} is the total complex of a bicomplex, defined by
\begin{equation}\label{eq:DefCGS}
C_{GS}^*(B,B)\cong \prod_{m,n\geq 1}Hom_{dg}(B^{\otimes m},B^{\otimes n})[2-m-n].
\end{equation}
The horizontal differential is defined, for every $n$, by the Hochschild differential associated to the Hochschild complex of $B$ seen as an associative algebra with coefficients in the $B$-bimodule $B^{\otimes n}$.
The vertical differential is defined, for every $m$, by the co-Hochschild differential associated to the co-Hochschild complex of $B$ seen as a coassociative coalgebra with coefficients in the $B$-bicomodule $B^{\otimes m}$.
The compatibility between these differentials, which gives us a well defined bicomplex,
follows from the distributive law relating the product and the coproduct of the bialgebra $B$ (see~\cite{GS, Mer2} for details).

The properadic deformation complex of a properad morphism $\varphi:Bialg_{\infty}\rightarrow End_B$ gives
\[
g_{Bialg_{\infty},B}^{\varphi}\cong \prod_{m,n\geq 1,m+n\geq 3}Hom_{dg}(B^{\otimes m},B^{\otimes n})[2-m-n].
\]
It is an $L_{\infty}$-algebra with differential also given by the Gerstenhaber-Schack differential, but without the term $Hom_{dg}(B,B)$.

The existence of an $L_{\infty}$-algebra structure on both complexes follows from their interpretation as properadic deformation complexes, one using $Bialg_{\infty}$ \cite[Theorem 25]{MV2}, the other using $Bialg_{\infty}^+$ (as explained in \cite{Mer2}).

Combining Theorem~\ref{T:Def+=hAut} with the computation of \cite{Mer2}, we get:
\begin{thm}\label{T:GSIdentific}
The Gerstenhaber-Schack $L_{\infty}$-algebra is quasi-isomorphic, as an $L_{\infty}$-algebra, to the $L_{\infty}$-algebra controlling the deformations of dg bialgebras up to quasi-isomorphisms:
$$C_{GS}^*(B,B) \cong  g_{Bialg_{\infty}^+,B}^{\varphi^+} \simeq Lie(\underline{haut}_{Bialg_{\infty}}(B))$$.
\end{thm}
Hence the Gerstenhaber-Schack complex is indeed the $L_{\infty}$-algebra controling the derived deformation theory of dg bialgebras in a precise meaning, something new since the introduction of this complex by Gerstenhaber and Schack in their seminal paper \cite{GS}.
Moreover, as emphasized by the results of \cite{Mer2} and \cite{GY}, this complex plays a crucial role in deformation quantization.

\section{Concluding remarks and perspectives} \label{S:Concluding}

To conclude, let us give an overview of the various deformation complexes considered in the litterature and their \emph{derived} and \emph{underived} formal moduli problems.

\subsection{Algebras over operads in vector spaces}
\label{SS:DefOperad}
Let $X$ be a vector space and $P$ an operad with Koszul dual $C$. Then the cohomological grading on the convolution Lie algebra $g_{P,X}=Hom_{\Sigma}(C,End_X)$ is entirely determined by the ``weight grading'' of operations in the cooperad $C$. In particular, in the case where $X$ is of finite dimension $n$, this means that the degree $0$ Lie subalgebra of $g_{P,X}$ is nothing but $gl(X)$, whose associated Lie group is the general Lie group $GL(X)$. This is the gauge group acting on the Maurer-Cartan elements of $g_{P,X}$, so that the moduli set of Maurer-Cartan elements is
\[
\mathcal{MC}(g_{P,A})=Mor(P,End_A)/GL(A).
\]
The deformation complex $g_{P,X}^{\varphi}$ of a $P$-algebra $A=(X,\varphi)$ controls then the deformations of $A$ \emph{as a $P$-algebra}, up to \emph{linear automorphisms of $A$}.
If we replace $C$ by $\overline{C}$ in the definition of $g_{P,X}$, which is what we did in the present paper, then there is no non trivial gauge group acting anymore, and the Maurer-Cartan moduli set is just
\[
\mathcal{MC}(g_{P,A})=Mor(P,End_A).
\]
The corresponding \emph{underived} formal moduli problem, or classical deformation functor, controls the deformations of $A$ as a $P$-algebra in the category of vector spaces, up to isomorphisms.

In the derived setting, one replaces Artinian algebras by their dg enhancement, so that the simple description above in terms of gauge group action does not exist anymore (notice that, although $V$ is in degree $0$, we have to consider $haut(V\otimes A)$ for any \emph{differential graded} local Artinian algebra $A$, and $V\otimes A$ is not in degree zero anymore). The relevant theory, described in Section~\ref{S:Plus}, defines the appropriate deformation problems as loops over the homotopy quotient of the moduli space of $P$-algebra structures by a homotopy automorphisms group. Moreover, it turns out, as we explained in Remark \ref{R:VectStack}, that the corresponding derived formal moduli problem is given by the formal completion $\widehat{\underline{\mathcal{N}wP-Alg}}_A$ at $A$ of the $n$-geometric derived Artin stack of $n$-dimensional $P$-algebras.

To clarify the link between our \emph{derived} construction and the \emph{underived} deformation functor described above, let us restrict our derived moduli problem to local Artinian algebras. In this context, provided that $P$ is an operad in vector spaces, the simplicial presheaves $\underline{P_{\infty}\{X\}}$ and $\underline{haut}(X,\varphi)$ are actually discrete. Indeed, given a local Artinian algebra $R$, each vertex of the Kan complex $P_{\infty}\{X\otimes m_R\}$ factors uniquely through the composite
\[
P_{\infty}\twoheadrightarrow P\rightarrow End_{V\otimes m_R}
\]
because of degree reasons, commutation with the differentials and the fact that $End_{X\otimes m_R}$ is concentrated in degree zero. Moreover, for the same reasons, a homotopy between two such maps cannot be anything else than the identity, so that finally
\[
P_{\infty}\{X\otimes m_R\}\cong Mor_{prop}(P, End_{X\otimes m_R}) = MC(g_{P,X})
\]
and the corresponding pointed functor over $(X,\varphi)$ is
\[
P_{\infty}\{X\otimes m_R\}^{\varphi} \cong MC(g_{P,X}^{\varphi}).
\]
Note that this is coherent with the fact that left homotopies between such maps are in bijection with $\infty$-isotopies of $(X,\varphi)$, which boils down to $Id_{(X,\varphi)}$ when $X$ is in degree zero.
Also for degree reasons, the simplicial presheaf $\underline{haut}(X,\varphi)$ is equivalent to the discrete presheaf defined, for each local Artinian algebra $R$, by the strict automorphism group $Aut(X\otimes m_R)$, that is, the algebraic group of automorphisms of $(X,\varphi)$. The homotopy action of $\underline{haut}(X,\varphi)$ on $\underline{P_{\infty}\{X\}}$ is then nothing but the gauge action described above, so that the semi-direct product $g_{P^+,X}^{\varphi^+}\simeq \simeq g_{P,X}^{\varphi}\rtimes End(X)$ becomes the dg convolution Lie algebra considered in \cite[Section 12.2.22]{LV}.

\subsection{Differential graded algebras over operads}\label{SS:DefOperad2}

A differential graded structure on the object $X$ carries non trivial homotopies, and taking into account this new homotopy data that do not exist in the degree zero case involves deformations of $P$-algebra structures into $P$-algebra structures \emph{up to homotopy}, that is $P_{\infty}$-algebra structures, and therefore taking into account the non trivial homotopy type of the moduli space $P_{\infty}\{X\}$. Given a $P_{\infty}$-algebra $A=(X,\varphi)$, there are a priori three possible variants of derived deformation problems one could look at:
\begin{itemize}
\item[(1)] Deformation theory of the operad morphism $\varphi:P_{\infty}\rightarrow End_A$;

\item[(2)] Deformation theory of $A$ in the $\infty$-category of $P_{\infty}$-algebras \emph{up to $\infty$-isotopies};

\item[(3)] Deformation theory of $A$ in the $\infty$-category of $P_{\infty}$-algebras \emph{up to quasi-isomorphisms}.
\end{itemize}
Problem (1) is, as we saw before, controled by the derived formal moduli problem
\[
\underline{P_{\infty}\{X\}}(R) = hofib_{\varphi}(P_{\infty}\{X\otimes R\}\rightarrow P_{\infty}\{X\})
\]
whose associated $L_{\infty}$-algebra is $g_{P,X}^{\varphi}$ (constructed with $\overline{C}$). Problem (2) is the setting in which \cite[Section 12.2.22]{LV} takes place: an $R$-deformation of a $P$-algebra $A$ in the sense of (2) is a an $R$-linear $P_{\infty}$-algebra $\tilde{A}\simeq A\otimes R$ with a $\mathbb{K}$-linear $P_{\infty}$-algebra $\infty$-isomorphism $\tilde{A}\otimes_R\mathbb{K}\stackrel{\sim}{\rightarrow}A$, where $(-)\otimes_R\mathbb{K}$ is defined by the augmentation of $R$. Two deformations are equivalent if they are related by an $R$-linear $\infty$-isomorphism whose restriction modulo $m_R$ is the identity, that is $\infty$-isotopies. It turns out that, in the operadic case, problems (1) and (2) are equivalent: by \cite[Theorem 5.2.1]{Fre-cyl} homotopies between morphisms from $P_{\infty}$ to $End_X$ are in bijection with $\infty$-isotopies between the corresponding $P_{\infty}$-algebras, and by \cite[Section 12.2.22]{LV} the later are also controled by the convolution $L_{\infty}$-algebra. Here the gauge group of the deformation complex (for this moduli problem) of a $P_{\infty}$-algebra $A$ is isomorphic to the group of $\infty$-isotopies of $A$.

We spent some time in this article dealing with Problem (3), which had previously no known construction in the framework of derived deformation theory. As explained before, an $R$-deformation of a $P$-algebra $A$ in the sense of (3) is a an $R$-linear $P_{\infty}$-algebra $\tilde{A}\simeq A\otimes R$ with a $\mathbb{K}$-linear $P_{\infty}$-algebra quasi-isomorphism $\tilde{A}\otimes_R\mathbb{K}\stackrel{\sim}{\rightarrow}A$, where $(-)\otimes_R\mathbb{K}$ is defined by the augmentation of $R$.
We built a \emph{derived formal group} $\widehat{\underline{haut}_{P_{\infty}}(A)}_{id}$ whose corresponding $L_{\infty}$-algebra admits two equivalent descriptions
\[
Lie(\widehat{\underline{haut}_{P_{\infty}}(A)}_{id})\simeq g_{P,X}^{\varphi}\rtimes_{\xi} End(X)\simeq g_{P^+,X}^{\varphi^+}
\]
where the middle one exhibits this moduli problem as originating from the homotopy quotient of the space of $P_{\infty}$-algebra structures on $X$ by the homotopy action of self-quasi-isomorphisms $haut(X)$, and the right one explains how one can encode this explicitely as simultaneous compatible deformations of the $P_{\infty}$-algebra structure \emph{and} the differential of $X$.

Another way to compare deformation problems (2) and (3) is to recall that there are equivalences of $\infty$-categories
\[
P_{\infty}-Alg[W_{qiso}^{-1}]\simeq P-Alg[W_{qiso}^{-1}]\simeq \infty-P_{\infty}-Alg[W_{\infty-qiso}^{-1}]
\]
where the first equivalence is induced by the operadic quasi-isomorphism $P_{\infty}\overset{\sim}{\rightarrow}P$, and the second equivalence is induced by the strictification theorem of \cite[Chapter 12]{LV}, the later $\infty$-category being the one of $P_{\infty}$-algebras with $\infty$-morphisms, and with $\infty$-quasi-isomorphisms as weak equivalences. Problem (3) concerns deformation theory in the $\infty$-category of $P_{\infty}$-algebras \emph{up to $\infty$-quasi-isomorphisms}, hence is a relaxed version of Problem (2) in this sense.

\subsection{Dg coalgebras over an operad}

For coalgebras over an operad $P$, one can make sense of the three aforementioned derived deformation problems as follows:
\begin{itemize}
\item[(1)] A $P$-coalgebra structure on a complex $X$ is the datum of an operad morphism $P\rightarrow coEnd_X$ where $coEnd_X$ is the ``operad of coendomorphisms'' of $X$, given in each arity $n$ by $Hom(X,X^{\otimes n})$. One then considers the deformations of this operad morphism as before.

\item[(2)] The $\infty$-morphisms of $P$-coalgebras are well defined via the cobar functor as constructed in \cite[Section 9.1]{LGL} and behaves like $\infty$-morphisms of algebras, in particular it makes sense to consider deformations of a $P$-coalgebra with respect to $\infty$-isotopies.

\item[(3)]One works in the $\infty$-categorical localization of $P$-coalgebras with respect to quasi-isomorphisms, which admits an explicit description as the localization of a model category, and consider the same problem as before.
\end{itemize}
Derived deformation problems (1) and (2) are still equivalent, and the methods and results of this paper still apply to (3).
\begin{rem}
For conilpotent coalgebras over a cooperad $C$, the situation turns out to be a bit different. First, problem (1) cannot make sense anymore because a general operad or cooperad morphism does not impose conilpotency assumptions on the corresponding coalgebra. Second, for problems (2) and (3) one may want to specify (typically in the situation where $C$ is the Koszul dual cooperad of an operad $P$) whether we consider the $\infty$-categorical localization of conilpotent $C$-coalgebras with respect to quasi-isomorphisms or with respect to $\mathcal{B}^{co}$-equivalences (morphisms whose image by the cobar functor $\mathcal{B}^{co}$ is a quasi-isomorphism of $P$-algebras). In the first case, one would have to analyze how to implement conilpotency assumptions in the methods developed in this paper.
In the second case however, one applies Theorem \ref{T:equivhautLie} to the bar-cobar equivalence of $\infty$-categories between conilpotent $C$-coalgebras and $P$-algebras (actually presented by a Quillen equivalence of model categories) to go back to the algebra case.
\end{rem}

\subsection{Algebras over properads}

There is no well defined (homotopy invariant) notion of $\infty$-morphism of algebras over properads at present, though recent progresses have been made in~\cite{HLV}. So problem (2) does not make sense anymore in this more general setting. However, as we proved in the previous sections, problems (1) and (3) can be properly formalized and explicitely described by means of homotopy theory and derived algebraic geometry methods. One of the main additional difficulties when passing from operads to properads is the absence of model category structure on the corresponding kinds of algebras, which makes the situation more subtle to deal with both from the viewpoints of $\infty$-category theory and derived algebraic geometry.

\section{Appendix: recollections on props, homotopical algebra and $\infty$-categories}

The goal of this appendix is to briefly review several key notions and results from model categories and props that will be used in the present paper, as well as their homotopical algebra and associated $\infty$-categories.   

\subsection{Symmetric monoidal categories over a base category}

Symmetric monoidal categories over a base category formalize how a given symmetric monoidal category can be tensored and enriched over another category, in a way compatible with the monoidal structure:
\begin{defn}
Let $\mathcal{C}$ be a symmetric monoidal category. A symmetric monoidal category over $\mathcal{C}$ is a symmetric monoidal category $(\mathcal{E},\otimes_{\mathcal{E}},1_{\mathcal{E}})$ endowed with a symmetric monoidal functor $\eta:\mathcal{C}\rightarrow\mathcal{E}$, that is, an object under $\mathcal{C}$ in the $2$-category of symmetric monoidal categories.

This defines on $\mathcal{E}$ an external tensor product $\otimes :\mathcal{C}\times\mathcal{E}\rightarrow\mathcal{E}$
by $C\otimes X = \eta(C)\otimes_{\mathcal{E}} X$ for every $C\in\mathcal{C}$ and $X\in\mathcal{E}$.
This external tensor product is equipped with the following natural unit, associativity and symmetry isomorphisms:

(1) $\forall X\in\mathcal{E},1_{\mathcal{C}}\otimes X\cong X$,

(2) $\forall X\in\mathcal{E},\forall C,D\in\mathcal{C},(C\otimes D)\otimes X\cong C\otimes (D\otimes X)$,

(3) $\forall C\in\mathcal{C},\forall X,Y\in\mathcal{E},C\otimes (X\otimes Y)\cong(C\otimes X)\otimes Y\cong X\otimes(C\otimes Y)$.

The coherence constraints of these natural isomorphisms (associativity pentagons, symmetry hexagons and unit triangles which mix both internal and external tensor products) come from the symmetric monoidal structure of the functor $\eta$.

We will implicitly assume throughout the paper that all small limits and small
colimits exist in $\mathcal{C}$ and $\mathcal{E}$, and that each of these categories admit an internal hom bifunctor.
We suppose moreover the existence of an external hom bifunctor $Hom_{\mathcal{E}}(-,-):\mathcal{E}^{op}\times\mathcal{E}\rightarrow\mathcal{C}$ satisfying an adjunction relation
\[
\forall C\in\mathcal{C},\forall X,Y\in\mathcal{E},Mor_{\mathcal{E}}(C\otimes X,Y)\cong Mor_{\mathcal{C}}(C,Hom_{\mathcal{E}}(X,Y))
\]
(so $\mathcal{E}$ is naturally an enriched category over $\mathcal{C}$).
\end{defn}
Throughout this paper we will deal with symmetric monoidal categories equipped with a model structure. We assume that the reader is familiar with the basics of model categories. We refer to to Hirschhorn \cite{Hir} and Hovey \cite{Hov} for a comprehensive treatment of homotopical algebra. We just recall the axioms of symmetric monoidal model categories formalizing the interplay between the tensor and the model structures (in a word, these conditions ensure that the tensor product forms a Quillen bifunctor). From the point of view of $\infty$-categories, if a model category is equipped with a compatible symmetric monoidal structure (that is, satisfying the conditions below), then its associated $\infty$-category is symmetric monoidal as well (as an $\infty$-category).
\begin{defn}
(1) A symmetric monoidal model category is a symmetric monoidal category
$\mathcal{C}$ equipped with a model category structure such that
the following axioms holds:

\textbf{MM0.} For any cofibrant object $X$ of $\mathcal{C}$, the map $Q1_{\mathcal{C}}\otimes X\rightarrow 1_{\mathcal{C}}\otimes X\cong X$ induced by a cofibrant resolution $Q1_{\mathcal{C}}\rightarrow 1_{\mathcal{C}}$ of the unit $1_{\mathcal{C}}$ is a weak equivalence.

\textbf{MM1.} The pushout-product $(i_{*},j_{*}):A\otimes D\oplus_{A\otimes C}B\otimes C\rightarrow B\otimes D$
of cofibrations $i:A\rightarrowtail B$ and $j:C\rightarrowtail D$ is a cofibration
which is also acyclic as soon as $i$ or $j$ is so.

(2) Suppose that $\mathcal{C}$ is a symmetric monoidal model category.
A symmetric monoidal category $\mathcal{E}$ over $\mathcal{C}$ is
a symmetric monoidal model category over $\mathcal{C}$ if the axiom MM1 holds for both the internal and external tensor
products of $\mathcal{E}$.
\end{defn}
\begin{example}
The usual projective model category $Ch_{\mathbb{K}}$ of unbounded chain complexes over
a field $\mathbb{K}$ forms a symmetric monoidal model category.
\end{example}
A useful property of the pushout-product axiom MM1  is that it is equivalent to the following standard dual version:
\begin{lem}(cf. \cite[Lemma 4.2.2]{Hov})
In a symmetric monoidal model category $\mathcal{C}$, the axiom MM1 is equivalent to the
following one:

\textbf{MM1'.} The morphism
\[
(i^{*},p_{*}):Hom_{\mathcal{C}}(B,X)\rightarrow Hom_{\mathcal{C}}(A,X)\times_{Hom_{\mathcal{C}}(A,Y)}Hom_{\mathcal{C}}(B,Y)
\]
induced by a cofibration $i:A\rightarrowtail B$ and a fibration $p:X\twoheadrightarrow Y$ is a fibration in $\mathcal{C}$ which
is also acyclic as soon as $i$ or $p$ is so.
\end{lem}

\subsection{Props, properads and their algebras}

Props generalize operads, so that algebras over props can be defined by operations with multiple outputs, contrary to operads which parametrize only operations with one single output. In particular, they are  adapted to the study of bialgebra-like structures. Properads are an intermediate object between operads and props, which are close enough to operads in the sense that they are defined, like operads, as monoids in a category of symmetric sequences (contrary to props), but are sufficient to encode many interesting bialgebra-like structures. One of the key feature of \emph{properads} is that, contrary to props, there is a good theory of bar-cobar constructions and Koszul duality for them,  allowing  to get explicit resolutions in deformation theory of algebraic structures. We detail some of these ideas below.

\subsubsection{Props and their algebras}

Let $\mathcal{C}$ be a symmetric monoidal category.
A $\Sigma$-biobject is a double sequence $\{M(m,n)\in\mathcal{C}\}_{(m,n)\in\mathbb{N}^2}$
where each $M(m,n)$ is equipped with a right action of $\Sigma_{m}$
and a left action of $\Sigma_{n}$ commuting with each other. We write $\mathcal{C}^{\mathbb{S}}$ for the category of $\Sigma$-biobjects in $\mathcal{C}$.
\begin{defn}
A prop is a $\Sigma$-biobject endowed with associative horizontal composition products
\[
\circ_{h}:P(m_1,n_1)\otimes P(m_2,n_2)\rightarrow P(m_1+m_2,n_1+n_2),
\]
associative vertical composition products
\[
\circ_{v}:P(k,n)\otimes P(m,k)\rightarrow P(m,n)
\]
and units $1\rightarrow P(n,n)$ which are neutral for $\circ_v$.
These products satisfy the exchange law
\[
(f_1\circ_h f_2)\circ_v(g_1\circ_h g_2) = (f_1\circ_v g_1)\circ_h(f_2\circ_v g_2)
\]
and are compatible with the actions of symmetric groups. The elements of $P(m,n)$ are said to be of arity $(m,n)$.

Morphisms of props are equivariant morphisms of collections compatible with the composition products.
\end{defn}
There is a functorial free prop construction $\mathcal{F}$ leading to an adjunction
\[
\mathcal{F}:\mathcal{C}^{\mathbb{S}}\rightleftarrows Prop:U
\]
where  $U$ is  the forgetful functor. As for operads, there is a notion of ideal in a prop, so that one can define a prop by generators and relations. This approach is particularly useful considering the definition of algebras over a prop:
\begin{defn}
(1) To any object $X$ of $\mathcal{C}$ we can associate an endomorphism prop $End_X$ defined by
\[
End_X(m,n)=Hom_{\mathcal{C}}(X^{\otimes m},X^{\otimes n}).
\]

(2) A $P$-algebra is an object $X\in\mathcal{C}$ equipped with a prop morphism $P\rightarrow End_X$.
\end{defn}
Operations of $P$ are sent to operations on tensor powers of $X$, and the compatibility of a prop morphism with composition products on both sides impose the relations that such operations satisfy. This means that given a presentation of a prop $P$ by generators and relations, the $P$-algebra structure on $X$ is determined by the images of these generators and their relations. Let us give some motivating examples related to our article:
\begin{example}\label{E:asscoass}
A differential graded associative and coassociative bialgebra is a triple $(B,\mu,\Delta)$ such that:
\begin{itemize}
\item[(i)] $(B,\mu)$ is a dg associative algebra;

\item[(ii)] $(B,\Delta)$ is a dg coassociative coalgebra;

\item[(iii)] the map $\Delta:B\rightarrow B\otimes B$ is a morphism of algebras and the map $\mu:B\otimes B\rightarrow B$ is a morphism of coalgebras. 
\end{itemize}
The prop $Bialg$ of associative-coassocative bialgebras is generated by two degree zero operations
, one generator of arity $(2,1)$ and one generator of arity $(1,2)$, which corresponds to the operations $\mu$ and $\Delta$ above wether one specifies a prop morphism $Bialg\rightarrow End_B$.
It is quotiented by the ideal generated the associativity relation, the coassociativity relation, and the compatibility relation describing the condition (iii) above.

In the unitary and counitary case, one adds a generator for the unit, a generator for the counit and the necessary compatibility relations with the product and the coproduct.
\end{example}
\begin{example}
Lie bialgebras originate from mathematical physics, in the study of integrable systems whose gauge groups are not only Lie
groups but Poisson-Lie groups, see the seminal work of Drinfeld \cite{Dri}. 

A differential graded Lie bialgebra is a triple $(\mathfrak{g},[,],\delta)$ such that:
\begin{itemize}
\item[(i)] $(\mathfrak{g},[,])$ is a dg Lie algebra;

\item[(ii)] $(\mathfrak{g},\delta)$ is a dg Lie coalgebra;

\item[(iii)] the cocycle relation : the coLie cobracket of a Lie bialgebra $g$ is a cocycle in the Chevalley-Eilenberg
complex $C^*_{CE}(g,\Lambda^2g)$, where $\Lambda^2g$ is equipped with the structure of $g$-module induced by the adjoint action.
\end{itemize}
The prop $BiLie$ encoding Lie bialgebras is generated by one generator of arity $(2,1)$ and one generator of arity $(1,2)$, both of degree zero, and with the signature action of $\Sigma_2$ (that is,
they are antisymmetric). It is quotiented by the ideal generated by the Jacobi relation, the co-Jacobi relation, and the cocycle relation.
\end{example}

We can also define a \emph{$P$-algebra in a symmetric monoidal category
over $\mathcal{C}$}:
\begin{defn}
Let $\mathcal{E}$ be a symmetric monoidal category over $\mathcal{C}$.

(1) The endomorphism prop of $X\in\mathcal{E}$ is given by $End_X(m,n)=Hom_{\mathcal{E}}(X^{\otimes m},X^{\otimes n})$
where $Hom_{\mathcal{E}}(-,-)$ is the external hom bifunctor of $\mathcal{E}$.

(2) Let $P$ be a prop in $\mathcal{C}$. A $P$-algebra in $\mathcal{E}$
is an object $X\in\mathcal{E}$ equipped with a prop morphism $P\rightarrow End_X$.
\end{defn}
This definition will be useful, for instance, in the case where $P$ is a dg prop (a prop in $Ch_{\mathbb{K}}$) but algebras over $P$ lie in a symmetric monoidal category over $Ch_{\mathbb{K}}$.

To conclude, props enjoy nice homotopical properties. Indeed, the category of $\Sigma$-biobjects $\mathcal{C}^{\mathbb{S}}$
is a diagram category over $\mathcal{C}$, so it inherits the usual projective model structure of diagrams, which can be transferred along the free-forgetful adjunction:
\begin{thm}
(cf. \cite[Theorem 5.5]{Frep}) The category of dg props $Prop$ equipped with the
classes of componentwise weak equivalences and componentwise fibrations forms a cofibrantly generated model category.
\end{thm}

\subsubsection{Properads}

Composing operations of two $\Sigma$-biobjects $M$ and $N$ amounts to consider $2$-levelled directed graphs
(with no loops) with the first level indexed by operations of $M$ and the second level by operations of $N$.
Vertical composition by grafting and horizontal composition by concatenation allows one to define props as before.
The idea of properads is to mimick the construction of operads as monoids in $\Sigma$-objects,
by restricting the vertical composition product to connected graphs.
The unit for this connected composition product $\boxtimes_c$ is the $\Sigma$-biobject $I$ given by $I(1,1)=\mathbb{K}$ and $I(m,n)=0$ otherwise.
The category of $\Sigma$-biobjects then forms a symmetric monoidal category $(Ch_{\mathbb{K}}^{\mathbb{S}},\boxtimes_c,I)$.
\begin{defn}
A dg properad $(P,\mu,\eta)$ is a monoid in $(Ch_{\mathbb{K}}^{\mathbb{S}},\boxtimes_c,I)$,
where $\mu$ denotes the product and $\eta$ the unit.
It is augmented if there exists a morphism of properads $\epsilon:P\rightarrow I$.
In this case, there is a canonical isomorphism $P\cong I\oplus\overline{P}$
where $\overline{P}=ker(\epsilon)$ is called the augmentation ideal of $P$.

Morphisms of properads are morphisms of monoids in $(Ch_{\mathbb{K}}^{\mathbb{S}},\boxtimes_c,I)$.
\end{defn}
Properads have also their dual notion, namely coproperads:
\begin{defn}
A dg coproperad $(C,\Delta,\epsilon)$ is a comonoid in $(Ch_{\mathbb{K}}^{\mathbb{S}},\boxtimes_c,I)$.
\end{defn}
As in the prop case, there exists a free properad functor $\mathcal{F}$ forming an adjunction
\[
\mathcal{F}:Ch_{\mathbb{K}}^{\mathbb{S}}\rightleftarrows Properad :U
\]
with $U$ the forgetful functor \cite{Val}. Dually, there exists a conilpotent cofree coproperad functor denoted $\mathcal{F}_c(-)$ having the same underlying $\Sigma$-biobject \cite[Theorem 5.7.9]{LV}. This adjunction equips dg properads with a cofibrantly generated model category structure with componentwise fibrations and weak equivalences \cite{MV2}. The notion of algebra over a properad is similar to algebra over a prop since the endomorphism prop restricts to an endomorphism properad.
Moreover, properads also form a model category for the same reasons as props:
\begin{thm}(cf. \cite[Appendix A]{MV2}) The category of dg props $Prop$ equipped with the
classes of componentwise weak equivalences and componentwise fibrations forms a cofibrantly generated model category.
\end{thm}
Properads are general enough to encode a wide range of bialgebra structures such as associative and coassociative bialgebras, Lie bialgebras, Poisson bialgebras, Frobenius bialgebras for instance. 

\subsection{Algebras and coalgebras over operads}

Operads are used to parametrize various kind of algebraic structures consisting of operations with
one single output. Fundamental examples of operads include the operad $As$ encoding associative algebras,
the operad $Com$ of commutative algebras, the operad $Lie$ of Lie algebras and the operad
$Pois$ of Poisson algebras. Dg operads form a model category with bar-cobar resolutions and Koszul duality \cite{LV}.
An algebra $X$ over a dg operad  $P$ can be defined in any symmetric monoidal category $\mathcal{E}$ over $Ch_{\mathbb{K}}$, alternatively as an algebra over the corresponding monad $P(-):Ch_{\mathbb{K}}\rightarrow Ch_{\mathbb{K}}$, which forms the free $P$-algebra functor, or as an operad morphism $P\rightarrow End_X$ where $End_X(n)=Hom_{\mathcal{E}}(X^{\otimes n},X)$ and $Hom_{\mathcal{E}}$ is the external hom bifunctor.
\begin{rem}
There is a free functor from operads to props, so that algebras over an operad are exactly the algebras over the corresponding prop. Hence algebras over props include algebras over operads as particular cases.
\end{rem}

Dual to operads is the notion of cooperad, defined as a comonoid in the category of $\Sigma$-objects. A coalgebra over a cooperad is a coalgebra over the associated comonad.
We can go from operads to cooperads and vice-versa by dualization.
Indeed, if $C$ is a cooperad, then the $\Sigma$-module $P$ defined by $P(n)=C(n)^*=Hom_{\mathbb{K}}(C(n),\mathbb{K})$
form an operad. Conversely, suppose that $\mathbb{K}$ is of characteristic zero and $P$ is an operad such that 
each $P(n)$ is finite dimensional. Then the $P(n)^*$ form a cooperad in the sense of \cite{LV}.
We also give the definition of \emph{coalgebras over an operad}:
\begin{defn}
(1) Let $P$ be an operad. A $P$-coalgebra is a complex $C$ equiped with linear applications
$\rho_n:P(n)\otimes C \rightarrow C^{\otimes n}$ for every $n\geq0$. These maps are $\Sigma_n$-equivariant
and associative with respect to the operadic compositions.

(2) Each $p\in P(n)$ gives rise to a cooperation $p^*:C\rightarrow C^{\otimes n}$.
The coalgebra $C$ is usually said to be conilpotent if for each $c\in C$, there exists $N\in\mathbb{N}$
so that $p^*(c)=0$ when we have $p\in P(n)$ with $n>N$.
\end{defn}
If $\mathbb{K}$ is a field of characteristic zero and the $P(n)$ are finite dimensional, then
it is equivalent to define a $P$-coalgebra via a family of applications
$\overline{\rho}_n:C\rightarrow P(n)^*\otimes_{\Sigma_n} C^{\otimes n}$.

\subsection{Homotopy algebras}\label{SS:HomotopyAlg}

Given a prop, properad or operad $P$, a homotopy $P$-algebra, or $P$-algebra up to homotopy, is
an algebra for which the relations are relaxed up to a coherent system of higher homotopies. this is encoded by
\begin{defn}
A homotopy $P$-algebra is an algebra over a cofibrant resolution $P_{\infty}$ of $P$.
\end{defn}
To make this definition meaningful, one has to prove that the notion of homotopy $P$-algebra does not depend (up to homotopy) on a choice of resolution:
\begin{thm}(cf. \cite{Yal})\label{Thm:Yalbis}
A weak equivalence of cofibrant dg props $P_{\infty}\stackrel{\sim}{\rightarrow}Q_{\infty}$ induces an equivalence of the corresponding $\infty$-categories of algebras
\[
P_{\infty}-Alg[W_{qiso}^{-1}]\stackrel{\sim}{\rightarrow}Q_{\infty}-Alg[W_{qiso}^{-1}],
\]
where $P_{\infty}-Alg[W_{qiso}^{-1}]$ denotes the $\infty$-categorical localization of $P_{\infty}-Alg$ with respect to its subcategory of quasi-isomorphisms.
\end{thm}

\begin{rem} Properads have a well defined theory of bar-cobar constructions and Koszul duality \cite{Val}, which allows to produce explicit cofibrant resolutions of properads. The bar-cobar resolution is a functorial cofibrant resolution but of a rather big size, whereas the resolution obtained from the Koszul dual (when $P$ is Koszul) is not functorial but smaller and better suited for computations.

These resolutions are of the form $P_{\infty}=(\mathcal{F}(V),\partial)$ where $\partial$
is a differential obtained by summing the differential induced by the $\Sigma$-biobject $V$ with a certain derivation.
To sum up, for a (pr)operad, one can always choose a homotopy $P$-algebra to be an algebra over a \emph{quasi-free resolution of $P$}, in which the generators give the system of higher homotopies and the relations defining a strict $P$-algebra become coboundaries.
\end{rem}


\subsection{Homotopy theory of cdgas and their modules}\label{SS:CDGA}

Before getting to the heart of the subject, let us precise that, as usual in deformation theory and (derived) algebraic geometry, the cdgas that we consider here are unital. That is, we consider the category of unital commutative monoids in the symmetric monoidal model category $Ch_{\mathbb{K}}$ and denote it by $cdga_{\mathbb{K}}$. Such monoids enjoy many useful homotopical properties, as they form a \emph{homotopical algebra context} in the sense of \cite[Definition 1.0.1.11]{TV}. We will not list all the properties satisfied by cdgas, but here is a non-exhaustive one that will be useful in this article:
\begin{itemize}
\item[(1)] The category $cdga_{\mathbb{K}}$ forms a cofibrantly generated model category with fibrations and weak equivalences being the degreewise surjections and quasi-isomorphisms.

\item[(2)] Given a cdga $A$, its category of dg $A$-modules $Mod_A$ forms a cofibrantly generated symmetric monoidal model category. The model structure is, again, right induced by the forgetful functor, and the tensor product is given by $-\otimes_A-$. In particular, we have a Quillen adjunction
\[
(-)\otimes A:Ch_{\mathbb{K}}\leftrightarrows Mod_A:U
\]
with a strong monoidal left adjoint (hence lax monoidal right adjoint). The unit $\eta$ of this adjunction is defined, for any complex $X$, by 
\begin{eqnarray*}
\eta(X):X & \rightarrow & X\otimes A \\
x & \longmapsto & x\otimes 1_A
\end{eqnarray*}
where $1_A$ is the unit element of $A$ (the image of $1_{\mathbb{K}}$ by the unit map of $A$).

\item[(3)] Base changes are compatible with the homotopy theory of modules. Precisely, a morphism of cdgas $f:A\rightarrow B$ induces a Quillen adjunction
\[
f_!:Mod_A\leftrightarrows Mod_B:f^*
\]
where $f^*$ equip a $B$-module with the $A$-module structure induced by the morphism $f$ and $f_!=(-)\otimes_AB$.
Moreover, if $f$ is a quasi-isomorphism of cdgas then this adjunction becomes a Quillen equivalence.

\item[(4)] The category of augmented cdgas $cdga_{\mathbb{K}}^{aug}$ is the category under $\mathbb{K}$ associated to $cdga_{\mathbb{K}}$, so it forms also a cofibrantly generated model category. Moreover, this model category is \emph{pointed} with $\mathbb{K}$ as initial and terminal object, so that one can alternately call them \emph{pointed} cdgas. Let us note also that augmented unital cdgas are equivalent to non-unital cdgas $cdga_{\mathbb{K}}^{nu}$ via the Quillen equivalence
\[
(-)_+:cdga_{\mathbb{K}}^{nu}\leftrightarrows cdga_{\mathbb{K}}^{aug}:(-)_-
\]
where $A_+=A\oplus\mathbb{K}$ for $A\in cdga_{\mathbb{K}}^{nu}$ and $A_-$ is the kernel of the augmentation map of $A$ for $A\in cdga_{\mathbb{K}}^{aug}$.
\end{itemize}
\begin{example}
There is a simplicial cdga called the \emph{Sullivan cdga of polynomial forms on the standard simplices}. It is given by 
\begin{equation}
 \Omega_n:= Sym \left(\bigoplus_{i=0}^n\big(\mathbb{K} t_i \oplus \mathbb{K} dt_i\big)\right)_{\slash_{\left(\begin{array}{l} t_0+\cdots +t_n=1 \\ dt_0+\cdots+dt_n=0 
 \end{array} \right)}}
\end{equation}
which is  the algebra of piecewise linear forms on the standard simplex $\Delta^n$, the differential being  defined, on the generators,
by $d(t_i) =dt_i$. The simplicial structure is induced by the cosimplicial structure of 
$n\mapsto \Delta^n$, see~\cite{Sul} for details.

This cdga and its modules will be essential when considering formal moduli problems and simplicial resolutions in the core of the paper.
\end{example}
\medskip

For any cdga $A$, the category $Mod_A$ of left dg $A$-modules is a (cofibrantly generated) symmetric monoidal model category tensored over chain complexes.
Therefore 
 one can define the category $P_{\infty}-Alg(Mod_A)$ of $P_{\infty}$-algebras in $Mod_A$,
 for any cofibrant prop $P_\infty$ as in section~\ref{SS:HomotopyAlg} and Theorem~\ref{Thm:Yalbis} extends to this context. 
 
 \medskip 
 
 An important subcategory of augmented cdgas is the one of artinian algebras, which are the coaffine formal moduli problems.
 \begin{defn} \label{D:ArtininaCDGA} An augmented cdga $A$ is Artinian if 
 \begin{itemize}
  \item its cohomology groups $H^n(A)$ vanish for $n>0$ and for $n << 0$, and each of them is finite dimensional over $k$;
  \item the (commutative) ring $H^0(A)$ is artinian in the standard meaning of commutative algebra.
 \end{itemize}
We denote 
$ dgArt_{\mathbb{K}}^{aug}$ the full subcategory of $CDGA_{\mathbb{K}}^{aug}$ of Artinian cdgas.
 \end{defn}

\subsection{Relative categories versus $\infty$-categories}

There are many equivalent ways to model $\infty$-categories. Precisely, there are several Quillen equivalent models for $\infty$-categories we can choose to work with \cite{Ber2}, for instance quasi-categories \cite{Lur1}, complete Segal spaces \cite{Rez2}, simplicial categories \cite{Ber1}, or relative categories \cite{BK1,BK2}. 
In this paper, it will often be convenient to consider $\infty$-functors which are associated to \lq\lq{}naive\rq\rq{} functors, provided-of course-that they preserve weak equivalences. This is not necessarily posible to do that in a straightforward naive way depending on the model chosen for $\infty$-categories. Therefore, here, we choose to work in the homotopy theory of relative categories as developed recently by Barwick-Kan \cite{BK1,BK2}. This will allow us to define more easily  $\infty$-functors starting from classical constructions, instead of going through, for instance, the cartesian fibration/opfibration formalism of \cite{Lur1}. For the sake of clarity, we start by recalling the main features of this theory and refer to \cite{BK1,BK2} for more details. Then we state some technical lemmas that will help us to go from equivalences of relative categories to equivalences of $\infty$-categories.

\subsubsection{$\infty$-categories associated to relative categories or  model categories}\label{SS:inftyCat}
We now recall and compare various standard ways to construct $\infty$-categories.
\begin{defn}\label{D:RelCat}
A relative category is a pair of categories $(\mathcal{C},W_{\mathcal{C}})$ such that $W_{\mathcal{C}}$ is a subcategory of $\mathcal{C}$ containing all the objects of $\mathcal{C}$. We call $W_{\mathcal{C}}$ the category of weak equivalences of $C$. A relative functor between two relative categories $(\mathcal{C},W_{\mathcal{C}})$ and $(\mathcal{D},W_{\mathcal{D}})$ is a functor $F:\mathcal{C}\rightarrow\mathcal{D}$ such that $F(W_{\mathcal{C}})\subset W_{\mathcal{D}}$.
\end{defn}
We denote by $RelCat$ the category of relative categories and relative functors. By Theorem 6.1 of \cite{BK1}, there is an adjunction between
the category of bisimplicial sets and the category of relative categories
\[
K_{\xi}:sSets^{\Delta^{op}}\leftrightarrows RelCat:N_{\xi}
\]
(where $K_{\xi}$ is the left adjoint and $N_{\xi}$ the right adjoint)
which lifts any Bousfield localization of the Reedy model structure of bisimplicial sets into a model structure on $RelCat$.
In the particular case of the Bousfield localization defining the complete Segal spaces \cite{Rez2}, one obtains
a Quillen equivalent homotopy theory of the homotopy theories in $RelCat$ \cite{BK1}.
In particular, a morphism of relative categories is a weak equivalence if and only if its image under $N_{\xi}$ is a weak equivalence of complete Segal spaces. We refer the reader to Section 5.3 of \cite{BK1} for the definition of the functor $N_{\xi}$. Let us just mention that it is weakly equivalent to 
the classifying diagram functor $N$ defined in \cite{Rez2}, which is a key  tool to construct complete Segal spaces.

A simplicial category is a category enriched over simplicial sets. We denote by $SCat$ the category of simplicial categories.
There exists functorial cosimplicial resolutions and simplicial resolutions in any model category (\cite{DK3},\cite{Hir}),
so model categories provide examples of (weakly) simplicially enriched categories. One recovers the morphisms of the homotopy category from a cofibrant object to a fibrant object by taking the set of connected components of the corresponding
simplicial mapping space. Another more general example is the simplicial localization developed by Dwyer and Kan \cite{DK1}.
To any relative category Dwyer and Kan associates a simplicial category $L(\mathcal{C},W_{\mathcal{C}})$ called its simplicial localization. They developed also another simplicial localization, the hammock localization $L^H(\mathcal{C},W_{\mathcal{C}})$ \cite{DK2}.
By taking the sets of connected components of the mapping spaces, we get $\pi_0L(\mathcal{C},W_{\mathcal{C}})\cong \mathcal{C}[W_{\mathcal{C}}^{-1}]$ where $\mathcal{C}[W_{\mathcal{C}}^{-1}]$ is the localization of $\mathcal{C}$ with respect to $W_{\mathcal{C}}$ (i.e. the homotopy category of $(\mathcal{C},W_{\mathcal{C}})$). The simplicial and hammock localizations are equivalent in the following sense:
\begin{prop}(Dwyer-Kan \cite{DK2}, Proposition 2.2)
Let $(\mathcal{C},W_{\mathcal{C}})$ be a relative category.
There is a zigzag of Dwyer-Kan equivalences
\[
L^H(\mathcal{C},W_{\mathcal{C}})\leftarrow diag L^H(F_*\mathcal{C},F_*W_{\mathcal{C}})\rightarrow L(\mathcal{C},W_{\mathcal{C}})
\]
where $F_*\mathcal{C}$ is a simplicial category called the standard resolution of $\mathcal{C}$ (see \cite{DK1} Section 2.5).
\end{prop}
Let us precise the definition of Dwyer-Kan equivalences:
\begin{defn}
Let $\mathcal{C}$ and $\mathcal{D}$ be two simplicial categories. A functor $F:\mathcal{C}\rightarrow\mathcal{D}$ is a Dwyer-Kan equivalence
if it induces weak equivalences of simplicial sets $Map_{\mathcal{C}}(X,Y)\stackrel{\sim}{\rightarrow}Map_{\mathcal{D}}(FX,FY)$
for every $X,Y\in\mathcal{C}$, as well as inducing
an equivalence of categories $\pi_0\mathcal{C}\stackrel{\sim}{\rightarrow}\pi_0\mathcal{D}$.
\end{defn}
Let us compile some useful results: first, every Quillen equivalence of model categories gives rise to a Dwyer-Kan equivalence of their
simplicial localizations, as well as a Dwyer-Kan equivalence of their hammock localizations
(see \cite{DK3} Proposition 5.4 in the case of simplicial model categories and \cite{Hin2} in the general case).
By Theorem 1.1 of \cite{Ber1}, there exists a model category structure on the category of (small) simplicial categories with the Dwyer-Kan equivalences as weak equivalences. Every simplicial category is Dwyer-Kan equivalent to the simplicial localization of a certain relative category
(see for instance \cite{BK2}, Theorem 1.7) and the associated model structure  is also a homotopy theory of homotopy theories.
The Reedy weak equivalences between two complete Segal spaces are precisely the Dwyer-Kan equivalences between
their associated homotopy theories (Theorem 7.2 of \cite{Rez2}).

\smallskip

Therefore the $\infty$-category associated to a relative category is thus, equivalently, the $\infty$-category associated to its simplicial localization or the $\infty$-category associated to its corresponding complete Segal space. The same construction applies to turn relative functors into $\infty$-functors. Moreover, it can be made functorial. For instance, given a relative category $(\mathcal{C},W_{\mathcal{C}})$, the associated quasi-category is given by the composite $N_{coh}L^H(\mathcal{C},W_{\mathcal{C}})^f$, where $L^H(-)$ is the Dwyer-Kan localization functor, $(-)^f$ is a functorial fibrant resolution in the Bergner model structure \cite{Ber1}, and $N_{coh}$ is the coherent nerve.

\subsubsection{From relative categories to homotopy automorphisms}

We  collect two useful lemmas  to obtain equivalences between $\infty$-categories of algebras, and see under which conditions they induce equivalences between the formal moduli problems controlling deformations of such algebras:
\begin{lem}\label{L: stricteq}
Let $F:(\mathcal{C},W_{\mathcal{C}})\rightleftarrows (\mathcal{D},W_{\mathcal{D}}):G$ be an adjunction of relative categories (that is, the functors $F$ and $G$ preserves weak equivalences) such that the unit and counit of this adjunction are pointwise weak equivalences. Then $F$ induces an equivalence of $\infty$-categories with inverse $G$.
\end{lem}
\begin{proof}
Let us denote by $RelCat$ the category of relative categories. The objects are the relative categories
and the morphisms are the relative functors, that is, the functors restricting to functors between the
categories of weak equivalences.
By \cite[Theorem 6.1]{BK1}, there is an adjunction between
the category of bisimplicial sets and the category of relative categories
\[
K_{\xi}:sSets^{\Delta^{op}}\leftrightarrows RelCat:N_{\xi}
\]
(where $K_{\xi}$ is the left adjoint and $N_{\xi}$ the right adjoint)
which lifts any Bousfield localization of the Reedy model structure of bisimplicial sets into a model structure on $RelCat$.
In the particular case of the Bousfield localization defining the model category $CSS$ of complete Segal spaces \cite[Theorem 7.2]{Rez2}, one obtains a Quillen equivalent homotopy theory of $\infty$-categories in $RelCat$ \cite{BK1}.

\smallskip

As recalled in~\ref{SS:inftyCat}, a way to build the $\infty$-category associated to a relative category $(\mathcal{C},W_{\mathcal{C}})$ is to take a functorial fibrant resolution $N_{\xi}(\mathcal{C},W_{\mathcal{C}})^f$ of the bisimplicial set $N_{\xi}(\mathcal{C},W_{\mathcal{C}})$ in $CSS$ to get a complete Segal space. So we want to prove that $N_{\xi}F^f$ is a weak equivalence of $CSS$. For this, let us note first that the assumption on the adjunction between $F$ and $G$ implies that $F$ is a strict homotopy equivalence in $RelCat$ in the sense of \cite{BK1}. By \cite[Proposition 7.5 (iii)]{BK1}, the functor $N_{\xi}$ preserves homotopy equivalences, so $N_{\xi}F$ is a homotopy equivalence of bisimplicial sets, hence a Reedy weak equivalence. Since $CSS$ is a Bousfield localization of the Reedy model structure on bisimplicial sets, Reedy weak equivalences are weak equivalences in $CSS$, then by applying the fibrant resolution functor $(-)^f$ we conclude that $N_{\xi}F^f$ is a weak equivalence of complete Segal spaces.
\end{proof}
In the formalism of Dwyer-Kan's hammock localization, an equivalence of simplicial categories $F:\mathcal{C}\rightarrow \mathcal{D}$ satisfies in particular the following property:
for every two objects $X$ and $Y$ of $\mathcal{C}$, it induces a weak equivalence of simplicial mapping spaces
\[
L^H(\mathcal{C},W_{\mathcal{C}})(X,Y)\stackrel{\sim}{\rightarrow}L^H(\mathcal{D},W_{\mathcal{D}})(F(X),F(Y)).
\]
(in particular, the associated functor $Ho(F)$ at the level of homotopy categories is an equivalence).
We would like this weak equivalence to restrict at the level of homotopy automorphisms:
\begin{lem}\label{L: equivhaut}
Let $F:(\mathcal{C},W_{\mathcal{C}})\rightleftarrows (\mathcal{D},W_{\mathcal{D}}):G$ be an adjunction of relative categories satisfying the assumptions of Lemma \ref{L: stricteq}. Then the restriction of $F$ to the subcategories of weak equivalences
\[
wF:W_{\mathcal{C}}\rightarrow W_{\mathcal{D}}
\]
is an equivalence of simplicial localizations (actually an equivalence of $\infty$-groupoids) inducing a weak equivalence of homotopy automorphisms
\[
L^HW_{\mathcal{C}}(X,X)\stackrel{\sim}{\rightarrow}L^HW_{\mathcal{D}}(F(X),F(X)),
\]
where $L^HW_{\mathcal{C}}$ is Dwyer-Kan's hammock localization of $W_{\mathcal{C}}$ with respect to itself.
\end{lem}
\begin{proof}
This adjunction of relative categories induces, by Lemma \ref{L: stricteq}, an equivalence of simplicial localizations between $L^H(\mathcal{C},W_{\mathcal{C}})$ and $L^H(\mathcal{D},W_{\mathcal{D}})$. By construction, this implies that the simplicial categories $L^HW_{\mathcal{C}}$ and $L^HW_{\mathcal{D}}$ are equivalent as well. Alternately, one could say that an equivalence of $\infty$-categories induces an equivalence of the associated $\infty$-groupoids of weak equivalences. By definition of an equivalence of simplicial categories, we get the desired equivalence between the simplicial mapping spaces of $L^HW_{\mathcal{C}}$ and their images under $F$ in $L^HW_{\mathcal{D}}$ (that is, an equivalence of homotopy automorphisms).
\end{proof}

\end{document}